\documentclass[11pt,a4paper]{amsart}
\usepackage{amsmath,amsthm,amsfonts,amssymb,graphicx,psfrag,dsfont,color,wrapfig,subfig,a4wide}

\usepackage{hyperref}
\usepackage[alphabetic,initials]{amsrefs}

\psfrag{S1}{$S^1$}
\psfrag{f}{$f$}
\psfrag{gee}{$g$}
\psfrag{r=0}{$0$}
\psfrag{r0}{$\rho_0$}
\psfrag{1+e}{$1+\epsilon$}
\psfrag{-e}{$-\epsilon$}
\psfrag{r1}{$\rho_1$}
\psfrag{r2}{$\rho_2$}
\psfrag{rho=0}{$\rho=0$}
\psfrag{slope}{$|\text{slope}| = \frac{1}{\tau} + \epsilon_0$}
\psfrag{c/3}{$\tau/3$}
\psfrag{2/3}{$\frac{2}{3}$}
\psfrag{one}{$1$}
\psfrag{r-d'}{$r - \delta'$}
\psfrag{2r/3}{$2r / 3$}
\psfrag{e}{$\epsilon_0$}
\psfrag{c}{$\tau$}
\psfrag{r/3}{$r / 3$}
\psfrag{c(1-e)}{$\tau (1 - \epsilon_0)$}
\psfrag{homot}{$\simeq$}
\psfrag{R}{$\RR$}
\psfrag{T}{$T$}
\psfrag{zero}{$0$}
\psfrag{T0}{$T_0$}
\psfrag{M}{$M$}
\psfrag{B}{$B$}
\psfrag{M0}{$M_0$}
\psfrag{ust}{$u_{\sigma,\tau}$}
\psfrag{T1}{$T_1$}
\psfrag{T2}{$T_2$}
\psfrag{g1e}{$\gamma_1^e$}
\psfrag{g1h}{$\gamma_1^h$}
\psfrag{g2e}{$\gamma_2^e$}
\psfrag{g2h}{$\gamma_2^h$}
\psfrag{v1+}{$v_1^+$}
\psfrag{v1-}{$v_1^-$}
\psfrag{v2+}{$v_2^+$}
\psfrag{v2-}{$v_2^-$}
\psfrag{u0}{$u_0$}
\psfrag{B}{$B$}
\psfrag{I}{$\iI$}
\psfrag{S1timesD}{$S^1 \times \DD$}
\psfrag{T3/Z2minusN(K)}{$(T^3 / \ZZ_2) \setminus \nN(K)$}
\psfrag{g1}{$\gamma_1$}
\psfrag{g2}{$\gamma_2$}
\psfrag{g3}{$\gamma_3$}
\psfrag{g12}{$2\gamma_1$}
\psfrag{g22}{$2\gamma_2$}
\psfrag{g}{$\gamma$}
\psfrag{2g}{$2\gamma$}
\psfrag{vk}{$v_k$}
\psfrag{v-1}{$v_-^1$}
\psfrag{v-2}{$v_-^2$}

\theoremstyle{plain}
\newtheorem*{thmu}{Theorem}
\newtheorem*{IFT}{Implicit Function Theorem}

\newtheorem{prop}{Proposition}[section]

\newtheorem{thm}{Theorem}
\newtheorem{thmp}{Theorem}
\newtheorem{cor}{Corollary}
\newtheorem{lemma}[prop]{Lemma}

\newtheorem*{conju}{Conjecture}
\newtheorem*{question}{Question}

\theoremstyle{definition}
\newtheorem{example}[prop]{Example}

\newtheorem{defn}[prop]{Definition}
\newtheorem*{notation}{Notation}

\theoremstyle{remark}
\newtheorem{remark}[prop]{Remark}

\newcommand{\AT}{\operatorname{AT}}

\newcommand{\End}{\operatorname{End}}
\newcommand{\ECH}{\operatorname{ECH}}

\newcommand{\Hom}{\operatorname{Hom}}

\newcommand{\ind}{\operatorname{ind}}

\newcommand{\interior}{\operatorname{int}}

\newcommand{\Lie}{\mathcal{L}}

\newcommand{\muCZ}{\mu_{\operatorname{CZ}}}

\newcommand{\PT}{\operatorname{PT}}

\newcommand{\SFT}{{\operatorname{SFT}}}

\newcommand{\wind}{\operatorname{wind}}

\newcommand{\Spp}{\operatorname{Sp}}

\newcommand{\CC}{{\mathbb C}}
\newcommand{\DD}{{\mathbb D}}

\newcommand{\NN}{{\mathbb N}}
\newcommand{\QQ}{{\mathbb Q}}
\newcommand{\RR}{{\mathbb R}}

\newcommand{\ZZ}{{\mathbb Z}}
\newcommand{\aA}{{\mathcal A}}

\newcommand{\fF}{{\mathcal F}}
\newcommand{\hH}{{\mathcal H}}
\newcommand{\iI}{{\mathcal I}}
\newcommand{\jJ}{{\mathcal J}}

\newcommand{\mM}{{\mathcal M}}
\newcommand{\nN}{{\mathcal N}}

\newcommand{\sS}{{\mathcal S}}
\newcommand{\tT}{{\mathcal T}}
\newcommand{\uU}{{\mathcal U}}

\newcommand{\wW}{{\mathcal W}}
\newcommand{\1}{\mathds{1}}
\newcommand{\p}{\partial}

\newcommand{\inter}{\bullet}

\newcommand{\defin}[1]{\textbf{#1}}

\numberwithin{equation}{section}

\title[A Hierarchy of Filling Obstructions for Contact Manifolds]{A 
Hierarchy of Local Symplectic Filling Obstructions for Contact $3$-Manifolds}
\author{Chris Wendl}
\address{Department of Mathematics \\ 
University College London \\
Gower Street \\
London WC1E 6BT \\ 
United Kingdom}
\email{c.wendl@ucl.ac.uk}
\urladdr{http://www.homepages.ucl.ac.uk/~ucahcwe/}
\thanks{Research supported by an Alexander von Humboldt
Foundation Fellowship.}

\begin{document}

\begin{abstract}
We generalize the familiar notions of overtwistedness and Giroux torsion
in $3$-dimensional contact manifolds, defining an infinite hierarchy of
local filling obstructions called \emph{planar torsion}, whose 
integer-valued \emph{order} $k \ge 0$ 
can be interpreted as measuring a gradation in ``degrees of tightness'' of  
contact manifolds.  We show in particular that any contact manifold
with planar torsion admits no contact type embeddings into any closed 
symplectic $4$-manifold, and has vanishing contact invariant in Embedded 
Contact Homology, and we give examples of contact manifolds that have
planar $k$-torsion for any $k \ge 2$ but no Giroux torsion.  We also show
that the complement of the binding of a supporting open book 
never has planar torsion.  The unifying idea in the background
is a decomposition of contact manifolds in terms of contact fiber sums of 
open books along their binding.  As the technical basis of these 
results, we establish existence, uniqueness and compactness theorems 
for certain classes of $J$-holomorphic curves in blown up summed
open books; these also imply algebraic obstructions to
planarity and embeddings of partially planar domains.  
\end{abstract}

\maketitle

\tableofcontents

\newpage

\section{Introduction}
\label{sec:intro}

Contact structures for odd-dimensional manifolds 
arise naturally on boundaries of symplectic manifolds
via the notion of \emph{convexity}.  A symplectic manifold $(W,\omega)$
is said to have \defin{convex boundary} if, on a neighborhood of~$\p W$,
there exists a vector field~$Y$ that points transversely outward at~$\p W$
and whose flow is a symplectic dilation, i.e.~$\Lie_Y\omega = \omega$.  
Writing $M = \p W$, the co-oriented hyperplane field 
$\xi = \ker\left( \iota_Y\omega|_{TM}\right) \subset TM$ 
then satisfies a certain
``maximal nonintegrability'' condition which makes it a \emph{contact structure},
and up to isotopy, it depends only on the symplectic structure
of $(W,\omega)$ near~$M$, not on the choice of vector field~$Y$.

Given the above relationship, it is interesting to ask which isomorphism
classes of contact manifolds $(M,\xi)$ do \emph{not} arise as boundaries 
of compact symplectic manifolds, i.e.~which ones are not \emph{symplectically
fillable}.  A variety of obstructions to symplectic filling are known,
and the following two examples give some hint as to the diversity of
such results:
\begin{itemize}
\item Lisca \cites{Lisca:curvature,Lisca:fillings3manifolds} used the
Seiberg-Witten monopole invariants of Kronheimer and Mrowka
\cite{KronheimerMrowka:contact} together with Donaldson's theorem
on the intersection forms of smooth $4$-manifolds 
\cite{Donaldson:intersectionForms} to find examples of oriented $3$-manifolds
that admit no symplectically fillable contact structures.
\item The author \cite{Wendl:fillable} used punctured holomorphic curve 
techniques to show that a contact $3$-manifold has no symplectic
filling if it is supported by a planar open book whose
monodromy is not a product of right-handed Dehn twists.
(See \cites{PlamenevskayaVanHorn,Plamenevskaya:surgeries} for some 
applications of this result.)
\end{itemize}
One common feature of the above examples is that they depend fundamentally
on the \emph{global} properties of the manifolds involved.  In contrast,
one can also consider filling obstructions which are \emph{local}, in
the sense that they answer the following question:

\vspace{6pt}
\begin{center}
\begin{minipage}{5in}
\textit{What kinds of contact subdomains can never exist in the
convex boundary of a compact symplectic manifold?}
\end{minipage}
\end{center}
\vspace{6pt}

The first known example of a symplectic filling obstruction was essentially
local in this sense: Gromov \cite{Gromov} and 
Eliashberg \cite{Eliashberg:diskFilling}
established that contact type boundaries of symplectic $4$-manifolds can
never contain an \emph{overtwisted disk}, and significantly, 
the related distinction between so-called ``overtwisted'' and
``tight'' contact structures, discovered by Eliashberg \cite{Eliashberg:overtwisted}, 
has played a pivotal role in classification
questions for contact structures in dimension three.  
This non-fillability result can be rephrased in terms of a certain 
$3$-dimensional contact domain with boundary that we call a
\emph{Lutz tube}: this is a solid torus $S^1 \times \DD$ with a radially
symmetric contact structure that makes a half-twist along radii from the
center to the boundary (see Figure~\ref{fig:LutzTube} and
Definition~\ref{defn:LutzTube}).  One can show (e.g.~using \cite{Eliashberg:overtwisted})
that a closed contact $3$-manifold contains an overtwisted disk if and only if
it contains a Lutz tube, thus the latter may be regarded as the
prototypical example of a local filling obstruction.

A more general local filling obstruction is furnished by the so-called
\emph{Giroux torsion domain}, a thickened torus $[0,1] \times T^2$ with
a $T^2$-invariant contact structure that makes one full twist from one
end of the interval to the other
(see Figure~\ref{fig:GirouxTorsionDomain} and
Definition~\ref{defn:GirouxTorsion}).  Contact manifolds containing
such an object are said to have \emph{Giroux torsion}, and the fact that
they are not fillable in general is a comparatively recent result, due to
Gay \cite{Gay:GirouxTorsion}.  Giroux
torsion domains have also played an important role in the classification
of contact structures, most notably through the work of Colin, Giroux and Honda
\cites{ColinGirouxHonda:coarse,ColinGirouxHonda:finitude}.

\begin{figure}
\begin{center}
\includegraphics{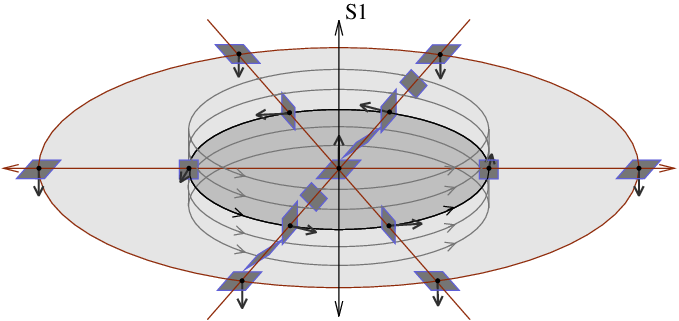}
\caption{\label{fig:LutzTube}
Contact planes twist around the radii emerging from the central axis
of a Lutz tube.  The picture also shows an embedded $J$-holomorphic plane 
asymptotic to a Reeb orbit of small
period in a Morse-Bott family (arrows indicate the Reeb vector field);
every Lutz tube contains such planes, which are the reason why the
contact homology of an overtwisted contact manifold vanishes.}
\end{center}
\end{figure}
\begin{figure}
\begin{center}
\includegraphics{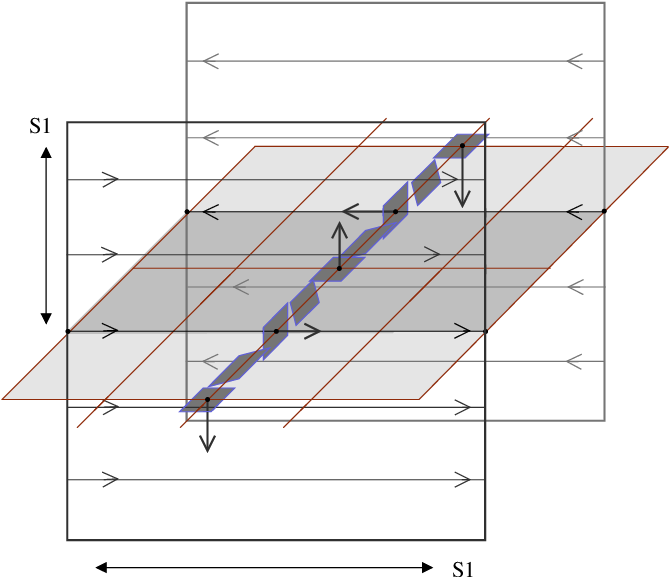}
\caption{\label{fig:GirouxTorsionDomain}
In a Giroux torsion domain $[0,1] \times T^2$, contact planes twist 
around segments in the $[0,1]$-direction.  Such domains are
foliated by $J$-holomorphic cylinders asymptotic to Morse-Bott Reeb orbits.}
\end{center}
\end{figure}

These two examples of local filling obstructions 
create the intuitive impression that contact
manifolds tend to become non-fillable whenever they contain regions where
the contact planes exhibit some threshold amount of \emph{twisting}.
In this paper we shall introduce a geometric formalism that makes 
this notion precise, and in so doing, greatly expands the known repertoire of 
local filling obstructions.  We will demonstrate in particular 
that the examples above occupy the
first two levels in an infinite \emph{hierarchy}: for each integer 
$k \ge 0$, we shall define a special class of
compact contact $3$-manifolds, possibly with boundary, which we call
\emph{planar $k$-torsion domains}, such that the Lutz tube and Giroux
torsion domain are special cases with $k=0$ and~$1$ respectively.
Our use of the word ``hierarchy'' is not incidental, as it turns out that
a planar torsion domain yields quantifiably stricter or less strict filling
obstructions depending on its \emph{order}, i.e.~the integer~$k$.
In particular, the overtwisted contact manifolds are precisely those which
have planar $0$-torsion, and these can be thought of as the
``most non-fillable'' among all contact $3$-manifolds, while the
fillable contact manifolds are the ``tightest,'' and those which
have only higher orders of planar torsion are non-fillable but are
in some sense ``tighter'' than their lower order counterparts.

The definition
of planar torsion, which will be given in a precise form in
\S\ref{sec:definitions}, 
combines the fundamental contact topological notion of a
\emph{supporting open book decomposition}, as introduced by
Giroux \cite{Giroux:openbook}, with a simple topological operation
known as the \emph{contact fiber sum} along codimension~2
contact submanifolds,
originally due to Gromov \cite{Gromov:PDRs} and Geiges 
\cite{Geiges:constructions}.  Roughly speaking, a planar $k$-torsion domain
is a compact contact $3$-manifold $(M,\xi)$, possibly with boundary, 
that contains a non-empty set of disjoint pre-Lagrangian tori  
dividing it into two pieces:
\begin{itemize}
\item A \emph{planar piece} $M^P$, which is disjoint from $\p M$ and
looks like a connected open book with some binding components blown up
and/or attached to each other by contact fiber sums.
The pages must have genus zero and $k+1$ boundary components.
\item The \emph{padding} $M \setminus M^P$, which contains $\p M$ and
consists of one or more arbitrary open books, again with some binding
components blown up or fiber summed together.
\end{itemize}
Planar torsion domains are thus examples of what are called
\emph{partially planar domains}, a notion that was first hinted at in
\cite{AlbersBramhamWendl}.  
The interior of such a domain~$M$ always contains a special 
set $\iI \subset M$ of pre-Lagrangian tori which arise by blowing up 
binding components of open books: we refer to these tori all together as
the \emph{interface} of~$(M,\xi)$.  Postponing the exact definitions
until \S\ref{sec:definitions}, let us for now merely point out
that in a Lutz tube $M = S^1 \times \DD$ (Figure~\ref{fig:LutzTube}), 
the planar piece is some smaller solid torus $M^P = S^1 \times \DD_r$
for $0 < r < 1$, and the pages of the blown up open book in~$M^P$ are
the disks $\{*\} \times \DD_r$.  Likewise,
the planar piece in a Giroux torsion domain $M = [0,1] \times T^2$
(Figure~\ref{fig:GirouxTorsionDomain}) is a smaller thickened torus
$M^P = [r_1,r_2] \times T^2$ for $0 < r_1 < r_2 < 1$, foliated by
cylindrical pages of the form $[r_1,r_2] \times S^1 \times \{*\}$,
and for both examples $\iI = \p M^P$.
We will see that in the more general definition, the topology of the
planar piece and the whole domain may differ from each other considerably,
and interface tori may also be found in the interior of the planar piece
or the padding.  Some simple examples of the form $S^1 \times \Sigma$ are 
shown in Figure~\ref{fig:torsionDomains}.  We should also mention that the
idea of decomposing contact manifolds in this way via fiber sums of
open books has further applications beyond filling obstructions, e.g.~it is
used in \cite{Wendl:cobordisms} to define a ``blown up'' version of
Eliashberg's capping construction \cite{Eliashberg:cap}, producing a
wide range of existence results for non-exact symplectic cobordisms.

Let us now recall some basic definitions in preparation for stating the
main results.  A \defin{contact structure} on an oriented $3$-dimensional
manifold is a hyperplane distribution~$\xi$ that can be written locally as the
kernel of a smooth $1$-form~$\alpha$ with $\alpha \wedge d\alpha \ne 0$.
We call~$\xi$ \defin{positive} if $\alpha \wedge d\alpha > 0$.  Every
contact structure in this paper will be assumed to be positive and to
carry a \emph{co-orientation}, which can be defined via a global choice
of $1$-form~$\alpha$; any~$\alpha$ with $\ker\alpha = \xi$ that is
compatible with the chosen co-orientation is called a \defin{contact form}
for $(M,\xi)$.  Note that a co-oriented contact structure also inherits
a natural orientation.  Given two contact $3$-manifolds $(M_0,\xi_0)$ and
$(M,\xi)$, a \defin{contact embedding} of $(M_0,\xi_0)$ into $(M,\xi)$ is
an orientation preserving embedding $\iota : M_0 \hookrightarrow M$ such
that $\iota_* : TM_0 \hookrightarrow TM$ defines an orientation 
preserving map of~$\xi_0$ to~$\xi$.

Suppose $(W,\omega)$ is a compact $4$-dimensional 
symplectic manifold (oriented by $\omega \wedge \omega$) and $(M,\xi)$ is a 
closed contact $3$-manifold.  
A \defin{weak contact type embedding} of $(M,\xi)$
into $(W,\omega)$ is an embedding $\iota : M \hookrightarrow W$ for which 
$\iota^*\omega|_\xi > 0$.  It is called a (strong) \defin{contact type embedding}
if a neighborhood of $\iota(M) \subset W$ admits a $1$-form~$\lambda$ such that
$d\lambda = \omega$ and $\iota^*\lambda$ defines a contact form for $(M,\xi)$;
note that in this case, the vector field $\omega$-dual to~$\lambda$ defines 
a symplectic dilation positively transverse to~$\iota(M)$.  The image of
a (weak or strong) contact type embedding is called a (weak or strong)
\defin{contact type hypersurface} in $(W,\omega)$.
If the image is~$\p W$ and
$\iota$ maps the orientation of~$M$ to the natural boundary orientation,
then we say $(W,\omega)$ is a (weak or strong) 
\defin{symplectic filling} of~$(M,\xi)$.

\subsection{Obstructions to symplectic fillings}

Given the notion of a planar $k$-torsion domain which was sketched above and
will be explained
fully in \S\ref{sec:definitions}, it is natural to define the following.

\begin{defn}
\label{defn:planarTorsion}
A contact $3$-manifold is said to have \defin{planar torsion of order~$k$}
(or \defin{planar $k$-torsion}) if it admits a contact embedding of a
planar $k$-torsion domain (see Definition~\ref{defn:planarTorsionDomain}).
\end{defn}

\begin{thm}
\label{thm:non-fillable}
If $(M,\xi)$ is a closed contact $3$-manifold with planar torsion of any
order, then it does not admit a contact type embedding into
any closed symplectic $4$-manifold.  In particular, it is not strongly
fillable.
\end{thm}

\begin{figure}
\begin{center}
\includegraphics{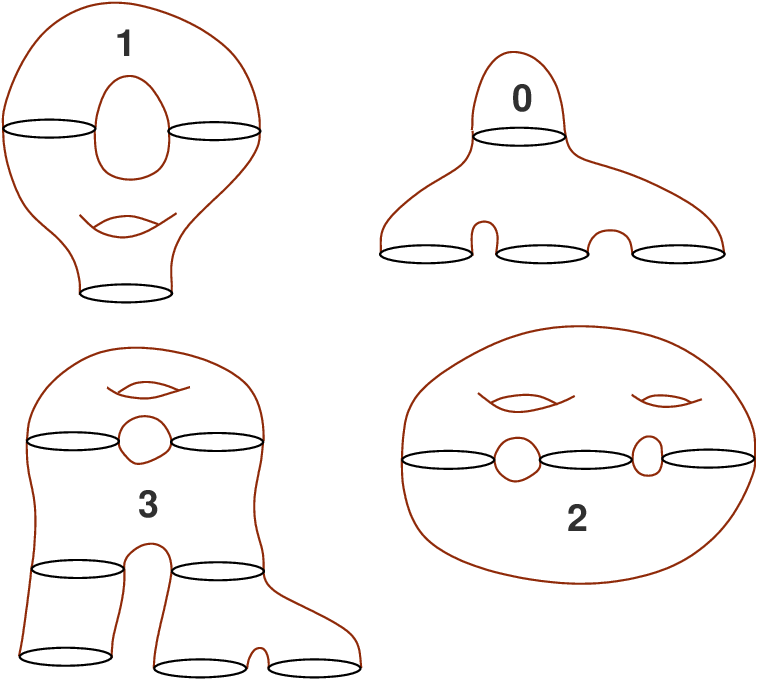}
\caption{\label{fig:torsionDomains}
Various planar $k$-torsion domains, with the order~$k \ge 0$ indicated 
within the planar piece.  Each picture shows a surface $\Sigma$ that
defines a manifold $S^1 \times \Sigma$ with an
$S^1$-invariant contact structure~$\xi$.  The multicurves that divide
$\Sigma$ are the sets of all points $z \in \Sigma$ at which
$S^1 \times \{z\}$ is Legendrian.  
See also Example~\ref{ex:S1invariant} and Figure~\ref{fig:moreTorsion}.}
\end{center}
\end{figure}

Though our proof of non-fillability will not depend on it,
the implication that $(M,\xi)$ is not strongly fillable follows from
the above statement due to
a result of Etnyre and Honda \cite{EtnyreHonda:cobordisms}, that
every contact $3$-manifold is \emph{concave} fillable: this means that 
strong fillings can always be capped off to produce closed symplectic
$4$-manifolds containing contact type hypersurfaces.  

We will also prove an algebraic counterpart to the above result in terms of
Embedded Contact Homology, or ``ECH'' for short
(see e.g.~\cite{Hutchings:ICM}).  
The definition of ECH will be reviewed in \S\ref{sec:ECH};
for now it suffices to recall that given a closed contact $3$-manifold
$(M,\xi)$ with nondegenerate contact form~$\lambda$ and generic compatible
complex structure $J : \xi \to \xi$, one can define a chain complex
generated by so-called \emph{orbit sets},
$$
\boldsymbol{\gamma} = ((\gamma_1,m_1),\ldots,(\gamma_n,m_n)),
$$
where $\gamma_1,\ldots,\gamma_n$ are distinct simply covered periodic
Reeb orbits and $m_1,\ldots,m_n$ are positive
integers, called \emph{multiplicities}.  A differential operator is
then defined by counting a certain class of embedded rigid 
$J$-holomorphic curves in the
symplectization of $(M,\xi)$, which can be viewed as cobordisms 
between orbit sets.  The homology of the resulting
chain complex is the Embedded Contact Homology 
$\ECH_*(M,\lambda,J)$.  Though the complex obviously 
depends on~$\lambda$ and~$J$, Taubes has shown
\cites{Taubes:ECH=SWF1,Taubes:ECH=SWF5} that $\ECH_*(M,\lambda,J)$ is
isomorphic to a version of Seiberg-Witten Floer
homology, and thus actually only depends (up to natural isomorphisms) 
on the contact manifold $(M,\xi)$, so we can write
$$
\ECH_*(M,\xi) := \ECH(M,\lambda,J).
$$
The case $n=0$ is also allowed
among the generators, i.e.~the ``empty'' orbit set
$\boldsymbol{\emptyset} := ()$,
which is always a cycle in the homology, thus defining
a distinguished class 
$$
c(\xi) := [\boldsymbol{\emptyset}] \in \ECH(M,\xi),
$$
which we call the \defin{ECH contact invariant}.  It corresponds under
Taubes's isomorphism to a similar contact invariant in Seiberg-Witten
theory, and conjecturally\footnote{Recent progress on this
conjecture has been made in parallel projects by 
Colin-Ghiggini-Honda \cite{ColinGhigginiHonda:openbook}
and Kutluhan-Lee-Taubes \cites{KutluhanLeeTaubes1,KutluhanLeeTaubes2}.}
also to the Ozsv\'ath-Szab\'o contact invariant
in Heegaard Floer homology.

\begin{thm}
\label{thm:ECH}
If $(M,\xi)$ is a closed contact $3$-manifold with planar torsion of any
order, then its ECH contact invariant $c(\xi)$ vanishes.
\end{thm}

This calculation is in some sense a generalization of the well-known
fact that overtwisted contact manifolds have trivial contact homology
(cf.~Figure~\ref{fig:LutzTube}), and our proof of it has some commonalities
with the proof of the latter sketched by Eliashberg in the appendix
of \cite{Yau:overtwisted}.
The result implies another proof that planar torsion is a filling 
obstruction, albeit a very indirect one: under the isomorphism of Taubes 
\cite{Taubes:ECH=SWF5}, the ECH contact invariant corresponds to a
similar invariant in Seiberg-Witten theory, whose vanishing gives a
filling obstruction due to results of Kronheimer and Mrowka
\cite{KronheimerMrowka:contact}.  We will however give a proof of 
Theorem~\ref{thm:non-fillable} that uses only holomorphic curve methods,
requiring no assistance from Seiberg-Witten theory.

\begin{remark}
\label{remark:cobordisms}
Aside from the direct holomorphic curve proof of Theorem~\ref{thm:non-fillable}
that we will give in \S\ref{sec:non-fillable}, there are at least two 
alternative approaches/generalizations one can imagine:
\begin{enumerate}
\renewcommand{\labelenumi}{(\alph{enumi})}
\item \textsl{Algebraic:} find a contact invariant whose vanishing contradicts
symplectic filling, and which must always vanish in the presence of planar
torsion.
\item \textsl{Topological:} given $(M,\xi)$ with planar torsion, 
find a symplectic cobordism with negative boundary $(M,\xi)$ whose
positive boundary is already known to be not fillable.
\end{enumerate}
The first approach is pursued in the present article and in the related paper
\cite{LatschevWendl}, however the second approach also works.  Indeed,
after the first version of this paper was completed, the author defined
in \cite{Wendl:cobordisms} a generalized handle attaching construction
which yields symplectic cobordisms from any contact manifold with planar
torsion to another that is overtwisted.  The decomposition of contact
manifolds via blown up summed open books that we will explain in
\S\ref{subsec:summed} is a crucial ingredient in this construction, which
also yields alternative proofs of Theorem~\ref{thm:nonseparating}
and the weak filling obstructions of \cite{NiederkruegerWendl} mentioned
below.
\end{remark}

Under stronger geometric assumptions one also obtains stronger results
in terms of ECH with \emph{twisted coefficients}, which gives correspondingly
stricter obstructions to symplectic fillings.  As we will review
in \S\ref{sec:ECH}, a twisted version of the ECH chain complex can be
defined as a module over the group ring $\ZZ[H_2(M;\RR)]$, so that the
differential keeps track of the $2$-dimensional relative
homology classes of the 
holomorphic curves it counts.  We shall denote this twisted version of ECH
by $\widetilde{\ECH}_*(M,\xi)$.  It also contains a preferred
homology class $\tilde{c}(\xi) \in \widetilde{\ECH}_*(M,\xi)$ 
represented by the empty orbit set, called the
\defin{twisted ECH contact invariant}.

\begin{defn}
\label{defn:separating}
A contact $3$-manifold is said to have 
\defin{fully separating planar $k$-torsion}
if it contains a planar $k$-torsion domain with a planar
piece $M^P \subset M$ that has the following properties:
\begin{enumerate}
\item There are no interface tori in the interior of~$M^P$.
\item Every connected component of $\p M^P$ separates~$M$.
\end{enumerate}
\end{defn}
We will see that the fully separating condition is always satisfied if $k=0$,
and for the case of a Giroux torsion domain, it is satisfied if and only
if the domain separates~$M$.

\setcounter{thmp}{\value{thm}}
\begin{thmp}
\label{thm:twisted}
If $(M,\xi)$ is a closed contact $3$-manifold with fully separating
planar torsion, then its twisted ECH contact invariant
$\tilde{c}(\xi)$ vanishes.
\end{thmp}

Appealing again to the isomorphism of \cite{Taubes:ECH=SWF5} together with
results from Seiberg-Witten theory \cite{KronheimerMrowka:contact}
on weak symplectic fillings, we
obtain the following consequence, which is also proved by a more direct 
holomorphic curve argument in joint work of the author with
Klaus Niederkr\"uger \cite{NiederkruegerWendl}.

\begin{cor}
\label{cor:weak}
If $(M,\xi)$ is a closed contact $3$-manifold with fully separating
planar torsion, then it is not weakly fillable.
\end{cor}

As we will show shortly, Theorem~\ref{thm:non-fillable}
and Corollary~\ref{cor:weak} yield many previously unknown examples
of non-fillable contact manifolds.
Observe that the fully separating condition in Corollary~\ref{cor:weak} cannot
be removed in general, as for instance, there are infinitely many tight
$3$-tori which have non-separating Giroux torsion (and hence
planar $1$-torsion by Theorem~\ref{thm:lowerOrder} below) but are
weakly fillable by a construction of Giroux \cite{Giroux:plusOuMoins}.
Further examples of this phenomenon are constructed in
\cite{NiederkruegerWendl} for planar $k$-torsion with any $k \ge 1$.

\begin{remark}
\label{remark:OmegaSeparating}
One can refine the above vanishing result with twisted coefficients
as follows: for a given closed $2$-form $\Omega$ on~$M$, define $(M,\xi)$ to 
have \emph{$\Omega$-separating planar torsion} if it contains a planar torsion
domain such that every interface torus~$T$ lying in
the planar piece satisfies
$\int_T \Omega = 0$ (cf.~Definition~\ref{defn:containsPartiallyPlanar}).  
Under this condition, our computation implies 
a similar vanishing
result for the ECH contact invariant with twisted coefficients in
$\ZZ[H_2(M;\RR) / \ker\Omega]$, with the consequence that $(M,\xi)$
admits no weak filling $(W,\omega)$ for which $\omega|_{TM}$
is cohomologous to~$\Omega$.  A direct proof of the latter is given in
\cite{NiederkruegerWendl}.
\end{remark}

We now consider examples of contact manifolds with planar torsion.  We will
show in \S\ref{subsec:torsion} that the previously known local filling
obstructions fit into the first two levels of the hierarchy, i.e.~$k=0$
and~$1$.

\begin{thm}
\label{thm:lowerOrder}
A closed contact $3$-manifold has planar $0$-torsion if and only
if it is overtwisted, and every closed contact manifold with Giroux
torsion also has planar $1$-torsion.
\end{thm}

For this reason, Theorems~\ref{thm:ECH} and~\ref{thm:twisted} imply 
ECH versions of the vanishing results of Ghiggini, Honda and Van Horn-Morris
\cites{GhigginiHondaVanhorn,GhigginiHonda:twisted} for the
Ozsv\'ath-Szab\'o contact invariant in the presence of Giroux torsion.
We'll see below that it is also easy to
construct examples of contact manifolds that have planar
torsion of any order greater than~$1$ but no Giroux torsion.  It is not clear
whether there exist contact manifolds with planar $1$-torsion but no
Giroux torsion.

To find examples for $k \ge 2$,
suppose $\Sigma$ is a closed oriented surface containing a non-empty multicurve
$\Gamma \subset \Sigma$ that divides it into two (possibly disconnected)
pieces $\Sigma_+$ and $\Sigma_-$.  We define the contact manifold
$(M_\Gamma,\xi_\Gamma)$, where
$$
M_\Gamma := S^1 \times \Sigma
$$
and $\xi_\Gamma$ is the (up to isotopy) 
unique $S^1$-invariant contact structure that
makes $\{\text{const}\} \times \Sigma$ into a convex surface with
dividing set~$\Gamma$.  The existence and uniqueness of such a contact
structure follows from a result of Lutz \cite{Lutz:77}.
We will see in Examples~\ref{ex:BUSOBD} and~\ref{ex:S1invariant}
that $(M_\Gamma,\xi_\Gamma)$
is a partially planar domain whenever any connected component~$\Sigma_0$ of
$\Sigma \setminus \Gamma$ has genus zero: indeed, the surfaces
$\{*\} \times \Sigma_0$ are then the pages of a blown up
planar open book.  Moreover, $(M_\Gamma,\xi_\Gamma)$ is then a planar torsion
domain unless $\Sigma \setminus \Gamma$ has exactly two connected components
and they are diffeomorphic, and it is fully separating
if every connected component of
$\p\overline{\Sigma}_0$ separates~$\Sigma$.

\begin{cor}
\label{cor:examples}
Suppose $\Sigma\setminus\Gamma$ has a connected component $\Sigma_0$ 
of genus zero, and either $\Sigma \setminus \Gamma$ has more than two
connected components or $\Sigma\setminus \overline{\Sigma}_0$
is not diffeomorphic to $\Sigma_0$.  Then $(M_\Gamma,\xi_\Gamma)$ has
vanishing (untwisted) ECH contact invariant and is not strongly fillable.
Moreover, if every connected component of
$\p\overline{\Sigma}_0$ separates~$\Sigma$, then
the invariant with twisted coefficients also vanishes and
$(M_\Gamma,\xi_\Gamma)$ is not weakly fillable.
\end{cor}

Note that $(M_\Gamma,\xi_\Gamma)$ is always universally tight whenever 
$\Gamma$ contains no contractible connected components.
This follows from \cite{Giroux:cercles}*{Prop.~4.1(b)}, and can also
be deduced (via \cite{Hofer:weinstein}) from the observation that
$(M_\Gamma,\xi_\Gamma)$ then admits contact forms with no contractible
Reeb orbits (e.g.~any Giroux form in the sense of
Definition~\ref{defn:GirouxForm} will have this property).
Whenever this is true, an argument due to Giroux 
(see \cite{Massot:vanishing}*{Theorem~3})
implies that $(M_\Gamma,\xi_\Gamma)$ also has no Giroux torsion
if no two connected components of~$\Gamma$ are isotopic.
We thus obtain infinitely many examples of contact manifolds 
that have planar torsion of any order greater than~$1$ but no Giroux torsion:

\begin{cor}
\label{cor:higherOrder}
For any integers $g \ge k \ge 1$, let $(V_g,\xi_k)$ denote the
$S^1$-invariant contact manifold $(M_\Gamma,\xi_\Gamma)$
described above for the case where $\Gamma \subset \Sigma$ has~$k$
connected components and divides~$\Sigma$ into two connected components,
one with genus zero and the other with genus $g - k + 1 > 0$.
Then $(V_g,\xi_k)$ has no Giroux torsion if $k \ge 3$, but for any $k \ge 1$
it has planar torsion of order~$k-1$.
In particular $(V_g,\xi_k)$ always has vanishing
ECH contact invariant and is not strongly fillable.
\end{cor}

Some more examples of planar torsion without Giroux torsion are shown
in Figure~\ref{fig:noGirouxTorsion}.

\begin{remark}
\label{remark:Seifert}
In many cases, one can easily generalize the above results from products
$S^1 \times \Sigma$ to general Seifert fibrations over~$\Sigma$.  
In particular, whenever $\Sigma$ has
genus at least four, one can find dividing sets on $\Sigma$ such that
$(S^1 \times \Sigma,\xi_\Gamma)$ has no Giroux torsion but contains a
proper subset that is a planar torsion domain 
(see Figure~\ref{fig:noGirouxTorsion}).  Then modifications outside of the
torsion domain can change the trivial fibration into arbitrary nontrivial 
Seifert fibrations with planar torsion but no Giroux torsion.  This trick
reproduces many (though not all) of the Seifert fibered $3$-manifolds for
which \cite{Massot:vanishing} proves the vanishing of the Ozsv\'ath-Szab\'o
contact invariant.
\end{remark}

\begin{remark}
\label{remark:Heegaard}
There is a significant overlap between our
ECH vanishing results and the Heegaard vanishing results proved by
Massot in \cite{Massot:vanishing} 
(see also \cites{HondaKazezMatic,Mathews:sutured}),
but neither set of results contains the other.  
In particular, the examples $(V_g,\xi_k)$ in Corollary~\ref{cor:higherOrder}
with planar torsion of order greater than~$1$
seem thus far to be beyond the reach of Heegaard Floer homology.
\end{remark}

\begin{figure}
\begin{center}
\includegraphics{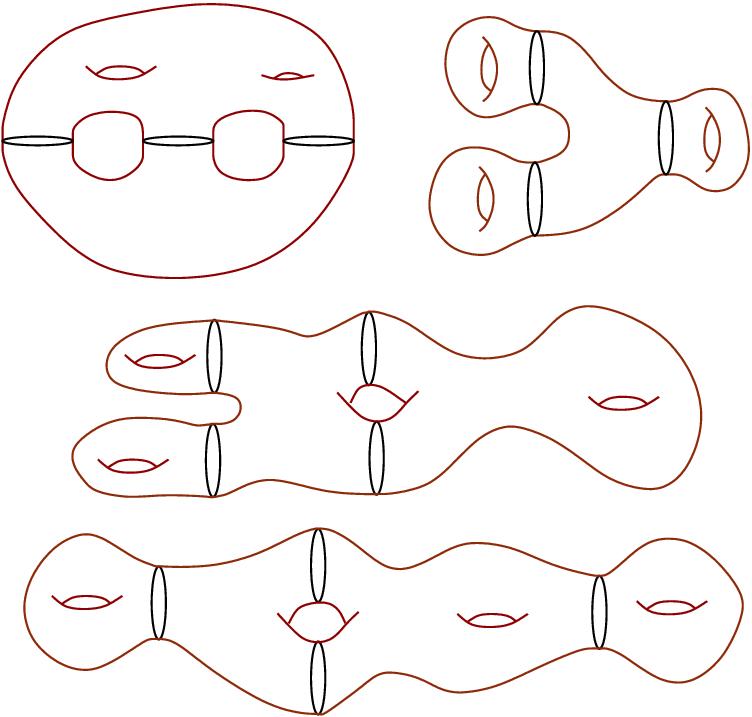}
\caption{\label{fig:noGirouxTorsion}
Some contact manifolds of the form $S^1 \times \Sigma$ that have no 
Giroux torsion but have planar torsion of orders~$2$, $2$, $3$
and~$2$ respectively.  In each case the contact structure
is $S^1$-invariant and induces the dividing set shown on $\Sigma$ in the
picture.  For the example at the upper right, Theorem~\ref{thm:twisted}
implies that the \emph{twisted} ECH contact invariant also vanishes,
so this one is not weakly fillable.  In the bottom example, the planar
torsion domain is a proper subset, thus one can make modifications outside
of this subset to produce arbitrary nontrivial Seifert fibrations
(see Remark~\ref{remark:Seifert}).}
\end{center}
\end{figure}

By a recent result of Etnyre and Vela-Vick \cite{EtnyreVelavick},
the complement of the binding of a supporting open book never contains a
Giroux torsion domain.  We will prove a natural generalization of this:

\begin{thm}
\label{thm:complement}
Suppose $(M,\xi)$ is a contact $3$-manifold supported by an open book
$\pi : M \setminus B \to S^1$.  Then any planar torsion domain in $(M,\xi)$
must intersect the binding~$B$.
\end{thm}

In order to explain our choice of terminology and the use of the word
``hierarchy,'' we now mention some related joint results
with Janko Latschev which are proved in \cite{LatschevWendl}.
These are most easily expressed by defining a contact invariant
$$
\PT(M,\xi) := \sup \left\{ k \ge 0 \ \big|\ 
\text{$(M,\xi)$ has no planar $\ell$-torsion for any $\ell < k$} \right\},
$$
which takes values in $\NN \cup \{0,\infty\}$ and is infinite if and only if
$(M,\xi)$ has no planar torsion.  Then the results stated above show that
$\PT(M,\xi) < \infty$ always implies $(M,\xi)$ is not
strongly fillable; moreover $\PT(M,\xi) \le 1$ whenever $(M,\xi)$ has 
Giroux torsion, $\PT(M,\xi) = 0$ if and only if $(M,\xi)$ is
overtwisted, and there exist contact manifolds without Giroux torsion
such that $\PT(M,\xi) < \infty$.  We claim now that contact
manifolds with larger
values of $\PT(M,\xi)$ not only exist but are, in some quantifiable sense, 
``closer'' to being fillable.  
This statement can be made precise by considering
the existence or non-existence of symplectic cobordisms between contact
manifolds with different values of $\PT(M,\xi)$, as in the following result.

\begin{thmu}[\cite{LatschevWendl}]
\label{thm:surgery}
For the contact manifold $(V_g,\xi_k)$ in
Corollary~\ref{cor:higherOrder}, $\PT(V_g,\xi_k) = k-1$.
Moreover, if $(M,\xi)$ is any contact manifold that appears as the
positive boundary of an exact symplectic cobordism whose negative boundary is 
$(V_g,\xi_k)$, then $\PT(M,\xi) \ge k-1$.
\end{thmu}

Since a contact $3$-manifold $(M,\xi)$ is tight if and only if
$\PT(M,\xi) \ge 1$, the above result can be regarded as demonstrating a
``higher order'' variant of the well-known conjecture that contact
$(-1)$-surgery on a Legendrian in a closed tight contact manifold always
produces something tight.  Indeed, since contact surgery gives rise to
a Stein cobordism, the above implies that contact surgery (or for that
matter, contact connected sums) on $(V_g,\xi_k)$ always
produces examples with $\PT(M,\xi) \ge k-1$.

\begin{remark}
It should be emphasized here that the scale defined by the invariant
$\PT(M,\xi)$ measures something completely different from the standard
quantitative measurement of Giroux torsion; the latter counts
the maximum number of
adjacent Giroux torsion domains that can be embedded in $(M,\xi)$, and can
take arbitrarily large values while $\PT(M,\xi) \le 1$.
Likewise, $(M,\xi)$ has Giroux torsion zero whenever $\PT(M,\xi) \ge 2$.
\end{remark}

The theorem above follows from some results proved in 
\cite{LatschevWendl} using notions from
Symplectic Field Theory, which also lie in the background of our
choice of terminology.  Recall that SFT is a generalization
of contact homology introduced by Eliashberg, Givental and Hofer \cite{SFT}
(see also \cite{CieliebakLatschev:propaganda} 
for the reformulation discussed here),
that defines contact invariants by counting $J$-ho\-lo\-mor\-phic 
curves with arbitrary genus and positive and negative ends 
in symplectizations of arbitrary dimension.  The chain complex of SFT
is a graded algebra of the form $\aA[[\hbar]]$, where $\hbar$ is an even
variable and $\aA$ is a graded unital algebra generated by symbols
$q_\gamma$ corresponding to closed Reeb orbits~$\gamma$.  There is then
a differential operator $\mathbf{D}_{\SFT} : \aA[[\hbar]] \to
\aA[[\hbar]]$ which counts holomorphic curves and vanishes by definition
on the ``constant'' elements $\RR[[\hbar]] \subset \aA[[\hbar]]$, hence
defining prefered homology classes in
$$
H_*^\SFT(M,\xi) := H_*(\aA[[\hbar]], \mathbf{D}_\SFT). 
$$
One then defines $(M,\xi)$ to have
\defin{algebraic $k$-torsion} if the homology satisfies the relation
$$
[\hbar^k] = 0 \in H_*^\SFT(M,\xi).
$$
For $k=0$, this means $[1] = 0$ and coincides with the notion of
\emph{algebraic overtwistedness} 
(cf.~\cite{BourgeoisNiederkrueger:algebraically}).
It follows easily from the formalism\footnote{For this informal discussion
we are taking it for granted that SFT is well defined, which was not proved
in \cite{SFT} and is quite far from obvious.  The rigorous definition of SFT, 
including the necessary abstract perturbations to achieve transversality,
is a large project in progress by Hofer-Wysocki-Zehnder, see for example
\cite{Hofer:CDM}.  The application stated above
however does not depend on this,
as it can also be proved using the ECH methods in Hutchings's appendix
to \cite{LatschevWendl}.}
of SFT that algebraic torsion of any order gives an obstruction 
to strong symplectic
filling, but in fact it is stronger, as it also implies obstructions to the
existence of exact symplectic cobordisms between certain contact manifolds.
To state this succinctly, one can define an algebraic cousin of the
invariant $\PT(M,\xi)$ by
$$
\AT(M,\xi) := \sup \left\{ k \ge 0 \ \big|\ 
\text{$(M,\xi)$ has no algebraic $\ell$-torsion for any $\ell < k$} \right\}.
$$
The above result is then a consequence of the following set of
results, which serve as our main motivation for keeping track of the
integer $k \ge 0$ in planar $k$-torsion.

\begin{thmu}[\cite{LatschevWendl}]
\label{thm:surgeryAlgebraic}
The invariant $\AT(M,\xi)$ has the following properties.
\begin{enumerate}
\item Any contact manifold $(M,\xi)$ with $\AT(M,\xi) < \infty$ is not
strongly fillable.
\item If there is an exact symplectic cobordism with positive boundary
$(M_+,\xi_+)$ and negative boundary $(M_-,\xi_-)$, then
$\AT(M_-,\xi_-) \le \AT(M_+,\xi_+)$.
\item Every contact $3$-manifold $(M,\xi)$ satisfies
$\AT(M,\xi) \le \PT(M,\xi)$.
\item For the examples $(V_g,\xi_k)$ in
Corollary~\ref{cor:higherOrder}, $\AT(V_g,\xi_k) = k-1$.
\end{enumerate}
\end{thmu}

In particular, the computation $\AT(M,\xi) \le \PT(M,\xi)$ follows
from a variation on our proof of 
Theorems~\ref{thm:ECH} and~\ref{thm:twisted}, and thus makes
essential use of the holomorphic curve results in the present article.

\subsection{Obstructions to non-separating embeddings and planarity}

We now discuss a parallel stream of results that apply to a wider class of
contact manifolds, some of which are fillable.  Observe that
in addition to ruling out symplectic fillings,
Theorem~\ref{thm:non-fillable} implies 
that contact manifolds with planar torsion can never 
appear as \emph{non-separating} contact type hypersurfaces 
in any closed symplectic $4$-manifold.  This is actually a consequence
of the following generalization of a result proved in
\cite{AlbersBramhamWendl}:

\begin{thm}
\label{thm:nonseparating}
Suppose $(M,\xi)$ is a closed contact $3$-manifold that contains a 
partially planar domain (see Definition~\ref{defn:partiallyPlanar}) and admits
a contact type embedding $\iota : (M,\xi) \hookrightarrow (W,\omega)$
into some closed symplectic $4$-manifold $(W,\omega)$.
Then $\iota$ separates~$W$.
\end{thm}

\begin{cor}
\label{cor:semifillings}
If $(M,\xi)$ is a closed contact $3$-manifold containing a partially planar
domain, then it does not admit any strong symplectic semifilling with
disconnected boundary.
\end{cor}

Recall that a \defin{semifilling} of a contact manifold $(M,\xi)$ is defined
to be a filling of $(M,\xi) \sqcup (M',\xi')$ for any (perhaps empty)
closed contact manifold $(M',\xi')$.  The corollary follows from an
observation due to Etnyre (cf.~\cite{AlbersBramhamWendl}*{Example~1.3}),
that given a filling of $(M,\xi) \sqcup (M',\xi')$ with~$M'$ non-empty,
one can attach a symplectic $1$-handle to connect~$M$ and~$M'$ and then
cap off the resulting boundary in order to realize $(M,\xi)$ 
as a non-separating contact type hypersurface.
Corollary~\ref{cor:semifillings} also generalizes similar results
proved by McDuff for the tight
$3$-sphere \cite{McDuff:boundaries} and Etnyre for all planar contact 
manifolds \cite{Etnyre:planar}.

The algebraic counterpart to Corollary~\ref{cor:semifillings} involves
the so-called \emph{$U$-map} in Embedded Contact Homology.  This is a
natural endomorphism
$$
U : \ECH_*(M,\xi) \to \ECH_{*-2}(M,\xi)
$$
defined at the chain level by counting embedded index~$2$ holomorphic 
curves through a generic point in the symplectization.  
The same definition also gives a map on ECH with twisted coefficients,
$$
\widetilde{U} : \widetilde{\ECH}_*(M,\xi) \to \widetilde{\ECH}_{*-2}(M,\xi).
$$

\begin{thm}
\label{thm:Umap}
If $(M,\xi)$ is a closed contact $3$-manifold containing a partially
planar domain, then for all integers $d \ge 1$, the image of
$U^d : \ECH_*(M,\xi) \to \ECH_*(M,\xi)$ contains~$c(\xi)$.
\end{thm}
This implies Corollary~\ref{cor:semifillings} due to some recent results
involving maps on ECH induced by cobordisms
(cf.~\cite{HutchingsTaubes:Arnold2}), though again, those results depend
on Seiberg-Witten theory, and our proof of Theorem~\ref{thm:nonseparating}
will not.

Theorem~\ref{thm:Umap} applies in particular to all planar contact
manifolds and can thus be viewed as an obstruction to planarity.  The
corresponding obstruction in Heegaard Floer homology is a known result of
Ozsv\'ath, Stipsicz and Szab\'o
\cite{OzsvathSzaboStipsicz:planar}.  Our version of the obstruction can easily 
be strengthened by observing that a planar open book is also a \emph{fully 
separating} partially planar domain, so analogously to Theorem~\ref{thm:twisted},
it yields a result with twisted coefficients---the Heegaard Floer theoretic
analogue of this result is apparently not known.

\setcounter{thmp}{\value{thm}}
\begin{thmp}
\label{thm:UmapTwisted}
If $(M,\xi)$ is a planar contact manifold,
then for all integers $d \ge 1$, 
the image of $\widetilde{U}^d : \widetilde{\ECH}_*(M,\xi) \to
\widetilde{\ECH}_*(M,\xi)$ contains $\tilde{c}(\xi)$.
\end{thmp}

\begin{remark}
\label{remark:OmegaSeparating2}
Similarly to Remark~\ref{remark:OmegaSeparating}, one can generalize
the above by defining (cf.~Definition~\ref{defn:containsPartiallyPlanar})
the notion of an \emph{$\Omega$-separating} embedding
of a partially planar domain, where $\Omega$ is a closed $2$-form 
on~$M$.  Then such an embedding produces a version of 
Theorem~\ref{thm:UmapTwisted} for ECH with coefficients 
in $\ZZ[H_2(M;\RR) / \ker\Omega]$, and implies corresponding generalizations
of Corollary~\ref{cor:semifillings}.
\end{remark}

\begin{remark}
Note that by Theorem~\ref{thm:Umap} above, there are also many non-planar
examples for which $c(\xi)$ is in the image of
$U^d$, but the corresponding statement with twisted coefficients is not
true.  The most obvious example is the standard~$T^3$, which is a
partially planar domain (see Example~\ref{ex:symmetric}) but also
admits weak semifillings with disconnected boundary
(due to Giroux \cite{Giroux:plusOuMoins}).
\end{remark}

\subsection{Holomorphic curves and open book decompositions}

The technical work in the background of the above results is a set of
theorems that we will prove in \S\ref{sec:openbook}
relating holomorphic curves and a suitably generalized notion of 
open book decompositions.  For illustration purposes, 
we now state some simplified versions of these results.

Recall that if $M$ is a closed and oriented $3$-manifold, an 
\defin{open book decomposition} is a fibration 
$$
\pi : M \setminus B \to S^1,
$$ 
where 
$B \subset M$ is an oriented link called the \defin{binding}, and the closures
of the fibers are called \defin{pages}: these are compact, oriented and
embedded surfaces with oriented boundary equal to~$B$.  
An open book is called \defin{planar}
if the pages are connected and have genus zero, and it is said to \defin{support}
a contact structure $\xi$ if the latter can be written as $\ker\alpha$
for some contact form $\alpha$ (called a \defin{Giroux form}) whose induced
Reeb vector field $X_\alpha$ is positively transverse to the interiors
of the pages and positively tangent to the binding.  
The latter definition is due to Giroux \cite{Giroux:openbook}, who established
a groundbreaking one-to-one correspondence between isomorphism classes of 
contact manifolds and
their supporting open books up to right-handed stabilization.

We refer to \S\ref{subsec:definitions} for all the technical definitions
needed to understand the following statement.  A substantial generalization
will appear in \S\ref{subsec:bigTheorem} as Theorem~\ref{thm:openbook}.

\begin{prop}
\label{thm:openbookSimple}
Suppose $(M,\xi)$ is a closed connected contact $3$-manifold with a supporting
open book decomposition $\pi : M \setminus B \to S^1$ whose pages have
genus $g \ge 0$.  Then for any numbers $\tau_0 > 0$ and $m_0 \in \NN$,
$(M,\xi)$ admits a nondegenerate Giroux form~$\alpha$ and generic
compatible almost
complex structure~$J$ on its symplectization such that the following
conditions hold:
\begin{enumerate}
\item The Reeb orbits in~$B$ have minimal period less than $\tau_0$, and 
their covers up to multiplicity~$m_0$ all have Conley-Zehnder index~$1$
with respect to the framing determined by the open book.
All Reeb orbits in $M \setminus B$ have period at least~$1$.
\item If $g=0$, then after a small isotopy of~$\pi$ fixing the binding,
there is an $(\RR\times S^1)$-parametrized family of embedded
finite energy punctured $J$-holomorphic curves
$$
u_{(\sigma,\tau)} : \dot{\Sigma} \to \RR \times M, \qquad
(\sigma,\tau) \in \RR \times S^1
$$
which are Fredholm regular and have index~$2$ and have only positive ends,
such that for each $(\sigma,\tau) \in \RR \times S^1$,
the projection of $u_{(\sigma,\tau)}$ to~$M$ is an embedding that
parametrizes $\pi^{-1}(\tau)$.
\item
If $g=0$, then every somewhere injective finite energy punctured
$J$-holomorphic curve in $\RR\times M$ whose positive ends all approach
orbits in~$B$ of covering multiplicity up to~$m_0$ is part of the
$(\RR\times S^1)$-family described above.
\item
If $g > 0$, then there is no $J$-holomorphic curve in $\RR\times M$
whose positive ends all approach distinct simply covered orbits in~$B$.
\end{enumerate}
\end{prop}

The $(\RR\times S^1)$-parametrized family of $J$-holomorphic curves in
this theorem is called a \emph{holomorphic open book}; such objects
have appeared previously in the work of Hofer-Wysocki-Zehnder
\cites{HWZ:tight3sphere,HWZ:convex} and Abbas \cite{Abbas:openbook}.
Their existence for the case $g=0$ was already established in 
\cite{Wendl:openbook} and generalized in \cite{Abbas:openbook}, and lies
in the background of various contact topological results on planar contact
manifolds, such as the proof of the Weinstein conjecture by 
Abbas-Cieliebak-Hofer \cite{ACH} and the author's proof that strong
and Stein fillability are equivalent \cite{Wendl:fillable}.
Given existence, the uniqueness statement for the $g=0$ case follows from
a straightforward but surprisingly powerful intersection
theoretic argument, using the homotopy invariant intersection number for
punctured holomorphic curves developed by Siefring 
\cite{Siefring:intersection}.  The non-existence result for $g > 0$
relies on this same argument but is much subtler, because for
analytical reasons, the existence part of the above theorem fails in
the case $g > 0$.\footnote{Holomorphic open books with pages of 
positive genus cannot be
expected to exist in general because the necessary moduli spaces of
holomorphic curves have negative virtual dimension.  Hofer \cite{Hofer:real}
suggested that this problem might be solved by introducing a 
``cohomological perturbation''
into the nonlinear Cauchy-Riemann equation in order to raise the Fredholm
index.  This program has recently been carried out by
Casim Abbas \cite{Abbas:openbook} (see also \cite{vonBergmann:embedded}),
though applications to problems such as the Weinstein conjecture are as
yet elusive, as the compactness theory for the modified nonlinear
Cauchy-Riemann equation is quite difficult.}
The situation is saved by the observation, explained in \cite{Wendl:openbook},
that one can find a highly \emph{non-generic} choice of data for which
higher genus holomorphic open books exist, and this data is compatible
with an exact stable Hamiltonian structure,
which admits a well behaved perturbation to a suitable contact form.

In \S\ref{subsec:bigTheorem}, we will state and prove a generalization
of Proposition~\ref{thm:openbookSimple} in the context of blown up and summed 
open books,
which gives us existence and uniqueness for certain holomorphic curves
in partially planar domains that have only positive ends.  Such results
make it easy to find orbit sets in the ECH chain complex that satisfy
$\p\boldsymbol{\gamma} = \boldsymbol{\emptyset}$ or
$U^d\boldsymbol{\gamma} = \boldsymbol{\emptyset}$, thus proving
Theorems~\ref{thm:ECH}, \ref{thm:twisted}, \ref{thm:Umap} 
and~\ref{thm:UmapTwisted}.

As already mentioned, our main results on fillability and embeddability
(Theorems~\ref{thm:non-fillable}, \ref{thm:complement} and
\ref{thm:nonseparating}) can also
be proved without recourse to ECH and Seiberg-Witten theory, and we
shall do this in \S\ref{sec:non-fillable}.  The main idea behind such
arguments appeared already in \cite{Wendl:fillable}: given a strong filling
whose boundary contains a planar torsion domain, we can attach a
cylindrical end and use the above correspondence between open books and
holomorphic curves to find a region near infinity that is
foliated by a stable $2$-dimensional family of holomorphic curves.
This family can then be expanded into the filling and, due to 
the analytical properties of the holomorphic curves in question,
must foliate it.  But the latter produces a
contradiction, as one can then follow the family back into a different
region of the cylindrical end where our uniqueness statement in fact
\emph{excludes} the existence of such holomorphic curves.

To make this type of argument work, we need compactness and deformation
results for families of curves in a symplectic filling that
arise from the pages of a holomorphic open book.  An example of such a
result is the following.
Suppose $(M,\xi)$ is supported by a planar open book $\pi : M \setminus B
\to S^1$, and $\alpha$ and~$J_+$ are the contact form and almost complex
structure respectively provided by Proposition~\ref{thm:openbookSimple}.
Assume also that $(M,\xi)$ is the contact type boundary of a compact
symplectic manifold $(W,\omega)$ such that near~$\p W$, 
$\omega = d\lambda$ for a $1$-form $\lambda$ that matches~$\alpha$ 
at $M = \p W$.  We can then complete $(W,\omega)$ to a noncompact symplectic
manifold by attaching a cylindrical end
$$
(W^\infty,\omega) := (W,\omega) \cup_M \left( [0,\infty) \times M, d(e^t\alpha) \right).
$$
Let $u_+ : \dot{\Sigma} \to \RR \times M$ denote one of the holomorphic
planar pages provided by Proposition~\ref{thm:openbookSimple}; applying a
suitable $\RR$-translation to $u_+$, we may assume without loss of
generality that it lies in $[0,\infty) \times M \subset W^\infty$.  
Now choose an open neighborhood $\nN(B) \subset M$ of the binding~$B$ and an
open subset $\uU \subset M$ such that
$$
u_+(\dot{\Sigma}) \subset [0,\infty) \times \uU.
$$
Finally, choose any set of data $\alpha'$, $\omega'$, $J_+'$ and~$J'$ 
with the following properties:
\begin{itemize}
\item $\alpha'$ is a nondegenerate contact form on~$M$ that matches
$\alpha$ in $\uU \cup \nN(B)$ and has only Reeb orbits of period at
least~$1$ outside of $\nN(B)$
\item $\omega'$ is a symplectic form on~$W^\infty$ that matches
$d(e^t\alpha')$ on $[0,\infty) \times M$
\item $J'_+$ is a generic almost complex structure on 
$\RR\times M$ compatible with~$\alpha'$ that matches $J_+$ on
$\RR \times (\uU \cup \nN(B))$
\item $J'$ is an $\omega'$-compatible almost complex structure on $W^\infty$
which is generic in~$W$ and matches $J'_+$ in $[0,\infty) \times M$
\end{itemize}
We then denote by $\mM(J')$ the moduli space of all unparametrized finite
energy $J'$-holomorphic curves in $W^\infty$, and let $\mM_0(J')$ denote
the connected component of this space containing~$u_+$.
A standard application of the implicit function theorem 
(see e.g.~\cite{AlbersBramhamWendl}*{Theorem~4.7}) shows that
$\mM_0(J')$ is a smooth $2$-dimensional manifold whose elements are all
embedded and do not intersect each other; in particular they foliate an
open subset of~$W^\infty$.  The key to the proofs in \S\ref{sec:non-fillable}
as well as various other applications in \cites{NiederkruegerWendl,
LisiVanhornWendl} is to show that the curves in $\mM_0(J')$ also fill a
\emph{closed} subset outside of some harmless subvariety of codimension two.
That is the main point of the following result, which is a simplified
version of Theorem~\ref{thm:compactness} proved in \S\ref{sec:compactness}.

\begin{prop}
\label{thm:compactnessSimple}
$\mM_0(J')$ is compact except for convergence in the sense of
\cite{SFTcompactness} to holomorphic buildings
of the following types:
\begin{enumerate}
\item Buildings with empty main level and a single non-empty upper level
curve in $\RR\times M$ whose projection to~$M$ is embedded,
\item Finitely many nodal curves in $W^\infty$ consisting of two embedded
index~$0$ components that intersect each other transversely.
\end{enumerate}
\end{prop}

It is instructive perhaps to compare this with the results of
McDuff \cite{McDuff:rationalRuled}: in particular, the role of McDuff's 
symplectic sphere
with nonnegative self-intersection is played by our holomorphic curve~$u_+$,
which generates a smooth $2$-dimensional family of curves that, due to
the above compactness result and the aforementioned implicit function
theorem, must fill the entirety of~$W^\infty$.  In the form stated above, this
result follows from \cite{AlbersBramhamWendl}*{Theorem~4.8}.  The version
we will prove in \S\ref{sec:compactness} for a general partially planar
domain is more complicated because one cannot generally avoid holomorphic
buildings with multiply covered components, nonetheless one can still show
that only finitely many such buildings can appear.

\subsection{Open questions and recent progress}
\label{sec:discussion}

Let us now discuss a few questions that arise from the above results,
some of which have been partially answered since the first version of this
paper appeared.
In light of the equivalence between the ECH and Ozsv\'ath-Szab\'o contact 
invariants, recently established in independent work of
Colin-Ghiggini-Honda \cite{ColinGhigginiHonda:HF=ECH3} and
Kutluhan-Lee-Taubes \cite{KutluhanLeeTaubes5}, our vanishing results for
the ECH contact invariants imply corresponding results in Heegaard Floer
homology.  Some of these were already known from the work of various
authors \cites{GhigginiHondaVanhorn,GhigginiHonda:twisted,HondaKazezMatic,
Massot:vanishing,Mathews:sutured}, but their results
appear thus far to recognize planar torsion only up to order~$1$.

\begin{question}
Can one prove within the context of Heegaard Floer homology (i.e.~without
using ECH) that the
contact invariant vanishes in the presence of planar $k$-torsion for $k \ge 2$?
\end{question}

As we sketched in the above discussion of related results in 
\cite{LatschevWendl}, the hierarchical structure encoded by 
the order $k \ge 0$ of planar $k$-torsion can be detected 
algebraically via Symplectic Field Theory, and it also can be detected by
a refinement of the ECH contact invariant explained in Hutchings's
appendix to \cite{LatschevWendl}.  The latter raises the question of
what structure in Heegaard Floer homology might also be able to see this
hierarchy, but apparently nothing is yet known about this.

\begin{question}
Can Heegaard Floer homology distinguish between two contact manifolds with
vanishing Ozsv\'ath-Szab\'o invariant but differing minimal 
orders of planar torsion?
Does this imply obstructions to the existence of exact or Stein 
cobordisms?
\end{question}

It should be mentioned that in presenting this introduction to planar
torsion, we neither claim nor believe it to be the most general source of
vanishing results for the various invariants under discussion.
For the Ozsv\'ath-Szab\'o invariant, \cite{Massot:vanishing} produces
vanishing results on some Seifert fibered $3$-manifolds that fall under
the umbrella of our Corollary~\ref{cor:examples} and 
Remark~\ref{remark:Seifert}, but also some that do not since
there is no condition requiring the existence of a planar
piece.  This phenomenon appears to be related to a generalization
of planar torsion that has recently emerged from joint work of the author
with Lisi and Van Horn-Morris: the idea is to replace the contact fiber sum
with a more general
``plumbing'' construction that produces a notion of
``higher genus binding.''  Among its applications, this allows a
substantial generalization of Corollary~\ref{cor:examples} that encompasses
all of the examples in \cite{Massot:vanishing} and many more; details of
this will appear in the forthcoming paper \cite{LisiVanhornWendl}.

And now the obvious question: what can be done in higher dimensions?
There has been significant activity in this area in the last few years.
Atsuhide Mori \cite{Mori:Lutz} showed that certain blown up open books
in dimension~$5$ produce a filling obstruction that strongly resembles the
Lutz tube and is related to Niederkr\"uger's speculative notion of 
higher-dimensional overtwistedness \cite{Plastikstufe}.
After the preprint version of the present article first appeared, Mori's
construction was generalized to all dimensions in a joint paper of the author
with Massot and Niederkr\"uger \cite{MassotNiederkruegerWendl} which also
defined a higher-dimensional notion of
Giroux torsion, giving the first examples of non-fillable
contact manifolds in all dimensions that cannot be called ``overtwisted'' in
any reasonable sense.  The constructions in \cite{MassotNiederkruegerWendl}
also give some hints as to how one might define something analogous to
higher-order planar
torsion that could be detected algebraically via SFT in all dimensions.
This subject is still in its infancy, but it now at least seems safe to
state the following conjecture:

\begin{conju}
For all $n \ge 1$ and $k \ge 0$, there exist $(2n+1)$-dimensional contact
manifolds $(M,\xi)$ with $\AT(M,\xi) = k$.  In particular, there exists in every
dimension greater than one a sequence of non-fillable contact manifolds
$\{ (M_k,\xi_k) \}_{k \ge 0}$ such that $(M_k,\xi_k)$ admits exact symplectic
cobordisms to $(M_\ell,\xi_\ell)$ if and only if $k \le \ell$.
\end{conju}

\section{The definition of planar torsion}
\label{sec:definitions}

\subsection{Blown up summed open books}
\label{subsec:summed}

We now explain the decomposition of a contact manifold into ``binding sums'' 
of supporting open books, which underlies the notion of a planar torsion domain.

Assume $M$ is an oriented smooth manifold containing two disjoint 
oriented submanifolds
$N_1,N_2 \subset M$ of real codimension~$2$, which admit an
orientation preserving diffeomorphism 
$\varphi : N_1 \to N_2$
covered by an orientation reversing isomorphism $\Phi : \nu N_1 \to \nu N_2$
of their normal bundles.  Then we can define a new smooth manifold $M_\Phi$, 
the \defin{normal sum} of $M$ along $\Phi$, by removing neighborhoods
$\nN(N_1)$ and $\nN(N_2)$ of $N_1$ and $N_2$ respectively, then gluing 
together the resulting manifolds with boundary along an orientation
reversing diffeomorphism
$$
\p\nN(N_1) \to \p\nN(N_2)
$$
determined by~$\Phi$.  This operation determines $M_\Phi$ up
to diffeomorphism, and is also well defined
in the contact cateogory: if $(M,\xi)$ is a contact manifold and $N_1,N_2$
are contact submanifolds with $\varphi : N_1 \to N_2$ a contactomorphism,
then $M_\Phi$ admits a contact structure 
$\xi_\Phi$, which agrees with $\xi$ away from $N_1$ and $N_2$ 
(cf.~\cite{Geiges:book}*{\S 7.4}).
Although the issue of uniqueness is not discussed in 
\cite{Geiges:book}*{\S 7.4}, one can show that the construction of
$\xi_\Phi$ explained there is canonical up to isotopy; in the specific
setting that we will be concerned with below, this is an obvious consequence
of the uniqueness of ``supported'' contact structures
(cf.~Definition~\ref{defn:GirouxForm} and the ensuing discussion).

We will consider the special case of the contact fiber sum where $N_1$ and
$N_2$ are disjoint components\footnote{We use the word \emph{component}
throughout to mean any open and closed subset, i.e.~a disjoint union of
connected components.}
of the binding of an open book decomposition
$$
\pi : M \setminus B \to S^1
$$
that supports~$\xi$.
Then $N_1$ and $N_2$ are automatically contact submanifolds, whose normal 
bundles come with distinguished trivializations determined by the open book.
In the following, we shall always assume that $M$ is oriented and 
the pages and binding are
assigned the natural orientations determined by the open book, so in 
particular the binding is the oriented boundary of the pages.

\begin{defn}
Assume $\pi : M \setminus B \to S^1$ is an open book decomposition on $M$.
By a \defin{binding sum} of the open book, we mean 
any normal sum $M_\Phi$ along an orientation reversing
bundle isomorphism $\Phi : \nu N_1 \to \nu N_2$ covering a diffeomorphism
$\varphi : N_1 \to N_2$, where $N_1,N_2 \subset B$ are disjoint components 
of the binding and $\Phi$ is constant with respect to the distinguished 
trivialization determined by~$\pi$.  The resulting smooth manifold will
be denoted by
$$
M_{(\pi,\varphi)} := M_\Phi,
$$
and we denote by $\iI_{(\pi,\varphi)} \subset M_{(\pi,\varphi)}$ the
closed hypersurface obtained by the identification of $\p\nN(N_1)$ with
$\p\nN(N_2)$, which we'll also call the \defin{interface}.
We will then refer to the data $(\pi,\varphi)$ as a \defin{summed open book
decomposition} of $M_{(\pi,\varphi)}$, whose
\defin{binding} is the (possibly empty) codimension~$2$ submanifold
$$
B_\varphi := B \setminus (N_1 \cup N_2) \subset M_{(\pi,\varphi)}.
$$
The \defin{pages} of $(\pi,\varphi)$ are the connected components of the
fibers of the naturally induced fibration
$$
\pi_\varphi : M_{(\pi,\varphi)} \setminus (B_\varphi \cup \iI_{(\pi,\varphi)})
\to S^1;
$$
if $\dim M = 3$, then these are
naturally oriented open surfaces whose closures are generally
immersed (distinct boundary components may sometimes coincide).

If $\xi$ is a contact structure
on $M$ supported by $\pi$, we will denote the induced contact structure
on $M_{(\pi,\varphi)}$ by
$$
\xi_{(\pi,\varphi)} := \xi_\Phi
$$
and say that $\xi_{(\pi,\varphi)}$ is \defin{supported by} the summed
open book $(\pi,\varphi)$.
\end{defn}

It follows from the corresponding fact for
ordinary open books that every summed open book decomposition supports a
contact structure, which is unique up to isotopy: in fact it depends only 
on the isotopy class of the open book $\pi : M \setminus B \to S^1$, the
choice of binding components $N_1,N_2 \subset B$ and isotopy class of 
diffeomorphism $\varphi : N_1 \to N_2$.

Throughout this discussion, $M$, $N_1$, $N_2$ and the pages of 
$\pi$ are all allowed to be disconnected (note that $\pi : M \setminus B
\to S^1$ will have disconnected pages if~$M$ itself is disconnected).
In this way, we can incorporate
the notion of a binding sum of \emph{multiple}, separate (perhaps summed)
open books, e.g.~given
$(M_i,\xi_i)$ supported by $\pi_i : M_i \setminus B_i \to S^1$ with
components $N_i \subset B_i$ for $i=1,2$, 
and a diffeomorphism $\varphi : N_1 \to N_2$,
a binding sum of $(M_1,\xi_1)$ with $(M_2,\xi_2)$ can be defined by
applying the above construction to the disjoint union $M_1 \sqcup M_2$.
We will generally use the shorthand notation
$$
M_1 \boxplus M_2
$$
to indicate manifolds constructed by binding sums of this type,
where it is understood that~$M_1$ and~$M_2$ both come with contact structures
and supporting summed open books, which combine to determine
a summed open book and supported contact structure on $M_1 \boxplus M_2$.

\begin{example}
\label{ex:T3}
Consider the tight contact structure on $M := S^1 \times S^2$ with its 
supporting open book decomposition
$$
\pi : M \setminus (\gamma_0 \cup \gamma_\infty) \to S^1 :
(t,z) \mapsto z / |z|,
$$
where $S^2 = \CC \cup \{\infty\}$, $\gamma_0 := S^1 \times \{0\}$,
$\gamma_\infty := S^1 \times \{\infty\}$ and $S^1$ is identified with the
unit circle in~$\CC$.  This open book has cylindrical
pages and trivial monodromy.  Now let $M'$ denote a second copy of the
same manifold and 
$$
\pi' : M' \setminus (\gamma_0' \cup \gamma_\infty') \to S^1
$$
the same open book.  Defining the binding sum $M \boxplus M'$ by
pairing $\gamma_0$ with $\gamma_0'$ and $\gamma_\infty$ with $\gamma_\infty'$, 
we obtain the standard contact $T^3$.  In fact, each of the tight contact 
tori $(T^3,\xi_n)$, where
$$
\xi_n = \ker\left[ \cos(2\pi n \theta)\ dx + \sin(2\pi n\theta)\ dy \right]
$$
in coordinates $(x,y,\theta) \in S^1 \times S^1\times S^1$, can be obtained
as a binding sum of $2n$ copies of the tight $S^1 \times S^2$; 
see Figure~\ref{fig:T3}.
\end{example}

\begin{figure}[hbt]
\begin{center}
\includegraphics{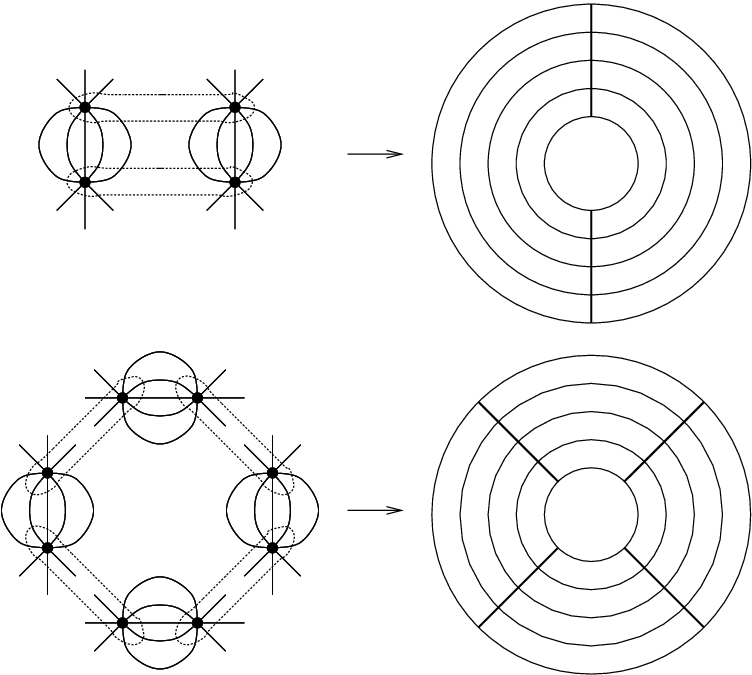}
\caption{\label{fig:T3} Two ways of producing tight contact tori from
$2n$ copies of the tight $S^1 \times S^2$.  At left, copies of $S^1 \times S^2$
are represented by open books with two binding components (depicted here
through the page)
and cylindrical pages.  For each dotted oval surrounding two binding
components, we construct the binding sum to produce the manifold
at right, containing $2n$ special pre-Lagrangian tori (the black line segments)
that separate regions foliated by cylinders.  
The results are $(T^3,\xi_n)$ for $n=1,2$.}
\end{center}
\end{figure}

\begin{example}
\label{ex:selfSum}
Using the same open book decomposition on the tight $S^1 \times S^2$
as in Example~\ref{ex:T3}, one can take only a single copy and perform
a binding sum along the two binding components~$\gamma_0$ and~$\gamma_\infty$.
The contact manifold produced by this operation 
is the quotient of $(T^3,\xi_1)$ by
the contact involution $(x,y,\theta) \mapsto (-x,-y,\theta + 1/2)$, and
is thus the torus bundle over $S^1$ with monodromy~$-1$.  The resulting
summed open book on~$T^3 / \ZZ_2$ has connected cylindrical pages, 
empty binding and a single interface
torus of the form $\iI_{(\pi,\varphi)} = \{ 2\theta = 0\}$, inducing a fibration
$$
\pi_\varphi : (T^3 / \ZZ_2) \setminus \iI_{(\pi,\varphi)} 
\to S^1 : [(x,y,\theta)] \mapsto 
\begin{cases}
y & \text{ if $\theta \in (0,1/2)$,}\\
-y & \text{ if $\theta \in (1/2,1)$.}
\end{cases}
$$
\end{example}

The following two special cases of summed open books are of crucial importance.

\begin{example}
An ordinary open book can also be regarded as a summed open book: we simply
take $N_1$ and $N_2$ to be empty.
\end{example}

\begin{example}
\label{ex:symmetric}
Suppose $(M_i,\xi_i)$ for $i=1,2$ are closed connected contact $3$-manifolds 
with supporting open books $\pi_i$ whose pages are diffeomorphic.
Then we can set $N_1=B_1$ and $N_2=B_2$, choose a diffeomorphism
$B_1 \to B_2$ and define $M = M_1 \boxplus M_2$ accordingly.  The
resulting summed open book is called \defin{symmetric}; observe that it
has empty binding, since every binding component of $\pi_1$ and
$\pi_2$ has been summed.
A simple example of this construction is $(T^3,\xi_1)$ as explained in
Example~\ref{ex:T3}, and for an even simpler example, summing two open
books with disk-like pages produces the tight $S^1 \times S^2$.
\end{example}

\begin{remark}
\label{remark:quasi}
There is a close relationship between summed open books and the notion of
open books with quasi-compatible contact structures, introduced by
Etnyre and Van Horn-Morris \cite{EtnyreVanhorn}.  A contact structure~$\xi$ 
is said to be \emph{quasi-compatible} with an open book if it admits a contact
vector field that is positively transverse to the pages and positively
tangent to the binding; if the contact vector field is also positively
transverse to~$\xi$, then
this is precisely the supporting condition, but quasi-compatibility is
quite a bit more general, and can allow e.g.~open books with empty binding.
A summed open book on a $3$-manifold 
gives rise to an open book with quasi-compatible contact
structure whenever a certain orientation condition is satisfied: this is
the result in particular whenever we construct binding sums of separate
open books that are labeled with signs in such a way that
every interface torus separates a positive piece from a negative piece.
Thus the tight $3$-tori in Figure~\ref{fig:T3} are examples, in this
case producing an open book with empty binding (i.e.~a fibration over~$S^1$)
that is quasi-compatible with all of the contact structures~$\xi_n$.  
However, it is easy to
construct binding sums for which this is not possible, 
e.g.~Example~\ref{ex:selfSum}.
\end{remark}

We now generalize the discussion to include manifolds with boundary.
Suppose $M_{(\pi,\varphi)}$ is a closed $3$-manifold with
summed open book $(\pi,\varphi)$, which has
binding~$B_\varphi$ and interface~$\iI_{(\pi,\varphi)}$, and 
$N \subset B_\varphi$ is a component of
its binding.  For each connected component $\gamma \subset N$,
identify a tubular neighborhood
$\nN(\gamma)$ of $\gamma$ with a solid torus $S^1 \times \DD$, defining
coordinates $(\theta,\rho,\phi) \in S^1 \times \DD$, where
$(\rho,\phi)$ denote polar coordinates\footnote{Throughout this paper,
we use polar coordinates $(\rho,\phi)$ on subdomains of~$\CC$ 
with the angular coordinate~$\phi$
normalized to take values in $S^1 = \RR / \ZZ$, i.e.~the actual \emph{angle}
is $2\pi\phi$.} on the disk $\DD$ and $\gamma$ is
the subset $S^1 \times \{0\} = \{ \rho = 0\}$.  Assume also that these
coordinates are adapted to the summed open book, in the sense that the 
orientation of $\gamma$ as a binding component agrees with the
natural orientation of $S^1 \times \{0\}$, and the intersections of the pages
with $\nN(\gamma)$ are of the form $\{ \phi = \text{const} \}$.
This condition determines the coordinates up to isotopy.  Then we define
the \emph{blown up} manifold $M_{(\pi,\varphi,\gamma)}$ from $M_{(\pi,\varphi)}$ by replacing
$\nN(\gamma) = S^1 \times \DD$ with $S^1 \times [0,1] \times S^1$, using
the same coordinates $(\theta,\rho,\phi)$ on the latter, i.e.~the binding
circle $\gamma$ is replaced by a $2$-torus, which now forms the boundary
of $M_{(\pi,\varphi,\gamma)}$.  If $\xi_{(\pi,\varphi)}$ is a contact structure
on $M_{(\pi,\varphi)}$ supported by $(\pi,\varphi)$, then we can define
an appropriate contact structure $\xi_{(\pi,\varphi,\gamma)}$ on
$M_{(\pi,\varphi,\gamma)}$ as follows.
Since $\gamma$ is a
positively transverse knot, the contact neighborhood theorem
allows us to choose the coordinates $(\theta,\rho,\phi)$ so that
$$
\xi_{(\pi,\varphi)} = \ker \left( d\theta + \rho^2 d\phi \right)
$$
in a neighborhood of~$\gamma$.
This formula also gives a well defined distribution on $M_{(\pi,\varphi,\gamma)}$, 
but the contact condition fails at the boundary $\{ \rho = 0 \}$.  We fix this
by making a $C^0$-small change in $\xi_{(\pi,\varphi)}$ to define a contact 
structure of the form
$$
\xi_{(\pi,\varphi,\gamma)} = \ker \left[ d\theta + g(\rho)\ d\phi \right],
$$
where $g(\rho) = \rho^2$ for $\rho$ outside a neighborhood of zero,
$g'(\rho) > 0$ everywhere and $g(0) = 0$.

Performing the above operation at all connected components $\gamma \subset N
\subset B_\varphi$ yields a compact manifold 
$M_{(\pi,\varphi,N)}$, generally with boundary, carrying a still
more general decomposition determined by the data $(\pi,\varphi,N)$, which
we'll call a \defin{blown up summed open book}.
We define its \defin{interface} to be the original interface
$\iI_{(\pi,\varphi)}$, and its \defin{binding} is
$$
B_{(\varphi,N)} = B_\varphi \setminus N.
$$
There is a natural diffeomorphism
$$
M_{(\pi,\varphi)} \setminus B_\varphi = M_{(\pi,\varphi,N)} 
\setminus \left(B_{(\varphi,N)} \cup \p M_{(\pi,\varphi,N)}\right),
$$
so the fibration
$\pi_\varphi : M_{(\pi,\varphi)} \setminus 
\left(B_\varphi \cup \iI_{(\pi,\varphi)}\right) \to S^1$ carries over to
$M_{(\pi,\varphi,N)} \setminus (B_{(\varphi,N)} \cup \iI_{(\pi,\varphi)} 
\cup \p M_{(\pi,\varphi,N)})$, and can then be
extended smoothly to the boundary to define a fibration
$$
\pi_{(\varphi,N)} :  M_{(\pi,\varphi,N)} \setminus \left(B_{(\varphi,N)} 
\cup \iI_{(\pi,\varphi)}\right) \to S^1.
$$
We will again refer to the connected components of the fibers
of~$\pi_{(\varphi,N)}$ as the \defin{pages} of $(\pi,\varphi,N)$, and
orient them in accordance with the co-orientations induced by the fibration.
Their closures are immersed surfaces which occasionally may have pairs of
boundary components that coincide as oriented $1$-manifolds, e.g.~this can
happen whenever two binding circles within the same connected open book
are summed to each other.

Note that the fibration $\pi_{(\varphi,N)} :  
M_{(\pi,\varphi,N)} \setminus \left(B_{(\varphi,N)} 
\cup \iI_{(\pi,\varphi)}\right) \to S^1$ is not enough information to fully
determine the blown up open book $(\pi,\varphi,N)$, as it does 
not uniquely determine the ``blown down'' manifold $M_{(\pi,\varphi)}$.
Indeed, $M_{(\pi,\varphi)}$ determines on each boundary torus
$T \subset \p M_{(\pi,\varphi,N)}$ a distinguished basis
$$
\{m_T,\ell_T\} \subset H_1(T),
$$
where $\ell_T$ is a boundary component of a page and 
$m_T$ is determined by the
meridian on a small torus around the binding circle to be blown up.
Two different manifolds $M_{(\pi,\varphi)}$ may sometimes produce diffeomorphic
blown up manifolds $M_{(\pi,\varphi,N)}$, which will however have different
meridians~$m_T$ on their boundaries.  Similarly, each interface torus
$T \subset \iI_{(\pi,\varphi)}$ inherits a distinguished basis
$$
\{\pm m_T,\ell_T \} \subset H_1(T)
$$
from the binding sum operation, with the difference that the meridian
$m_T$ is only well defined up to a sign.

The binding sum of an open book $\pi : M\setminus B \to S^1$
along components $N_1 \cup N_2 \subset B$
can now also be understood as a
two step operation, where the first step is to blow up $N_1$ and $N_2$,
and the second is to attach the resulting boundary tori to each other via a
diffeomorphism determined by $\Phi : \nu N_1 \to \nu N_2$.  One can choose
a supported contact structure on the blown up open book which fits
together smoothly under this attachment to reproduce the construction of
$\xi_{(\pi,\varphi,N)}$ described above.

\begin{defn}
A blown up summed open book $(\pi,\varphi,N)$ is called \defin{irreducible}
if the fibers of the induced fibration $\pi_{(\varphi,N)}$ are connected.
\end{defn}

In the irreducible case, the pages can be parametrized in a single
$S^1$-family, e.g.~an ordinary connected open book is irreducible, but a
symmetric summed open book is not.  Any blown up summed open book can however
be decomposed uniquely into \defin{irreducible subdomains}
$$
M_{(\pi,\varphi,N)} = M_{(\pi,\varphi,N)}^1 \cup \ldots \cup
M_{(\pi,\varphi,N)}^\ell,
$$
where each piece $M_{(\pi,\varphi,N)}^i$ for $i=1,\ldots,\ell$ is a compact manifold,
possibly with boundary, defined as the closure in
$M_{(\pi,\varphi,N)}$ of the region filled by some smooth $S^1$-family of pages.
Thus $M_{(\pi,\varphi,N)}^i$ carries a natural blown up summed open book
of its own, whose binding and interface are subsets of $B_\varphi$ and
$\iI_{(\pi,\varphi)}$ respectively, and
$\p M_{(\pi,\varphi,N)}^i \subset \iI_{(\pi,\varphi)} \cup \p M_{(\pi,\varphi,N)}$.
One can also write
$$
M_{(\pi,\varphi,N)} = \check{M}_{(\pi,\varphi,N)}^1 \boxplus \ldots \boxplus
\check{M}_{(\pi,\varphi,N)}^\ell,
$$
where the manifolds $\check{M}_{(\pi,\varphi,N)}^i$ also naturally carry
blown up summed open books and can be obtained from $M_{(\pi,\varphi,N)}^i$ by
blowing down $\p M_{(\pi,\varphi,N)}^i \cap \iI_{(\pi,\varphi)}.$

\begin{defn}
\label{defn:GirouxForm}
Given a blown up summed open book $(\pi,\varphi,N)$ on a manifold 
$M_{(\pi,\varphi,N)}$ with boundary, a \defin{Giroux form} for
$(\pi,\varphi,N)$ is a contact form $\lambda$ on $M_{(\pi,\varphi,N)}$
with Reeb vector field $X_\lambda$ satisfying the following conditions:
\begin{enumerate}
\item $X_\lambda$ is positively transverse to the interiors of the pages,
\item $X_\lambda$ is positively tangent to the boundaries of the closures
of the pages,
\item $\ker \lambda$ on each interface or boundary torus 
$T \subset \iI_{(\pi,\varphi)} \cup \p M_{(\pi,\varphi,N)}$ 
induces a characteristic foliation with closed leaves homologous to the 
meridian~$m_T$.
\end{enumerate}
\end{defn}

We will say that a contact structure on $M_{(\pi,\varphi,N)}$
is \defin{supported} by $(\pi,\varphi,N)$ whenever it is the kernel of a
Giroux form.
By the procedure described above, one can easily take a Giroux form for
the underlying open book $\pi : M\setminus B \to S^1$ and modify it near~$B$
to produce a Giroux form for the blown up summed open book on
$M_{(\pi,\varphi,N)}$.  Moreover, the same argument that proves uniqueness
of contact structures supported by open books
(cf.~\cite{Etnyre:lectures}*{Prop.~3.18}) shows that
any two Giroux forms are homotopic to each other through a family of Giroux 
forms.  We thus obtain the following uniqueness result for supported contact
structures.

\begin{prop}
\label{prop:GirouxUniqueness}
Suppose $M_{(\pi,\varphi,N)}$ is a compact $3$-manifold with boundary, with
a contact structure $\xi_{(\pi,\varphi,N)}$ supported by the
blown up summed open book $(\pi,\varphi,N)$, and
$(M_{(\pi,\varphi,N)},\xi_{(\pi,\varphi,N)})$ admits a contact embedding into
some closed contact $3$-manifold $(M',\xi')$.  If $\lambda$ is a
contact form on $M'$ such that
\begin{enumerate}
\item $\lambda$ defines a Giroux form on $M_{(\pi,\varphi,N)} \subset M'$, and
\item $\ker\lambda = \xi'$ on $M' \setminus M_{(\pi,\varphi,N)}$,
\end{enumerate}
then $\ker\lambda$ is isotopic to~$\xi'$.
\end{prop}

\begin{example}
\label{ex:BUSOBD}
Suppose $\Sigma$ is a compact, oriented and connected surface, 
possibly with boundary, containing a non-empty multicurve
$\Gamma \subset \Sigma$ such that $\p\Sigma \subset \Gamma$ and
$\Gamma$ divides $\Sigma$ into two (possibly disconnected) pieces 
$$
\Sigma = \Sigma_+ \cup_\Gamma \Sigma_-.
$$
By Lutz \cite{Lutz:77}, $S^1 \times \Sigma$ admits an $S^1$-invariant
contact structure $\xi_\Gamma$ which is determined uniquely up to isotopy
by the condition that the loops $S^1 \times \{z\}$ be positively/negatively
transverse to~$\xi_\Gamma$ for $z \in \interior{\Sigma_\pm}$ and Legendrian for 
$z \in \Gamma$.  Then $(S^1 \times \Sigma,\xi_\Gamma)$ is supported by
a blown up summed open book with empty binding, interface
$\iI = S^1 \times (\Gamma \setminus \p\Sigma)$ and fibration
$$
\pi : (S^1 \times \Sigma) \setminus \iI \to S^1 : (\phi,z) \mapsto
\begin{cases}
\phi & \text{ for $z \in \Sigma_+$},\\
-\phi & \text{ for $z \in \Sigma_-$}.
\end{cases}
$$
Indeed, one can write $\xi_\Gamma$ as the kernel of a contact form whose
Reeb vector field is positively/negatively transverse to the interior of
$\{*\} \times \Sigma_\pm$ and admits closed orbits of the form
$\{*\} \times \gamma$ for each dividing curve $\gamma \subset \Gamma$.
(An explicit construction of such a contact form may be found
e.g.~in \cite{LatschevWendl}.)
The distinguished meridians at $\iI$ and $\p(S^1 \times \Sigma)$ are
generated by the Legendrians $S^1 \times \{*\}$.
\end{example}

\subsection{Partially planar domains and planar torsion}
\label{subsec:torsion}

We are now ready to state the most important definitions in this paper.

\begin{defn}
\label{defn:partiallyPlanar}
A blown up summed open book on a compact manifold~$M$ is called 
\defin{partially planar} if $M \setminus \p M$ contains a planar page.
A \defin{partially planar domain} is then any contact $3$-manifold $(M,\xi)$
with a supporting blown up summed open book that is partially planar.
An irreducible subdomain
$$
M^P \subset M
$$
that contains planar pages and doesn't touch $\p M$ is called a
\defin{planar piece}, and we will refer to the complementary subdomain
$\overline{M \setminus M^P}$ as the \defin{padding}.
\end{defn}

By this definition, every planar contact manifold is a partially
planar domain (with empty padding), 
as is the symmetric summed open book obtained by summing
together two planar open books with the same number of binding components
(here one can call either side the planar piece, and the other side 
the padding).  As we'll soon see,
one can also use partially planar domains to characterize the solid torus that
appears in a Lutz twist, or the thickened torus in the
definition of Giroux torsion, as well as many more general objects.

\begin{defn}
\label{defn:containsPartiallyPlanar}
We say that a contact $3$-manifold $(M,\xi)$ with a closed $2$-form $\Omega$
contains an \defin{$\Omega$-separating} partially planar domain if there
exists a partially planar domain $(M_0,\xi_0)$ with planar piece $M_0^P
\subset M_0$ and a contact embedding
$\iota : (M_0,\xi_0) \hookrightarrow (M,\xi)$ such that for every interface
torus~$T$ of~$M_0$ lying in $M_0^P$, $\int_T \iota^*\Omega = 0$.  We say that the
domain is \defin{fully separating} if this is true for all choices of~$\Omega$.
\end{defn}

Note that in general, a $2$-torus $T$ embedded in a closed oriented 
$3$-manifold $M$ satisfies $\int_T \Omega = 0$
for all closed $2$-forms $\Omega$ on~$M$ if and only if~$T$ separates~$M$.
In a partially planar domain, any interface torus in the interior of the
planar piece is necessarily non-separating, thus the fully separating
condition implies that there are no such interface tori, and each component 
of the boundary of the planar piece also separates 
(cf.~Definition~\ref{defn:separating}).

We now come to the definition of a new symplectic filling obstruction.

\begin{defn}
\label{defn:planarTorsionDomain}
For any integer $k \ge 0$, a contact manifold $(M,\xi)$, possibly
with boundary, is called a
\defin{planar torsion domain of order~$k$} (or briefly a
\defin{planar $k$-torsion domain}) if it is supported by a
partially planar blown up summed open book $(\pi,\varphi,N)$
with a planar piece $M^P \subset M$ satisfying the following conditions:
\begin{enumerate}
\item The pages in~$M^P$ have $k+1$ boundary components.
\item The padding $\overline{M \setminus M^P}$ is not empty.
\item $(\pi,\varphi,N)$ is not a symmetric summed open book
(cf.~Example~\ref{ex:symmetric}).
\end{enumerate}

We say that a contact $3$-manifold $(M,\xi)$ has (perhaps $\Omega$-separating
or fully separating) \defin{planar $k$-torsion}
if it admits a (perhaps $\Omega$-separating or fully separating)
contact embedding of a planar $k$-torsion domain.
\end{defn}

\begin{remark}
\label{remark:itsPlanar}
The planar piece of a planar $0$-torsion domain has no interior interface
tori and only one boundary component, thus planar $0$-torsion is \emph{always}
fully separating.  It is easy to see from examples 
(cf.~Example~\ref{ex:S1invariant}) that this is not true for $k \ge 1$.
Observe also that whenever $(M,\xi)$ is closed and connected and contains
a fully separating partially planar domain $M_0 \subset M$, 
one of the following must be true:
\begin{enumerate}
\renewcommand{\labelenumi}{(\roman{enumi})}
\item $(M_0,\xi)$ is a planar torsion domain,
\item $M_0 = M$ and the interface is empty, i.e.~$(M,\xi)$ is supported
by an ordinary planar open book,
\item $M_0 = M$ and it carries a symmetric summed open book with
disk-like pages.
\end{enumerate}
In the last case, $(M,\xi)$ is contactomorphic to the tight $S^1 \times S^2$
(see Example~\ref{ex:symmetric}), which is planar.  
We thus conclude that under these assumptions, $(M,\xi)$ always either has
planar torsion or is planar.
\end{remark}

\begin{figure}
\begin{center}
\includegraphics{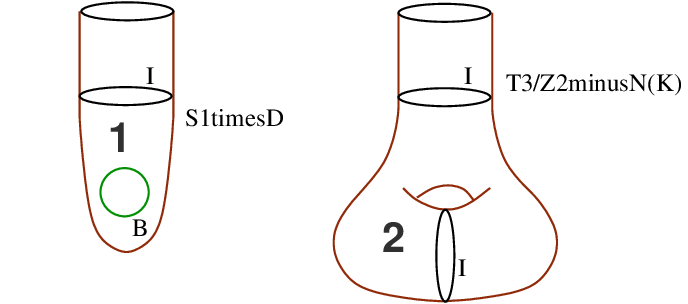}
\caption{\label{fig:moreTorsion} Schematic representations of two planar
torsion domains as described in Example~\ref{ex:schematic}.}
\end{center}
\end{figure}

\begin{example}
\label{ex:S1invariant}
The $S^1$-invariant contact manifold $(S^1 \times \Sigma,\xi_\Gamma)$
from Example~\ref{ex:BUSOBD} is a partially planar domain whenever
$\Sigma \setminus \Gamma$ has a connected component~$\Sigma_0$ of genus zero 
with $\overline{\Sigma}_0 \cap \p\Sigma = \emptyset$.  In this case
$S^1 \times \overline{\Sigma}_0$ is the planar piece, and $S^1 \times \Sigma$
is also a planar torsion domain unless the blown up summed open book
from Example~\ref{ex:BUSOBD} is symmetric, which would mean
$\p\Sigma = \emptyset$ and $\Sigma\setminus\Gamma$ has exactly two connected
components, which are diffeomorphic to each other.
Some special cases are shown in 
Figures~\ref{fig:torsionDomains} and~\ref{fig:noGirouxTorsion}.
\end{example}

\begin{example}
\label{ex:schematic}
More generally than the $S^1$-invariant examples described above, 
blown up summed open books can always be represented by schematic
pictures as in Figure~\ref{fig:moreTorsion}, which shows two examples of
planar torsion domains, each with the order labeled within the planar piece.
Here each picture shows a surface $\Sigma$ containing
a multicurve~$\Gamma$: each connected component $\Sigma_0 \subset
\Sigma \setminus \Gamma$ then represents an irreducible subdomain with
pages diffeomorphic to~$\Sigma_0$, and the components of~$\Gamma$
represent interface tori (labeled in the picture by~$\iI$).  Each irreducible
subdomain may additionally have binding circles, shown in the picture as
circles with the label~$B$.  The information in these pictures, together
with a specified monodromy map for each component of $\Sigma\setminus \Gamma$,
determine a blown up summed open book and supported contact structure
uniquely up to contactomorphism.
If we take these particular pictures with the assumption that all monodromy
maps are trivial, then the first shows a solid torus $S^1 \times \DD$
with an overtwisted contact structure that 
makes one full twist along a ray from the center 
(the binding~$B$) to the boundary.  The other picture shows the
complement of a solid torus in the torus bundle $T^3 / \ZZ_2$ from
Example~\ref{ex:selfSum}.  More precisely, one can construct it by taking
a loop $K \subset T^3 / \ZZ_2$ transverse to the pages in that example,
modifying the contact structure~$\xi$ near~$K$ by a full Lutz twist, 
and then removing
a smaller neighborhood $\nN(K)$ of~$K$ on which~$\xi$ makes a quarter twist.
Note that the appearance of genus in this picture is a bit misleading; due to
the interface torus in the interior of the bottom piece, it has planar pages
with three boundary components.
\end{example}

We can now proceed toward the proof of Theorem~\ref{thm:lowerOrder}.

\begin{defn}
\label{defn:LutzTube}
A \defin{Lutz tube} is the solid torus $S^1 \times \DD$ with
coordinates $(\theta,\rho,\phi)$, where $(\rho,\phi)$ are polar coordinates
on the closed unit disk $\DD \subset \CC$, together with the contact
structure~$\xi$ defined as the hyperplane field
\begin{equation}
\label{eqn:twisting}
\xi = \ker\left[ f(\rho)\ d\theta + g(\rho)\ d\phi \right]
\end{equation}
for some pair of smooth functions $f,g$ such that the path
$$
[0,1] \to \RR^2 \setminus\{0\} : \rho \mapsto (f(\rho),g(\rho))
$$
makes exactly one half-turn (counterclockwise) about the origin, moving from
the positive to the negative $x$-axis.  (See Figure~\ref{fig:LutzTube}.)
\end{defn}

\begin{defn}
\label{defn:GirouxTorsion}
A \defin{Giroux torsion domain} is the thickened torus $[0,1] \times T^2$
with coordinates $(\rho,\phi,\theta) \in [0,1] \times S^1 \times S^1$,
together with the contact structure $\xi$ defined via these coordinates 
as in \eqref{eqn:twisting}, where
the path $\rho \mapsto (f(\rho),g(\rho))$ makes one full (counterclockwise)
turn about the origin, beginning and ending on the positive $x$-axis.
(See Figure~\ref{fig:GirouxTorsionDomain}.)
\end{defn}

\begin{prop}
\label{prop:allerleiTorsion}
If $L \subset M$ is a Lutz tube in a closed contact $3$-manifold $(M,\xi)$,
then any open neighborhood of $L$ contains a planar $0$-torsion domain.
Similarly if $L$ is a Giroux torsion domain, then any open neighborhood of $L$
contains a planar $1$-torsion domain.
\end{prop}
\begin{proof}
Suppose $L \subset M$ is a Lutz tube.  Then for some 
$\epsilon > 0$, an open neighborhood of $L$ contains a region identified with
$$
L_\epsilon := S^1 \times \DD_{1 + \epsilon},
$$
where $\DD_r$ denotes the closed disk of
radius~$r$ and $\xi = \ker\lambda_\epsilon$ for a contact form
$$
\lambda_\epsilon = f(\rho)\ d\theta + g(\rho)\ d\phi
$$
with the following properties (see Figure~\ref{fig:fg}, left):
\begin{enumerate}
\item $f(0) > 0$ and $g(0) = 0$,
\item $f(1) < 0$ and $g(1) = 0$,
\item $f(\rho) g'(\rho) - f'(\rho) g(\rho) > 0$ for all $\rho > 0$,
\item $g'(1 + \epsilon) = 0$,
\item $f(1 + \epsilon) / g(1 + \epsilon) \in \ZZ$.
\end{enumerate}
\begin{figure}
\begin{center}
\includegraphics{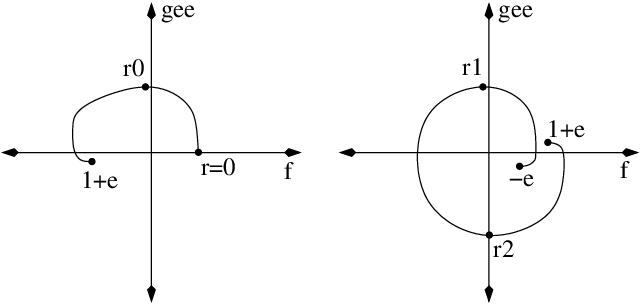}
\caption{\label{fig:fg}
The path $\rho \mapsto (f(\rho),g(\rho))$ used to define the contact form
on $L_\epsilon$ (for the Lutz tube at the left and Giroux torsion domain
at the right) in the proof of Prop.~\ref{prop:allerleiTorsion}.}
\end{center}
\end{figure}
Setting $D(\rho) := f(\rho)g'(\rho) - f'(\rho)g(\rho)$, the Reeb vector field
defined by $\lambda_\epsilon$ in the region $\rho > 0$ is
$$
X_\epsilon = \frac{1}{D(\rho)} \left[ g'(\rho) \p_\theta - f'(\rho) \p_\phi \right],
$$
and at $\rho = 0$, $X_\epsilon = \frac{1}{f(0)} \p_\theta$.  Thus $X_\epsilon$ in these 
coordinates depends only on $\rho$ and its direction is always determined by 
the slope of the path $\rho\mapsto (f(\rho),g(\rho))$ in $\RR^2$; in particular,
$X_\epsilon$ points in the $-\p_\phi$-direction at $\rho=1+\epsilon$, and in the
$+\p_\phi$-direction at some other radius $\rho_0 \in (0,1)$.  We can choose
$f$ and $g$ without loss of generality so that these are the only radii
at which $X_\epsilon$ is parallel to $\pm\p_\phi$.

We claim now that $L_\epsilon$ is a planar $0$-torsion domain with planar 
piece $L_\epsilon^P := S^1 \times \DD_{\rho_0}$.  Indeed, $L_\epsilon^P$
can be obtained from the open book on the tight $3$-sphere with disk-like
pages by blowing up the binding: the pages in the interior of $L_\epsilon^P$
are defined by $\{\theta = \text{const}\}$.  Similarly, the $\theta$-level
sets in the closure of $L_\epsilon \setminus L_\epsilon^P$ form the pages of
a blown up open book, obtained from an open book
with cylindrical pages.  The condition $f(1+\epsilon) / g(1 + \epsilon) \in \ZZ$
implies that the characteristic foliation on $T := \p L_\epsilon$
has closed leaves homologous to a primitive class $m_T \in H_1(T)$, which
together with the homology class of the Reeb orbits on~$T$ forms a basis
of $H_1(T)$.  Thus our chosen contact form
$\lambda_\epsilon$ is a Giroux form for some blown up
summed open book.  (Note that the monodromy of the blown up open book in 
$L_\epsilon \setminus L_\epsilon^P$ is not trivial
since the distinguished meridians on
$\p L_\epsilon$ and $\p L_\epsilon^P$ are not homologous.)

The argument for Giroux torsion is quite similar, so we'll only sketch it:
given $L = [0,1]\times T^2 \subset M$, we can expand $L$ slightly on 
\emph{both} sides to create a domain
$$
L_\epsilon = [-\epsilon,1+\epsilon] \times T^2,
$$
with a contact form $\lambda_\epsilon$ that induces a suitable characteristic
foliation on $\p L_\epsilon$ and
whose Reeb vector field points in the
$\pm\p_\phi$-direction at $\rho = -\epsilon$, $\rho = 1+\epsilon$ and 
exactly two other radii $0 < \rho_1 < \rho_2 < 1$ 
(see Figure~\ref{fig:fg}, right).  This splits $L_\epsilon$
into three pieces, of which 
$L_\epsilon^P := \{ \rho \in [\rho_1,\rho_2] \}$
is the planar piece of a planar $1$-torsion domain, as it can be obtained
from an open book with cylindrical pages and trivial monodromy by blowing up
both binding components.  The padding now consists of two separate 
blown up open books with cylindrical pages and nontrivial monodromy.
\end{proof}

\begin{proof}[Proof of Theorem~\ref{thm:lowerOrder}]
The only claim in the theorem that doesn't follow immediately from
Prop.~\ref{prop:allerleiTorsion} is that
$(M,\xi)$ must be overtwisted if it contains a planar $0$-torsion
domain~$M_0$.  One can see this as follows: note first that if we write
$$
M_0 = M_0^P \cup M_0',
$$
where $M_0^P$ is the planar piece and $M_0' = \overline{M_0 \setminus M_0^P}$
is the padding,
then $M_0'$ carries a blown up summed open book with pages that are not
disks (which means $(M_0,\xi)$ is not the tight $S^1 \times S^2$).
If the pages in $M_0'$ are surfaces with positive genus and one boundary
component, then one can glue one of these together with a page in $M_0^P$
to form a convex surface $\Sigma \subset M_0$ whose dividing set is
$\p M_0^P \cap \Sigma$.  The latter is the boundary of a disk in~$\Sigma$,
so Giroux's criterion (see \cite{Giroux:cercles}*{Th\'{e}or\`{e}me~4.5(a)}
or \cite{Geiges:book}*{Prop.~4.8.13})
implies the existence of an overtwisted disk near~$\Sigma$.

In all other cases the pages $\Sigma$ in $M_0'$ have multiple boundary
components
$$
\p\Sigma = C^P \cup C',
$$
where we denote by $C^P$ the connected component situated near the
interface $\p M_0^P$, and $C' = \p\Sigma \setminus C^P$.
We can then find overtwisted disks by constructing a particular Giroux form
using a small variation on the Thurston-Winkelnkemper construction as
described e.g.~in \cite{Etnyre:lectures}*{Theorem~3.13}.
Namely, choose coordinates $(s,t) \in (1/2,1] \times S^1$ on a collar
neighborhood of each component of $\p\Sigma$ and define
a $1$-form $\lambda_1$ on $\Sigma$ with the following properties:
\begin{enumerate}
\item $d\lambda_1 > 0$
\item $\lambda_1 = (1 + s)\ dt$ near each component of~$C'$
\item $\lambda_1 = (-1 + s)\ dt$ near~$C^P$
\end{enumerate}
Observe that all three conditions cannot be true unless $C'$ is non-empty,
due to Stokes's theorem.  Now following the construction described in 
\cite{Etnyre:lectures}, one can produce a Giroux form $\lambda$ on $M_0'$
which annihilates some boundary parallel curve $\ell$ near
$\p M_0^P$ in a page, and
fits together smoothly with some Giroux form in $M_0^P$, so that
$\ker\lambda$ is a supported contact structure and is isotopic to~$\xi$
by Prop.~\ref{prop:GirouxUniqueness}.  Then $\ell$ is the boundary of
an overtwisted disk.
\end{proof}

\section{Holomorphic summed open books}
\label{sec:openbook}

\subsection{Technical background}
\label{subsec:definitions}

We begin by collecting some definitions and background results on
punctured holomorphic curves that will be important for understanding
the remainder of the paper.

A \defin{stable Hamiltonian structure} on an oriented $3$-manifold $M$ is a
pair $\hH = (\lambda,\omega)$ consisting of a $1$-form $\lambda$ and $2$-form
$\omega$ such that $d\omega = 0$, $\lambda \wedge \omega > 0$ and
$\ker \omega \subset \ker d\lambda$.  Given this data, we define the 
co-oriented $2$-plane distribution $\xi = \ker\lambda$ and nowhere vanishing
vector field~$X$,
called the \defin{Reeb vector field}, which is determined by the conditions
$$
\omega(X,\cdot) \equiv 0, \qquad \lambda(X) \equiv 1.
$$
The conditions on $\lambda$ and $\omega$ imply that $\omega|_{\xi}$ gives
$\xi$ the structure of a symplectic vector bundle over~$M$, and this
distribution
with its symplectic structure is preserved by the flow of~$X$.
As an important special case, if $\lambda$ is a contact form, then one
can define a stable Hamiltonian structure in the form 
$\hH = (\lambda, h \ d\lambda)$ for any smooth function 
$h : M \to (0,\infty)$ such that $dh \wedge d\lambda \equiv 0$.  Then
$\xi$ is a positive and co-oriented contact structure, and~$X$ is the
usual contact geometric notion of the Reeb vector field: we will often
denote it in this case by~$X_\lambda$, since it is uniquely determined
by~$\lambda$.

For the rest of this section, assume $\hH = (\lambda,\omega)$ is a stable Hamiltonian
structure with the usual attached data $\xi$ and~$X$.  We say that an
almost complex structure $J$ on $\RR \times M$ is \defin{compatible with $\hH$}
if it satisfies the following conditions:
\begin{enumerate}
\item The natural $\RR$-action on $\RR\times M$ preserves~$J$.
\item $J \p_t \equiv X$, where $\p_t$ denotes the unit vector in the 
$\RR$-direction.
\item $J(\xi) = \xi$ and $\omega(\cdot,J\cdot)$ defines a symmetric,
positive definite bundle metric on~$\xi$.
\end{enumerate}
Denote by $\jJ(\hH)$ the (non-empty and contractible) space of almost 
complex structures compatible with~$\hH$.  Note that if $\lambda$ is
contact then $\jJ(\hH)$ depends only on~$\lambda$; we will in this
case say that $J$ is \defin{compatible with $\lambda$}.
 
A periodic orbit
$\gamma$ of~$X$ is determined by the data $(x,T)$, where $x : \RR \to M$
satisfies $\dot{x} = X(x)$ and $x(T) = x(0)$ for some $T > 0$.  We sometimes
abuse notation and identify $\gamma$ with the submanifold $x(\RR) \subset M$,
though technically the period is also part of the data defining~$\gamma$.
If $\tau > 0$ is the smallest positive number for which $x(\tau) = x(0)$, we 
call it the \defin{minimal period} of this orbit, and say that 
$\gamma = (x,\tau)$ is a \defin{simple}, or \defin{simply covered} orbit.  
The \defin{covering multiplicity} of an orbit $(x,T)$ is the
unique integer~$k \ge 1$ such that $T = k\tau$ for a simple orbit $(x,\tau)$.

If $\gamma = (x,T)$ is a periodic orbit and $\varphi_X^t$ 
denotes the flow of~$X$
for time~$t \in \RR$, then the restriction of the linearized flow to
$\xi_{x(0)}$ defines a symplectic isomorphism
$$
(\varphi_X^T)_* : (\xi_{x(0)},\omega) \to (\xi_{x(0)},\omega).
$$
We call $\gamma$ \defin{nondegenerate} if $1$ is not in the spectrum of this map.
More generally, a \defin{Morse-Bott submanifold} of $T$-periodic orbits is a
closed submanifold 
$N \subset M$ fixed by $\varphi_X^T$ such that for any $p \in N$,
$$
\ker \left( (\varphi_X^T)_* - \1 \right) = T_p N.
$$
We will call a single orbit $\gamma = (x,T)$ \defin{Morse-Bott} if it lies on
a Morse-Bott submanifold of $T$-periodic orbits.  Nondegenerate
orbits are clearly also Morse-Bott, with $N \cong S^1$.  We say that the
vector field $X$ is \defin{Morse-Bott} (or \defin{nondegenerate}) if all of its
periodic orbits are Morse-Bott (or nondegenerate respectively).
Since $X$ never vanishes, 
every Morse-Bott submanifold $N \subset M$ of dimension~$2$ is either a
torus or a Klein bottle.  One can show (cf.~\cite{Wendl:automatic}*{Prop.~4.1}) 
that in the former case, if $X$ is Morse-Bott,
then every orbit contained in~$N$ has the same minimal period.

To every orbit $\gamma = (x,T)$, one can associate an \defin{asymptotic
operator}, which is morally the Hessian of a certain 
functional whose critical points
are the periodic orbits.  To write it down, choose $J \in \jJ(\hH)$,
let $\mathbf{x} : S^1 \to M : t \mapsto x(Tt)$,
choose a symmetric connection~$\nabla$ on~$M$ and define
$$
\mathbf{A}_\gamma : \Gamma(\mathbf{x}^*\xi) \to \Gamma(\mathbf{x}^*\xi) :
\eta \mapsto -J (\nabla_t \eta - T \nabla_\eta X).
$$
One can show that this operator is well defined independently of the choice
of connection, and it extends to an unbounded self-adjoint operator on the
complexification of $L^2(\mathbf{x}^*\xi)$, with domain 
$H^1(\mathbf{x}^*\xi)$.  Its spectrum
$\sigma(\mathbf{A}_\gamma)$ consists of real eigenvalues
with multiplicity at most~$2$, which accumulate only at~$\pm\infty$.  
It is straightforward to show that solutions of the equation 
$\mathbf{A}_\gamma\eta = 0$ are given by $\eta(t) = (\varphi_X^{Tt})_*\eta(0)$,
thus $\gamma$ is nondegenerate if and
only if $0 \not\in \sigma(\mathbf{A}_\gamma)$, and in general if $\gamma$ belongs
to a Morse-Bott submanifold $N \subset M$, then
$$
\dim\ker\mathbf{A}_\gamma = \dim N - 1.
$$

Choosing a unitary trivialization $\Phi$ of $(\xi,J,\omega)$ along the 
parametrization
$\mathbf{x} : S^1 \to M$ identifies $\mathbf{A}_\gamma$ with a first-order
differential operator of the form
\begin{equation}
\label{eqn:localAsymptotic}
H^1(S^1,\RR^2) \to L^2(S^1,\RR^2) : \eta \mapsto -J_0 \dot{\eta} - S \eta,
\end{equation}
where $J_0$ denotes the standard complex structure on $\RR^2 = \CC$ and
$S : S^1 \to \End_\RR(\RR^2)$ is a smooth loop of symmetric real $2$-by-$2$
matrices.  Seen in this trivialization, $\mathbf{A}_\gamma \eta = 0$ defines
a linear Hamiltonian equation $\dot{\eta} = J_0 S\eta$ corresponding to the
linearized flow of~$X$ along $\gamma$, thus its flow
defines a smooth family of symplectic matrices
$$
\Psi : [0,1] \to \Spp(2)
$$
for which $1 \not\in \sigma(\Psi(1))$ if and only if~$\gamma$ is nondegenerate.
In this case, the homotopy class of the path $\Psi$ is described by its
\defin{Conley-Zehnder index} $\muCZ(\Psi) \in \ZZ$, which we use to define the
Conley-Zehnder index of the orbit $\gamma$ and of the asymptotic operator
$\mathbf{A}_\gamma$ with respect to the trivialization~$\Phi$,
$$
\muCZ^\Phi(\gamma) := \muCZ^\Phi(\mathbf{A}_\gamma) := \muCZ(\Psi).
$$
Note that in this way, $\muCZ^\Phi(\mathbf{A})$ can be defined for \emph{any}
injective operator $\mathbf{A} : \Gamma(\mathbf{x}^*\xi) \to 
\Gamma(\mathbf{x}^*\xi)$ that takes the form \eqref{eqn:localAsymptotic} in
a local trivialization.  In particular then, even if $\gamma$ is degenerate,
we can pick any $\epsilon \in \RR\setminus \sigma(\mathbf{A}_\gamma)$ and
define the ``perturbed'' Conley-Zehnder index
$$
\muCZ^\Phi(\gamma - \epsilon) := \muCZ^\Phi(\mathbf{A}_\gamma - \epsilon) :=
\muCZ(\Psi_\epsilon),
$$
where $\Psi_\epsilon : [0,1] \to \Spp(2)$ is the path of symplectic matrices
determined by the equation $(\mathbf{A}_\gamma - \epsilon) \eta = 0$ in the
trivialization~$\Phi$.  It is especially convenient to define Conley-Zehnder 
indices in this way for orbits that are degenerate but Morse-Bott: then
the discreteness of the spectrum implies that for sufficiently small
$\epsilon > 0$, the integer $\muCZ^\Phi(\gamma \pm \epsilon)$ depends only
on~$\gamma$, $\Phi$ and the choice of sign.

The eigenfunctions of $\mathbf{A}_\gamma$ are nowhere vanishing sections
$e \in \Gamma(\mathbf{x}^*\xi)$ and thus have well defined winding numbers
$\wind^\Phi(e)$ with respect to any trivialization~$\Phi$.  As shown in
\cite{HWZ:props2}, all sections in the same eigenspace have the same
winding, thus defining a function
$$
\sigma(\mathbf{A}_\gamma) \to \ZZ : \mu \mapsto \wind^\Phi(\mu),
$$
where we set $\wind^\Phi(\mu) := \wind^\Phi(e)$ for any nontrivial
$e \in \ker(\mathbf{A}_\gamma - \mu)$.  In fact, \cite{HWZ:props2} shows
that this function is nondecreasing and surjective: counting 
with multiplicity there are 
exactly two eigenvalues $\mu \in \sigma(\mathbf{A}_\gamma)$ such that
$\wind^\Phi(\mu)$ equals any given integer.  It is thus sensible to
define the integers,
\begin{equation*}
\begin{split}
\alpha^\Phi_-(\gamma - \epsilon) &= \max \{ \wind^\Phi(\mu) \ |\ 
\text{$\mu \in \sigma(\mathbf{A}_\gamma - \epsilon)$, $\mu < 0$} \}, \\
\alpha^\Phi_+(\gamma - \epsilon)  &= \min \{ \wind^\Phi(\mu) \ |\ 
\text{$\mu \in \sigma(\mathbf{A}_\gamma - \epsilon)$, $\mu > 0$} \}, \\
p(\gamma - \epsilon)      &= \alpha^\Phi_+(\gamma - \epsilon) 
- \alpha^\Phi_-(\gamma - \epsilon).
\end{split}
\end{equation*}
Note that the \defin{parity} $p(\gamma - \epsilon)$ does not depend on~$\Phi$,
and it always equals either~$0$ or~$1$ if
$\epsilon \not\in\sigma(\mathbf{A}_\gamma)$.  In this case, the Conley-Zehnder
index can be computed as
\begin{equation}
\label{eqn:CZwinding}
\muCZ^\Phi(\gamma - \epsilon) = 2\alpha^\Phi_-(\gamma - \epsilon) + p(\gamma - \epsilon) =
2\alpha^\Phi_+(\gamma - \epsilon) - p(\gamma - \epsilon).
\end{equation}

Given $\hH = (\lambda,\omega)$ and $J \in \jJ(\hH)$, fix $c_0 > 0$
sufficiently small so that $(\omega + c\ d\lambda)|_\xi > 0$ for all
$c \in [-c_0,c_0]$, and define
$$
\tT = \{ \varphi \in C^\infty(\RR,(-c,c)) \ |\ \varphi' > 0 \}.
$$
For $\varphi \in \tT$, we can define a symplectic form on $\RR \times M$ by
\begin{equation}
\label{eqn:omegaphi}
\omega_\varphi = \omega + d(\varphi\lambda),
\end{equation}
where $\omega$ and $\lambda$ are pulled back through the projection
$\RR\times M \to M$ to define differential forms on $\RR\times M$, and
$\varphi : \RR \to (-c,c)$ is extended in the natural way to a function
on $\RR\times M$.  Then any $J \in \jJ(\hH)$ is compatible with 
$\omega_\varphi$ in the
sense that $\omega_\varphi(\cdot,J\cdot)$ defines a Riemannian metric
on $\RR\times M$.  We therefore consider punctured pseudoholomorphic curves
$$
u : (\dot{\Sigma},j) \to (\RR \times M,J)
$$
where $(\Sigma,j)$ is a closed Riemann surface with a finite subset of
punctures $\Gamma \subset \Sigma$, $\dot{\Sigma} := \Sigma\setminus\Gamma$,
and $u$ is required to satisfy the \emph{finite energy} condition
\begin{equation}
\label{eqn:energy}
E(u) := \sup_{\varphi \in \tT} \int_{\dot{\Sigma}} u^*\omega_\varphi < \infty.
\end{equation}
An important example is the following: for any closed Reeb orbit
$\gamma = (x,T)$, the map
$$
u_\gamma : \RR\times S^1 \to \RR\times M : (s,t) \mapsto (Ts,x(Tt))
$$
is a finite energy $J$-holomorphic cylinder (or equivalently punctured plane), 
which we call the \defin{trivial cylinder} over~$\gamma$.  More generally,
we are most interested in punctured $J$-holomorphic curves $u : \dot{\Sigma} \to
\RR\times M$ that are asymptotically cylindrical, in the following sense.
Define the standard half cylinders
$$
Z_+ = [0,\infty) \times S^1 \quad \text{ and } \quad Z_- = (-\infty,0] \times S^1.
$$
We say that a smooth map $u : \dot{\Sigma} \to \RR\times M$ is
\defin{asymptotically cylindrical} if the punctures can be partitioned into positive
and negative subsets
$$
\Gamma = \Gamma^+ \cup \Gamma^-
$$
such that for each $z \in \Gamma^\pm$, there is a Reeb orbit 
$\gamma_z = (x,T)$, a 
closed neighborhood $\uU_z \subset \Sigma$ of~$z$ and a diffeomorphism
$\varphi_z : Z_\pm \to \uU_z \setminus \{z\}$ 
such that for sufficiently large $|s|$,
\begin{equation}
\label{eqn:asympCyl}
u \circ \varphi_z(s,t) = \exp_{(Ts,x(Tt))} h_z(s,t),
\end{equation}
where $h_z$ is a section of $\xi$ along $u_{\gamma_z}$ with
$h_z(s,t) \to 0$ for $s \to \pm\infty$, and the exponential map is
defined with respect to any choice of $\RR$-invariant connection on 
$\RR\times M$.  We often refer to the punctured neighborhoods 
$\uU_z \setminus \{z\}$ or their images
in $\RR\times M$ as the positive and negative \emph{ends} of~$u$, and
we call $\gamma_z$ the \defin{asymptotic orbit} of~$u$ at~$z$.

\begin{defn}
\label{defn:totalMultiplicity}
Suppose $N\subset M$ is a submanifold which is the union of a family of
Reeb orbits that all have the same minimal period.  Consider an asymptotically
cylindrical map $u : \dot{\Sigma} \to \RR\times M$ with punctures
$\Gamma^+ \cup \Gamma^- \subset \Sigma$ and corresponding asymptotic orbits
$\gamma_z$ with covering multiplicities~$k_z \ge 1$ for each 
$z \in \Gamma^\pm$.  Then if $k_N^\pm \ge 0$ denotes the sum of the
multiplicities $k_z$ for all punctures $z \in \Gamma^\pm$ at which
$\gamma_z$ lies in~$N$, we shall say that~$u$ approaches~$N$ with
\defin{total multiplicity}~$k_N^\pm$ at its positive or negative ends
respectively.
\end{defn}

Every asymptotically cylindrical map defines a \defin{relative homology
class} in the following sense.  Suppose $\boldsymbol{\gamma} =
\{ (\gamma_1,m_1),\ldots,(\gamma_N,m_N) \}$ is an \defin{orbit set},
i.e.~a finite collection of distinct simply covered Reeb orbits
$\gamma_i$ paired with positive integers~$m_i$.  This defines a
$1$-dimensional submanifold of~$M$,
$$
\bar{\boldsymbol{\gamma}} = \gamma_1 \cup \ldots \cup \gamma_N,
$$
together with homology classes
$$
[\boldsymbol{\gamma}] = m_1 [\gamma_1] + \ldots + m_N [\gamma_N]
$$
in both $H_1(M)$ and $H_1(\bar{\boldsymbol{\gamma}})$.  Given two
orbit sets $\boldsymbol{\gamma}^+$ and $\boldsymbol{\gamma}^-$ with
$[\boldsymbol{\gamma}^+] = [\boldsymbol{\gamma}^-] \in H_1(M)$,
denote by $H_2(M,\boldsymbol{\gamma}^+ - \boldsymbol{\gamma}^-)$ the
affine space over $H_2(M)$ consisting of equivalence classes of
$2$-chains $C$ in~$M$ with boundary $\p C$ in 
$\bar{\boldsymbol{\gamma}}^+ \cup \bar{\boldsymbol{\gamma}}^-$
representing the homology class $[\boldsymbol{\gamma}^+] - 
[\boldsymbol{\gamma}^-] \in
H_1(\bar{\boldsymbol{\gamma}}^+ \cup \bar{\boldsymbol{\gamma}}^-)$,
where $C \sim C'$ whenever $C - C'$ is the boundary of a $3$-chain in~$M$.
Now, the projection of any asymptotically cylindrical map
$u : \dot{\Sigma} \to \RR\times M$ to~$M$ can be extended as a continuous
map from a compact surface with boundary (the circle compactification of
$\dot{\Sigma}$) to~$M$, which then represents a relative homology class
$$
[u] \in H_2(M , \boldsymbol{\gamma}^+ - \boldsymbol{\gamma}^-)
$$
for some unique choice of orbit sets $\boldsymbol{\gamma}^+$ 
and~$\boldsymbol{\gamma}^-$.

As is well known 
(cf.~\cites{Hofer:weinstein,HWZ:props1,HWZ:props4}), 
every finite energy $J$-holomorphic curve
with nonremovable punctures is asymptotically cylindrical if the
contact form is Morse-Bott.  Moreover in this case, 
the section $h_z$ in \eqref{eqn:asympCyl},
which controls the asymptotic approach of $u$ to~$\gamma_z$ at 
$z \in \Gamma^\pm$, either is identically zero or 
satisfies a formula of the form\footnote{The asymptotic formula
\eqref{eqn:asympFormula} is a stronger version of a somewhat more complicated
formula originally proved in \cites{HWZ:props1,HWZ:props4}.  The stronger
version is proved in \cite{Mora}, and another exposition is given
in \cite{Siefring:asymptotics}.}
\begin{equation}
\label{eqn:asympFormula}
h_z(s,t) = e^{\mu s} (e_\mu(t) + r(s,t)),
\end{equation}
where $\mu \in \sigma(\mathbf{A}_\gamma)$ with $\pm\mu < 0$, 
$e_\mu$ is a nontrivial eigenfunction in the $\mu$-eigenspace, and the remainder term
$r(s,t) \in \xi_{x(Tt)}$ decays to zero as $s \to \pm\infty$.  It follows that unless
$h_z \equiv 0$, which is true only if $u$ is a cover of a trivial cylinder,
$u$ has a well defined \defin{asymptotic winding} about $\gamma_z$,
$$
\wind_z^\Phi(u) := \wind^\Phi(e_\mu),
$$
which is necessarily either bounded from above by $\alpha^\Phi_-(\gamma_z)$ or
from below by $\alpha^\Phi_+(\gamma_z)$, depending on the sign $z \in \Gamma^\pm$.
We say that this winding is \defin{extremal} whenever the bound is not strict.

Denote by $\mM(J)$ the moduli space of unparametrized finite energy punctured
$J$-holomorphic curves in $\RR \times M$: this consists of equivalence
classes of tuples $(\Sigma,j,\Gamma,u)$, where $\dot{\Sigma} = \Sigma\setminus
\Gamma$ is the domain of a pseudoholomorphic curve $u : (\dot{\Sigma},j) \to
(\RR\times M, J)$, and we define $(\Sigma,j,\Gamma,u) 
\sim (\Sigma',j',\Gamma',u')$
if there is a biholomorphic map $\varphi : (\dot{\Sigma},j) \to 
(\dot{\Sigma}',j')$ such that 
$u = u' \circ \varphi$.  We assign to $\mM(J)$ the natural topology defined
by $C^\infty_{\text{loc}}$-convergence on $\dot{\Sigma}$ and $C^0$-convergence
up to the ends.  It is often convenient to 
abuse notation by writing equivalence classes
$[(\Sigma,j,\Gamma,u)] \in \mM(J)$ simply as~$u$ when 
there is no danger of confusion.

If $u \in \mM(J)$ has asymptotic orbits $\{ \gamma_z \}_{z \in \Gamma}$
that are all Morse-Bott,
then a neighborhood of~$u$ in $\mM(J)$ can be described as the zero set of
a Fredholm section of a Banach space bundle (see e.g.~\cite{Wendl:automatic}).  
We say that $u$ is \defin{Fredholm regular} if this section has a surjective
linearization at~$u$, in which case a neighborhood of~$u$ in $\mM(J)$ is
a smooth finite dimensional orbifold.  Its dimension is then equal to its
\defin{virtual dimension}, which is given by the \defin{index} of~$u$,
\begin{equation}
\label{eqn:index}
\ind(u) := -\chi(\dot{\Sigma}) + 2 c_1^\Phi(u) + \sum_{z \in \Gamma^+}
\muCZ^\Phi(\gamma_z - \epsilon) - \sum_{z \in \Gamma^-}
\muCZ^\Phi(\gamma_z + \epsilon),
\end{equation}
where $\epsilon > 0$ is any small positive number, $\Phi$ is an arbitrary
choice of unitary trivialization of $\xi$ along all the asymptotic orbits
$\gamma_z$, and we abbreviate
$$
c_1^\Phi(u) := c_1^\Phi(u^*T(\RR\times M)),
$$
where the latter denotes the relative first Chern number with respect
to~$\Phi$ of the complex vector bundle $u^*T(\RR\times M) \to \dot{\Sigma}$.  
Since $T(\RR\times M)$ splits into the direct sum of~$\xi$ with a trivial 
complex line bundle, this Chern number is the same as $c_1^\Phi(u^*\xi)$, 
which can be computed by
counting the zeroes of a generic section of $u^*\xi$ that is nonzero and
constant at infinity with respect to~$\Phi$.

We say that an almost
complex structure $J \in \jJ(\hH)$ is \defin{Fredholm regular} if all
somewhere injective curves in $\mM(J)$ are Fredholm regular.
As shown in \cites{Dragnev} or the appendix of \cite{Bourgeois:homotopy}, 
the set of
Fredholm regular almost complex structures is of second category in $\jJ(\hH)$;
one therefore often refers to them as \emph{generic} almost complex structures.

It is sometimes convenient to have an alternative formula for $\ind(u)$ in
the case where $u$ is immersed.  Indeed, the linearization of the Fredholm
operator that describes $\mM(J)$ near~$u$ acts on the space of sections
of $u^*T(\RR\times M)$, which then splits naturally as $T\dot{\Sigma} \oplus
N_u$, where $N_u \to \dot{\Sigma}$ is the normal bundle, defined so that it
matches~$\xi$ at the asymptotic ends of~$u$.  As explained e.g.~in
\cite{Wendl:automatic}, the restriction of the 
linearization to $N_u$ defines a linear Cauchy-Riemann type operator
$$
\mathbf{D}_u^N : \Gamma(N_u) \to \Gamma(\overline{\Hom}_\CC(T\dot{\Sigma},N_u)),
$$
called the \defin{normal Cauchy-Riemann operator} at~$u$, and
the Fredholm index of this operator is precisely
$\ind(u)$.  Thus whenever $u$ is immersed, we can compute $\ind(u)$ directly
from the punctured version of the Riemann-Roch formula proved in
\cite{Schwarz}:
\begin{equation}
\label{eqn:normalCR}
\ind(\mathbf{D}_u^N) = \chi(\dot{\Sigma}) + 2 c_1^\Phi(N_u) + 
\sum_{z \in \Gamma^+} \muCZ^\Phi(\gamma_z - \epsilon) -
\sum_{z \in \Gamma^-} \muCZ^\Phi(\gamma_z + \epsilon).
\end{equation}

Finally, let us briefly summarize the intersection theory of punctured
$J$-holomorphic curves introduced by R.~Siefring \cite{Siefring:intersection}.
Given any asymptotically cylindrical smooth 
maps $u : \dot{\Sigma} \to \RR \times M$
and $v : \dot{\Sigma}' \to \RR\times M$, there is a symmetric pairing
$$
u * v \in \ZZ
$$
with the following properties:
\begin{enumerate}
\item $u * v$ depends only on the asymptotic orbits of $u$ and~$v$ and the
relative homology classes $[u]$ and $[v]$.
\item If $u$ and $v$ represent curves in $\mM(J)$ with non-identical 
images, then their algebraic count of intersections $u \inter v$ satisfies
$0 \le u \inter v \le u * v$.  In particular, $u * v = 0$ implies that $u$ and
$v$ never intersect.
\end{enumerate}
The first property amounts to homotopy invariance: it implies that
$u_0 * v = u_1 * v$ whenever $u_0$ and $u_1$ are connected to each other
by a continuous family of curves $u_\tau \in \mM(J)$ with fixed asymptotic
orbits.
The second property gives a \emph{sufficient} condition for two curves to
have disjoint images, but this condition is not in general \emph{necessary}:
sometimes one may have $0 = u \inter v < u * v$ if $u$ and $v$ have
an asymptotic orbit in common, and one must then expect intersections to
emerge from infinity under generic perturbations.  The number $u * v$
can also be defined when $u$ and~$v$ are holomorphic \emph{buildings} in
the sense of \cite{SFTcompactness}, so that it satisfies a similar
continuity property under convergence of curves to buildings.  The
computation of $u*v$ is then a sum of the intersection numbers between
corresponding levels, plus some additional nonnegative terms that count
``hidden'' intersections at the breaking orbits.

\begin{remark}
\label{remark:MorseBott}
The version of homotopy invariance described above assumes that~$u$
and~$v$ vary as asymptotically cylindrical maps with \emph{fixed}
asymptotic orbits, but if any of the orbits belong to Morse-Bott families,
one can define an alternative version of $u*v$ that permits the orbits to
move continuously.  This more general theory is sketched in the last section of 
\cite{Wendl:automatic}.  In general, the intersection number defined in this
way is greater than or equal to $u * v$, because it counts additional
nonnegative contributions for intersections that may emerge from infinity
as the asymptotic orbits move.  It's useful to observe however that in the
situation we will consider,
both versions agree: in particular, if $u$ and~$v$ are disjoint curves with
$u * v = 0$ and a common positive asymptotic orbit that is (for both
curves) simply covered and belongs to a Morse-Bott torus that doesn't 
intersect the images of~$u$ and~$v$, then no new intersections can appear
under a perturbation that moves the orbit (independently for both curves).
This follows from an easy computation of asymptotic winding numbers using
the definitions given in \cite{Wendl:automatic}.
\end{remark}

Similarly, if $u \in \mM(J)$ is somewhere injective, one can define the
integer $\delta(u) \ge 0$, which algebraically counts the self-intersections
of~$u$ after perturbing away its critical points, but in the punctured case
this need not be homotopy invariant.  One fixes this by introducing the
\defin{asymptotic contribution} $\delta_\infty(u) \in \ZZ$, which is also
nonnegative and counts ``hidden'' self-intersections that may emerge
from infinity under generic perturbations.  We then have
$$
0 \le \delta(u) \le \delta(u) + \delta_\infty(u),
$$
and the punctured version of the adjunction formula takes the form
\begin{equation}
\label{eqn:adjunction}
u * u = 2\left[\delta(u) + \delta_\infty(u)\right] + c_N(u) +
\left[ \bar{\sigma}(u) - \#\Gamma \right],
\end{equation}
where $\bar{\sigma}(u)$ is an integer that depends only on the
asymptotic orbits and satisfies $\bar{\sigma}(u) \ge \#\Gamma$, 
and $c_N(u)$ is the \defin{constrained normal Chern number}, 
which can be defined as\footnote{The version of $c_N(u)$ defined in
\eqref{eqn:cN} is adapted to the condition that homotopies in $\mM(J)$ are
required to fix asymptotic orbits.  A more general definition 
is given in \cite{Wendl:automatic} (see also Remark~\ref{remark:MorseBott}).}
\begin{equation}
\label{eqn:cN}
c_N(u) = c_1^\Phi(u) - \chi(\dot{\Sigma}) +
\sum_{z \in \Gamma^+} \alpha^\Phi_-(\gamma_z + \epsilon) -
\sum_{z \in \Gamma^-} \alpha^\Phi_+(\gamma_z - \epsilon).
\end{equation}
Observe that $c_N(u)$ also depends only on the asymptotic orbits
$\{\gamma_z\}_{z \in \Gamma}$ and the relative homology class~$[u]$.

\subsection{An existence and uniqueness theorem}
\label{subsec:bigTheorem}

We now prove a theorem on holomorphic open books which lies in the
background of all the results that were stated in~\S\ref{sec:intro}.
The setup is as follows.  
Assume $(M',\xi)$ is a closed $3$-manifold with a positive, co-oriented
contact structure, and it contains a compact $3$-dimensional submanifold
$M \subset M'$, possibly with boundary, on which $\xi$ is supported by a 
partially planar blown up summed open book
$$
\check{\boldsymbol{\pi}} = (\check{\pi},\check{\varphi},\check{N}).
$$
We will denote its binding and interface by~$B$ and~$\iI$ respectively,
and denote the induced fibration by
$$
\pi : M \setminus (B \cup \iI) \to S^1.
$$
Denote the irreducible subdomains by $M_i$ for $i=0,\ldots,N$, so
$$
M = M_0 \cup M_1 \cup \ldots \cup M_N
$$
for some $N \ge 0$.  If $B_i$ and $\iI_i$ denote the intersections of
$B$ and $\iI$ respectively with the interior of~$M_i$, then the restriction
of~$\pi$ to the interior of $M_i \setminus (B_i\cup \iI_i)$ extends smoothly
to its boundary as a fibration
$$
\pi_i : M_i \setminus (B_i \cup \iI_i) \to S^1.
$$
Denote by $g_i \ge 0$ the genus of the fibers of~$\pi_i$, and assume
without loss of generality that $M_0$ is a planar piece, thus
$g_0 = 0$ and $M_0 \cap \p M = \emptyset$; in particular $\p M_0 \subset \iI$.

\begin{defn}
\label{defn:subordinate}
Given the above setup, an integer $m \in \NN$ 
and an almost complex structure~$J$ compatible with
some contact form on $(M',\xi)$, we shall say that a finite energy
$J$-holomorphic curve $u : \dot{\Sigma} \to \RR\times M'$ is 
\defin{subordinate to~$\pi_0$ up to multiplicity~$m$} if the following 
conditions hold:
\begin{itemize}
\item $u$ is not a cover of a trivial cylinder,
\item All positive ends of~$u$ approach Reeb orbits
in $B_0 \cup \iI_0 \cup \p M_0$,
\item Each positive asymptotic orbit of~$u$ in~$B_0$ has covering
multiplicity at most~$m$.
\end{itemize}
Moreover, $u$ is \defin{strongly subordinate to~$\pi_0$} if the following
also holds:
\begin{itemize}
\item At its positive ends, $u$ approaches each connected component
of $B_0 \cup \p M_0$ with total multiplicity at most~$1$, and each
connected component of~$\iI_0$ with total multiplicity at most~$2$.
\end{itemize}
\end{defn}

See Definition~\ref{defn:totalMultiplicity} for an explanation of the
term \emph{total multiplicity}.  Note that the above condition allows the 
total multiplicity at any given component of $B_0 \cup \iI_0 \cup \p M_0$ to be~$0$,
which would mean that the curve has no asymptotic orbits in that component.

\begin{thm}
\label{thm:openbook}
For any numbers $\tau_0 > 0$ and $m_0 \in \NN$,
the contact manifold $(M',\xi)$ with subdomain $M \subset M'$ carrying the
blown up summed open book~$\check{\boldsymbol{\pi}}$ described above
admits a Morse-Bott contact
form $\lambda$ and compatible Fredholm regular almost complex structure $J$
with the following properties.
\begin{enumerate}
\item The contact structure $\ker\lambda$ is isotopic to~$\xi$.
\item On $M$, $\lambda$ is a Giroux form for~$\check{\boldsymbol{\pi}}$.
\item The components of $\iI \cup \p M$ are all Morse-Bott submanifolds, while
the Reeb orbits in $B$ are nondegenerate and elliptic, and their covers
for all multiplicities up to~$m_0$ have Conley-Zehnder index~$1$ with
respect to the natural trivialization determined by the pages.
\item All Reeb orbits in $B_0 \cup \iI_0 \cup \p M_0$ have minimal period at
most~$\tau_0$, while every other closed orbit of $X_\lambda$
in~$M'$ has minimal period at least~$1$.
\item For each component $M_i$ with $g_i = 0$, the fibration
$\pi_i : M_i \setminus (B_i \cup \iI_i) \to S^1$ 
admits a $C^\infty$-small perturbation
$\hat{\pi}_i : M_i \setminus (B_i \cup \iI_i) \to S^1$ such that the 
interior of each
fiber $\hat{\pi}_i^{-1}(\tau)$ for $\tau \in S^1$ lifts uniquely to an
$\RR$-invariant family of properly embedded surfaces
$$
S^{(i)}_{\sigma,\tau} \subset \RR \times M_i,
\qquad (\sigma,\tau) \in \RR\times S^1,
$$
which are the images of embedded finite energy $J$-holomorphic curves
$$
u^{(i)}_{\sigma,\tau} = (a^{(i)}_\tau + \sigma, F^{(i)}_\tau)
: \dot{\Sigma}_i \to \RR \times M_i,
$$
all of them Fredholm regular with index~$2$, and with only positive ends.
\item A finite energy $J$-holomorphic curve~$u$ in $\RR\times M'$
parametrizes one of the planar 
surfaces $S^{(i)}_{\sigma,\tau}$ described above whenever either of the
following holds:
\begin{itemize}
\item $u$ is strongly subordinate to~$\pi_0$,
\item $u$ is somewhere injective, subordinate to~$\pi_0$ up to 
multiplicity~$m_0$ and intersects the interior of~$M_0$.
\end{itemize}
\end{enumerate}
\end{thm}

In addition to the applications treated in
\S\ref{sec:proofs}, Theorem~\ref{thm:openbook} implies a wide range of 
existence results
for finite energy foliations, e.g.~it could be used to reduce the construction
in \cite{Wendl:OTfol} to a few lines,
after observing that every overtwisted contact structure is supported by
a variety of summed open books with only planar pages.
The proof of the theorem will occupy
the remainder of~\S\ref{subsec:bigTheorem}.

\subsubsection{A family of stable Hamiltonian structures}

The first step in the proof is to construct a specific almost
complex structure on $\RR\times M$ for which all pages 
of~$\check{\boldsymbol{\pi}}$ admit holomorphic lifts.  
We will follow the approach in
\cite{Wendl:openbook} and refer to the latter for details in a few places
where no new arguments are required.  The idea is to present each subdomain
$M_i$ as an abstract open book that supports a stable Hamiltonian
structure which is contact near $B \cup \iI \cup \p M$ 
and integrable elsewhere.

We must choose suitable coordinate systems near each component of
the binding, interface and boundary.  Choose $r > 0$ and let
$\DD_r \subset \RR^2$ denote the closed disk of radius~$r$.
For each binding circle $\gamma \subset B$, choose a small
tubular neighborhood $\nN(\gamma)$ and identify it with the solid torus
$S^1 \times \DD_r$ with coordinates $(\theta,\rho,\phi)$, where
$(\rho,\phi)$ denote polar coordinates on $\DD_r$.
If $r$ is sufficiently small then
we can arrange these coordinates so that the following conditions
are satisfied:
\begin{itemize}
\item $\gamma = S^1 \times \{0\}$, with the natural orientation of $S^1$
matching the co-orientation of $\xi$ along~$\gamma$
\item $\pi(\theta,\rho,\phi) = \phi$ on $\nN(\gamma) \setminus \gamma$
\item $\xi = \ker ( d\theta + \rho^2 \ d\phi)$
\end{itemize}

Similarly, for each connected component $T \subset \p M$, let
$\widehat{\nN(T)} \subset M'$ denote a neighborhood that is split into
two connected components by~$T$, and denote 
$\nN(T) = \widehat{\nN(T)} \cap M$.  Identify $\widehat{\nN(T)}$ with
$S^1 \times [-r,r] \times S^1$ with coordinates $(\theta,\rho,\phi)$
such that:
\begin{itemize}
\item $\nN(T) = S^1 \times [0,r] \times S^1$
\item For each $\phi_0 \in S^1 $ the oriented loop
$S^1 \times \{0,\phi_0\}$ in~$T$ is positively transverse to~$\xi$
\item $\pi(\theta,\rho,\phi) = \phi$ on $\nN(T)$
\item $\xi = \ker (d\theta + \rho\ d\phi)$
\end{itemize}

Finally, we choose two coordinate systems for neighborhoods $\nN(T)$
of each interface torus $T \subset \iI$, assuming that $T$ divides $\nN(T)$
into two connected components 
$$
\nN(T) \setminus T = \nN_+(T) \cup \nN_-(T).
$$
Choose an identification of $\nN(T)$ with $S^1 \times [-r,r] \times S^1$
and denote the resulting coordinates by 
$(\theta_+,\rho_+,\phi_+)$, which we arrange to have the following
properties:
\begin{itemize}
\item $T = S^1 \times \{0\} \times S^1$, $\nN_+(T) =
S^1 \times (0,r] \times S^1$ and $\nN_-(T) = S^1 \times [-r,0) \times S^1$
\item For each $\phi_0 \in S^1$ the oriented loop
$S^1 \times \{0,\phi_0\}$ in~$T$ is positively transverse to~$\xi$
\item $\pi(\theta_+,\rho_+,\phi_+) = \phi_+$ on $\nN_+(T)$ and
$\pi(\theta_+,\rho_+,\phi_+) = -\phi_+ + c$ on $\nN_-(T)$ for some
constant $c \in S^1$
\item $\xi = \ker (d\theta_+ + \rho_+\ d\phi_+)$
\end{itemize}
Given these coordinates, it is natural to define a second coordinate
system $(\theta_-,\rho_-,\phi_-)$ by
\begin{equation}
\label{eqn:plusMinus}
(\theta_-,\rho_-,\phi_-) = (\theta_+,-\rho_+,-\phi_+ + c).
\end{equation}
Then the coordinates $(\theta_-,\rho_-\phi_-)$ satisfy minor variations
on the properties listed above: in particular
$\xi = \ker ( d\theta_- + \rho_- \ d\phi_-)$ and
$\pi(\theta_-,\rho_-,\phi_-) = \phi_-$ on $\nN_-(T)$.  In the following,
we will use separate coordinates on the two components of $\nN(T) \setminus T$,
denoting both by $(\theta,\rho,\phi)$:
$$
(\theta,\rho,\phi) :=
\begin{cases}
(\theta_+,\rho_+,\phi_+) & \text{ on $\nN_+(T)$},\\
(\theta_-,\rho_-,\phi_-) & \text{ on $\nN_-(T)$}.
\end{cases}
$$
Then $\pi(\theta,\rho,\phi) = \phi$ and
$\xi = \ker(d\theta + \rho\ d\phi)$ everywhere on $\nN(T) \setminus T$.
Observe that these coordinates on $\nN_+(T)$ or $\nN_-(T)$ separately can
be extended smoothly to the closures
$\overline{\nN_+(T)}$ and $\overline{\nN_-(T)}$, though in particular
the two $\phi$-coordinates are different where they overlap at~$T$.

\begin{notation}
For any open and closed subset $N \subset B \cup \iI \cup \p M$, we shall
in the following denote by $\nN(N)$ the union of all the neighborhoods
$\nN(\gamma)$ and $\nN(T)$ constructed above for the connected components
$\gamma,T \subset N$.  Thus for example,
$$
\nN(B \cup \iI \cup \p M)
$$
denotes the union of all of them.
\end{notation}

The complement $M \setminus \nN(B \cup \iI \cup \p M)$ 
is diffeomorphic to a mapping
torus.  Indeed, let $P$ denote the closure of 
$\pi^{-1}(0) \cap (M \setminus \nN(B \cup \iI \cup \p M))$, a compact
surface whose boundary components are in one to one correspondence with
the connected components of $\nN(B \cup \iI \cup \p P) \setminus \iI$.
The monodromy map of the fibration $\pi$ defines a diffeomorphism
$\psi : P \to P$, which preserves connected components and without loss of
generality has support away from~$\p P$, so we define the mapping torus
$$
P_\psi = (\RR \times P) / \sim,
$$
where $(t + 1,p) \sim (t,\psi(p))$.  This comes with a natural fibration
$\phi : P_\psi \to S^1$ which is trivial near the boundary, so for a
sufficiently small collar neighborhood $\uU \subset P$ of $\p P$, a
neighborhood of $\p P_\psi$ can be identified with $S^1 \times \uU$.  
Choose positively oriented coordinates on each connected component
of~$\uU$
$$
(\theta,\rho) : \uU \to [r - \delta,r + \delta) \times S^1
$$
for some small $\delta > 0$.  This defines coordinates $(\phi,\theta,\rho)$
on a collar neighborhood of $\p P_\psi = S^1 \times \p P$,
so identifying these for $\rho \in (r - \delta,r]$ with the
$(\theta,\rho,\phi)$ coordinates chosen above on the corresponding
components of $\nN(B \cup \iI \cup \p M) \setminus \iI$ 
defines an attaching map, such that the union
$$
P_\psi \cup \nN(B \cup \iI \cup \p M)
$$
is diffeomorphic to~$M$, and the $\phi$-coordinate, which is globally
defined outside of $B \cup \iI$, corresponds to the fibration
$\pi : M \setminus (B \cup \iI) \to S^1$.

Choose a number $\delta' > \delta$ with $r - \delta' > 0$, and for each
of the coordinate neighborhoods in $\nN(B \cup \iI \cup \p M) \setminus \iI$, 
define a $1$-form of the form
$$
\lambda_0 = f(\rho)\ d\theta + g(\rho) \ d\phi,
$$
with smooth functions $f , g : [0,r] \to \RR$ chosen so that
\begin{enumerate}
\item $\ker\lambda_0 = \xi$ on a smaller neighborhood of $B \cup \iI\cup \p M$.
\item For $\nN(\iI)\setminus \iI$, $f(\rho)$ and $g(\rho)$ extend smoothly to
$[-r,r]$ as even and odd functions respectively.
\item The path $[0,r] \to \RR^2 : \rho \mapsto (f(\rho),g(\rho))$ 
moves through the first quadrant from the positive real axis to $(0,1)$
and is constant for $\rho \in [r-\delta,r]$.
\item The function
$$
D(\rho) := f(\rho) g'(\rho) - f'(\rho) g(\rho)
$$
is positive and $f'(\rho)$ is negative for all $\rho \in (0,r-\delta)$.
\item $g(\rho) = 1$ for all $\rho \in [r - \delta',r]$.
\end{enumerate}
Some possible pictures of the path $\rho\mapsto (f(\rho),g(\rho)) \in \RR^2$
(with extra conditions that will be useful in the proof of
Lemma~\ref{lemma:dynamics}) are shown in
Figure~\ref{fig:smallPeriods}.
Note that the functions $f$ and $g$ must generally be chosen individually
for each connected component of $\nN(B \cup \iI \cup \p M)$.
Extend $\lambda_0$ over $M' \setminus M$ so that $\ker\lambda_0 = \xi$
on this region, and extend it over $P_\psi$ as $\lambda_0 = d\phi$.
The kernel $\xi_0 := \ker\lambda_0$ is then a confoliation on $M'$: it
is contact outside of $M$ and near $B \cup \iI \cup \p M$, while integrable
and tangent to the fibers on~$P_\psi$.  In particular $\lambda_0$ is
contact in the region $\{ \rho < r - \delta \}$ near $B \cup \iI \cup
\p M$, and its Reeb vector field here is
\begin{equation}
\label{eqn:Reeb}
X_0 = \frac{g'(\rho)}{D(\rho)} \p_\theta - \frac{f'(\rho)}{D(\rho)} \p_\phi,
\end{equation}
which is positively transverse to the pages $\{ \phi = \text{const} \}$
and reduces to $\p_\phi$ for $\rho \in [r - \delta',r]$, which contains
the region where $P_\psi$ and $\nN(B \cup \iI \cup \p M)$ overlap.

Proceeding as in \cite{Wendl:openbook}, choose next a $1$-form $\alpha$
on~$P_\psi$ such that $d\alpha$ is positive on the fibers and, in the
chosen coordinates $(\phi,\theta,\rho)$ near $\p P_\psi$, 
$\alpha$ takes the form
$$
\alpha = (1 - \rho) \ d\theta,
$$
where we assume $r > 0$ is small enough so that $1 - \rho > 0$ when
$r \in [r-\delta,r+\delta)$.  Then if $\epsilon > 0$ is sufficiently
small, the $1$-form
$$
\lambda_\epsilon := d\phi + \epsilon \alpha
$$
is contact on $P_\psi$.  We extend it to the rest of $M'$ by setting
$\lambda_\epsilon = \lambda_0$ on $M' \setminus M$, and on
$\nN(B \cup \iI \cup \p M)$,
$$
\lambda_\epsilon = f_\epsilon(\rho)\ d\theta + g_\epsilon(\rho)\ d\phi,
$$
where the functions $f_\epsilon, g_\epsilon : [0,r] \to \RR$ satisfy
\begin{enumerate}
\item $(f_\epsilon(\rho)),g_\epsilon(\rho)) = (f(\rho),g(\rho))$ for
$\rho \le r - \delta'$,
\item $g_\epsilon(\rho) = 1$ and $f_\epsilon'(\rho) < 0$ for
$\rho \in [r - \delta',r - \delta]$,
\item $(f_\epsilon(\rho),g_\epsilon(\rho)) = (\epsilon(1 - \rho),1)$ for
$\rho \in [r - \delta,r]$,
\item $f_\epsilon \to f$ and $g_\epsilon \to g$ in $C^\infty$ as $\epsilon \to 0$.
\end{enumerate}
Now $\lambda_\epsilon$ is a contact form everywhere on~$M'$, and
$\lambda_\epsilon \to \lambda_0$ in $C^\infty$ as $\epsilon \to 0$.  
Denote the corresponding contact structure by
$$
\xi_\epsilon = \ker\lambda_\epsilon.
$$
The Reeb vector field $X_\epsilon$ of $\lambda_\epsilon$ is defined by the
obvious analogue of \eqref{eqn:Reeb} near $B \cup \iI \cup \p M$, is 
independent of $\epsilon$ on $M' \setminus M$, and on $P_\psi$ is
determined uniquely by the conditions
$$
d\alpha(X_\epsilon,\cdot) \equiv 0, \qquad d\phi(X_\epsilon) +
\epsilon\alpha(X_\epsilon) \equiv 1.
$$
It follows that as $\epsilon \to 0$, $X_\epsilon$ converges to a smooth
vector field $X_0$ that matches \eqref{eqn:Reeb} near
$B \cup \iI \cup \p M$ and on $P_\psi$ is determined by
\begin{equation}
\label{eqn:X0}
d\alpha(X_0,\cdot) \equiv 0 \quad\text{ and } \quad d\phi(X_0) \equiv 1.
\end{equation}
Observing that $X_\epsilon$ is always positively transverse to the pages
$\{ \phi = \text{const} \}$, and applying
Proposition~\ref{prop:GirouxUniqueness}, we have:

\begin{lemma}
For $\epsilon > 0$ sufficiently small, $\xi_\epsilon$ is a contact structure
on $M'$ isotopic to~$\xi$, 
and $\lambda_\epsilon$ is a Giroux form for~$\check{\boldsymbol{\pi}}$.
\end{lemma}

In order to turn $\lambda_\epsilon$ into a stable Hamiltonian structure,
we define an exact taming form as follows.
For each coordinate neighborhood in $\nN(B \cup \iI \cup \p M) \setminus \iI$,
fix a smooth function $h : [r - \delta',r - \delta] \to \RR$ such that
$h' < 0$, $h(\rho) = f(\rho) + c$ for $\rho$ near $r - \delta'$ and some
constant $c \ge 0$, and $h(\rho) = 1 - \rho$ for $\rho$ near 
$r - \delta$.  
For each interface torus $T \subset \iI$ the function $f(\rho)$ is the same 
on $\nN_+(T)$ as on $\nN_-(T)$, thus we may assume the same is true
of~$h(\rho)$ and~$c$.  Then
$$
F(\rho) :=
\begin{cases}
1 - \rho & \text{ for $\rho \in [r - \delta,r]$},\\
h(\rho)  & \text{ for $\rho \in [r - \delta',r - \delta]$},\\
f(\rho) + c & \text{ for $\rho \in [0, r - \delta']$}
\end{cases}
$$
defines a smooth function on $[0,r)$ which, for components of $\nN(\iI)$,
has a smooth even extension to $[-r,r]$.  By choosing $f(\rho)$ appropriately
on the components of $\nN(\p M)$, one can also arrange $c=0$; it will be
convenient (e.g.~for Lemma~\ref{lemma:dynamics} below) to assume this for 
$\nN(\p M)$ but leave the choice of $c \ge 0$ and thus $f(\rho)$
arbitrary everywhere else.  Now there is a smooth $1$-form
$\hat{\alpha}$ on~$M'$ such that
$$
\hat{\alpha} = \begin{cases}
\alpha + d\phi & \text{ on $P_\psi$},\\
F(\rho)\ d\theta + g(\rho)\ d\phi & \text{ on $\nN(B \cup \iI \cup \p M)$},\\
\lambda_0 & \text{ on $M' \setminus M$},
\end{cases}
$$
and we use this to define an exact $2$-form
$$
\omega = d\hat{\alpha}.
$$
We claim that $(\lambda_0,\omega)$ defines a stable Hamiltonian structure
on~$M'$.  Indeed, outside $M$ and in a sufficiently small neighborhood of
$B \cup \iI \cup \p M$ this is clear since $\lambda_0$ is contact
and $\omega = d\lambda_0$.  On the subsets described in coordinates
by $r - \delta' \le \rho < r - \delta$, $\lambda_0$ is still contact
and $\omega = -h'(\rho)\ d\theta \wedge d\rho = \frac{h'(\rho)}{f'(\rho)}
d\lambda_0$, thus $\omega$ has maximal rank and its kernel is spanned
by~$X_0$.  On $P_\psi$, $d\lambda_0 = 0$ and $\omega = d\alpha$ annihilates
$X_0$ by \eqref{eqn:X0}, so the claim is proved.  In fact, for
$\epsilon > 0$ sufficiently small, we still have
$\omega|_{\xi_\epsilon} > 0$ and the kernel of $\omega$ is still spanned
by $X_\epsilon$, thus we've proved:

\begin{prop}
For sufficiently small $\epsilon \ge 0$, 
$$
\hH_\epsilon := (\lambda_\epsilon,\omega)
$$
defines a stable Hamiltonian structure on~$M'$.
\end{prop}

\begin{defn}
\label{defn:goodSHS}
Any smooth family $\hH_\epsilon = (\lambda_\epsilon,\omega)$ of 
stable Hamiltonian structures on $M'$ defined for small $\epsilon \ge 0$ 
by the procedure above will be said to be
\defin{adapted} to~$\check{\boldsymbol{\pi}}$.
\end{defn}

\begin{lemma}
\label{lemma:dynamics}
There exists a number $\tau_1 > 0$ so that for any $\tau_0 > 0$
and $m_0 \in \NN$, a family
of stable Hamiltonian structures $\hH_\epsilon = (\lambda_\epsilon,\omega)$
on $M'$ adapted to~$\check{\boldsymbol{\pi}}$
can be constructed so as to satisfy the following additional conditions
on the Reeb vector fields $X_\epsilon$:
\begin{enumerate}
\item The interface and boundary tori are Morse-Bott submanifolds, and all
closed orbits in a neighborhood of $\iI \cup \p M$ are also Morse-Bott.
\item Each connected component $\gamma \subset B$ and all its multiple
covers are nondegenerate elliptic orbits, and their covers up to 
multiplicity $m_0$ all have
Conley-Zehnder index~$1$ with respect to
the natural trivialization of $\xi$ along $\gamma$ 
determined by the coordinates.
\item All orbits in $B_0 \cup \iI_0 \cup \p M_0$ have minimal period at most
$\tau_0$, while all other orbits have period at least~$\tau_1$.
\end{enumerate}
Moreover for each $\epsilon > 0$ sufficiently small, the contact form
$\lambda_\epsilon$ admits a $C^\infty$-small perturbation to a
globally Morse-Bott contact form whose Reeb vector field still 
satisfies the above conditions.
\end{lemma}
\begin{proof}
We first prove that the stated conditions can be established for~$X_0$.

If $\gamma \subset B$ is a binding circle, then $\gamma$ and all its
multiple covers can be made nondegenerate and elliptic by choosing
the functions $f$ and $g$ so that
$$
f'(\rho) / g'(\rho) \in \RR \setminus \QQ \quad
\text{for all $\rho > 0$ sufficiently small}.
$$
This implies that the slope of the curve $\rho \mapsto (f(\rho),g(\rho)) \in
\RR^2$ is constant for $\rho$ near~$0$, and this slope determines the
Conley-Zehnder index of~$\gamma$; in particular, the stated
condition is satisfied whenever $f''(0) / g''(0)$ is a negative number
sufficiently close to~$0$.  Assume this from now on.

Similarly, we make every orbit in a neighborhood of $\iI \cup \p M$
Morse-Bott by assuming that in such a neighborhood,
$\lambda_0 = f(\rho)\ d\theta + g(\rho)\ d\phi$ where $f$ and $g$
satisfy
$$
f'(\rho) g''(\rho) - f''(\rho) g'(\rho) > 0.
$$
This means that the path $\rho \mapsto (f(\rho),g(\rho)) \in \RR^2$
has nonzero inward angular acceleration as it winds (counterclockwise)
about the origin; clearly for $\nN(\iI)$ we can also still safely 
assume that $f$ and $g$ are restrictions of even and odd functions 
respectively on $[-r,r]$.

We now show that the periods of the orbits in $B_0 \cup \iI_0 \cup \p M_0$ 
can be made arbitrarily small compared to all other periods.  Observe that by
\eqref{eqn:Reeb}, the Reeb flow as we've constructed it preserves the
concentric tori $\{ \rho = \text{const} \}$ in the neighborhood
$\nN(B_0 \cup \iI_0 \cup \p M_0)$, thus it also preserves $M' 
\setminus \nN(B_0 \cup \iI_0 \cup \p M_0)$.
Since the latter has compact closure, there is a positive lower bound for the
periods of all closed orbits in $M' \setminus \nN(B_0 \cup \iI_0 \cup \p M_0)$, 
so it 
will suffice to leave $\lambda_0$ fixed in this region and reduce the periods
in $B_0 \cup \iI_0 \cup \p M_0$ while preserving a lower bound for all other orbits
in $\nN(B_0 \cup \iI_0 \cup \p M_0)$.

Consider a binding orbit $\gamma \subset B_0$: writing $\lambda_0$ as
$f(\rho)\ d\theta + g(\rho)\ d\phi$ near~$\gamma$, the period of $\gamma$
is $f(0) > 0$.  Choosing sufficiently small constants $\tau > 0$ and 
$\epsilon_0 > 0$, we impose the following additional 
conditions on $f$ and~$g$ (see Figure~\ref{fig:smallPeriods}, left):
\begin{itemize}
\item $(f(0),g(0)) = (\tau,0)$,
\item For all $\rho \in (0,r]$,
$$
\frac{g'(\rho)}{-f'(\rho)} \le \frac{1}{\tau} + \epsilon_0 \in
\RR\setminus \QQ,
$$
with equality for $\rho \le 2r / 3$.
\item For $\rho \in [2r/3, r]$, $g(\rho) \ge 2/3$ and $f(\rho) \le \tau/3$.
\end{itemize}
Since $f'(\rho) / g'(\rho)$ is irrational for $\rho \le 2r/3$, all closed
orbits in $\nN(\gamma) \setminus \gamma$ are outside this region.  
For any $\rho_0 \in [2r/3,r]$, \eqref{eqn:Reeb} implies that a
Reeb orbit in $\{ \rho = \rho_0 \}$ has its $\phi$-coordinate increasing
at the constant rate of $-f'(\rho_0) / D(\rho_0)$.  Its period is thus at least
\begin{equation}
\label{eqn:lowerBound}
\begin{split}
\left| \frac{D(\rho_0)}{f'(\rho_0)} \right| &=
\left| \frac{f(\rho_0) g'(\rho_0) - f'(\rho_0) g(\rho_0)}{f'(\rho_0)} \right| \ge
| g(\rho_0) | - \left| f(\rho_0) \frac{g'(\rho_0)}{f'(\rho_0)} \right| \\ 
&\ge \frac{2}{3} - \left| \frac{\tau}{3} \left( \frac{1}{\tau} + \epsilon_0 \right) \right|
= \frac{2}{3} - \frac{1}{3}( 1 + \tau \epsilon_0) > 0.
\end{split}
\end{equation}
We can therefore keep these periods bounded away from zero while shrinking
$f(0) = \tau$ to make both the period at $\gamma$ and the ratio 
$- f'(\rho) / g'(\rho)$ near $\gamma$ arbitrarily small.

\begin{figure}
\begin{center}
\includegraphics{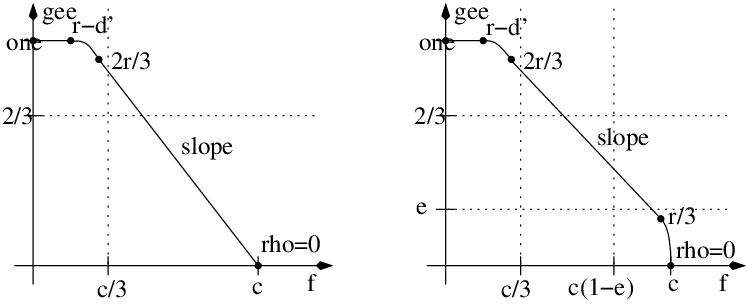}
\caption{\label{fig:smallPeriods}
The path $\rho \mapsto (f(\rho),g(\rho)) \in \RR^2$ with the extra 
conditions imposed in the proof of Lemma~\ref{lemma:dynamics} for the
nondegenerate case (left) and Morse-Bott case (right).}
\end{center}
\end{figure}

The above requires only a small modification for the neighborhood of a
torus $T \subset \iI_0 \cup \p M_0$: here we need $f$ and $g$ to extend over
$\rho \in [-r,r]$ as even and odd functions respectively, so it is no longer
possible to fix the slope $f'(\rho) / g'(\rho)$ throughout $\rho \in [0,2r/3]$.
In fact $f'(0)$ must vanish, so we amend the above conditions by allowing
them to hold for $\rho \in [r/3,r]$, but requiring the following for
$\rho \in [0,r/3]$,
\begin{itemize}
\item $- g'(\rho) / f'(\rho) \ge 1 / \tau + \epsilon_0$,
\item $f(\rho) \ge \tau (1 - \epsilon_0)$,
\item $g(\rho) \le \epsilon_0$.
\end{itemize}
This modification is shown at the right of Figure~\ref{fig:smallPeriods}.
Now for $\rho \le r/3$, the lower bound calculated in \eqref{eqn:lowerBound} 
becomes
\begin{equation*}
\begin{split}
\left| \frac{D(\rho_0)}{f'(\rho_0)} \right| &\ge
\left| f(\rho_0) \frac{g'(\rho_0)}{f'(\rho_0)} \right| - |g(\rho_0)| \ge
\tau (1 - \epsilon_0) \left( \frac{1}{\tau} + \epsilon_0 \right) - \epsilon_0 \\
&= 1 + \epsilon_0 \left( \tau - 2 - \tau \epsilon_0^2 \right) > 0.
\end{split}
\end{equation*}
Thus we can freely shrink $f(0) = \tau$, the minimal period of the
Morse-Bott family at~$T$, while bounding all other periods away from zero.

Since $X_\epsilon$ is a small perturbation of $X_0$ outside a neighborhood
of $B \cup \iI \cup \p M$, the same results immediately hold for
$X_\epsilon$: in particular, for any sequence $\epsilon_k \to 0$,
$M' \setminus \nN(B_0 \cup \iI_0 \cup \p M_0)$ cannot contain a sequence of orbits
of $X_{\epsilon_k}$ with periods below a certain threshold, as a subsequence 
of these would converge (by Arzel\`{a}-Ascoli) to an orbit of~$X_0$.
Similarly, this constraint on the periods will be satisfied by any 
sufficiently small perturbation of~$X_\epsilon$.  We can now choose such
a perturbation to a globally Morse-Bott contact form as follows:
let $\uU \subset M'$ denote a union of coordinate neighborhoods of the form
$\{ |\rho| < r_0 \}$ near each component of $B \cup \iI \cup \p M$,
where $r_0 > 0$ is chosen such that all periodic orbits inside $\uU$ are 
Morse-Bott and none exist near $\p\overline{\uU}$ (because
$f' / g'$ is irrational).
After a generic perturbation of $\lambda_\epsilon$ in 
$M' \setminus \overline{\uU}$, every Reeb orbit not fully contained in
$\overline{\uU}$ becomes nondegenerate (cf.~the appendix of
\cite{AlbersBramhamWendl}), which means all orbits outside $\overline{\uU}$
are nondegenerate, while all the others (which are inside $\uU$) are
Morse-Bott by construction.
\end{proof}

\begin{remark}
To satisfy the conditions stated in Theorem~\ref{thm:openbook}, we need
a version of Lemma~\ref{lemma:dynamics} with $\tau_1 = 1$.  This can always
be achieved by rescaling $\lambda_\epsilon$ by a constant, and thus replacing
$\hH_\epsilon = (\lambda_\epsilon,\omega)$ by $(c \lambda_\epsilon,\omega)$
for some $c > 0$.
\end{remark}

\subsubsection{A symplectic cobordism}

As a quick detour away from the proof of Theorem~\ref{thm:openbook}, we
now explain a construction that will be useful for proving
Theorem~\ref{thm:complement}.  Namely, we will need to know that the
stable Hamiltonian structures $\hH_0$ and $\hH_\epsilon$ for some
$\epsilon > 0$ can be related
to each other by a cylindrical 
symplectic cobordism that looks standard near the binding.

To simplify the statement of the following result, let us restrict to
the special case where $M = M'$ and $\pi : M \setminus B \to S^1$ is an
ordinary (not summed or blown up) open book; this will suffice for the
application we have in mind.

\begin{prop}
\label{prop:cobordism}
There exists a family of stable Hamiltonian structures 
$\hH_\epsilon = (\lambda_\epsilon,\omega)$ on~$M$ adapted to the open
book $\pi : M \setminus B \to S^1$ such that $[0,1] \times M$ admits
a symplectic form~$\Omega$ with the following properties:
\begin{itemize}
\item $\Omega = \omega + d(t\lambda_0)$ near $\{0\} \times M$.
\item $\Omega = d(e^t\lambda)$ near $\{1\} \times M$ for some contact form
$\lambda$ with $\ker\lambda = \xi_\epsilon$ and some $\epsilon > 0$.
\item $\Omega = d(\varphi(t)\lambda_0)$ on $[0,1] \times \uU$ for some
neighborhood $\uU \subset M$ of $B$ on which $\lambda_\epsilon = \lambda_0$,
and some smooth function $\varphi : [0,1] \to (0,\infty)$ with
$\varphi' > 0$.
\end{itemize}
\end{prop}
\begin{remark}
We are not claiming that $\hH_\epsilon$ in this result can be chosen to
make the periods of binding orbits small as in Lemma~\ref{lemma:dynamics}
and Theorem~\ref{thm:openbook}.  For our application we will not need this.
\end{remark}
\begin{proof}[Proof of Prop.~\ref{prop:cobordism}]
In $(\theta,\rho,\phi)$-coordinates on $\nN(B)$, we can write
$\lambda_0 = f(\rho)\ d\theta + g(\rho)\ d\phi$ with $f$ and~$g$ chosen
such that $f(\rho) = 1 - \rho$ for $\rho$ near $r - \delta'$.  Then setting
$$
F(\rho) = \begin{cases}
1 - \rho & \text{ for $\rho \in [r - \delta',r]$},\\
f(\rho)  & \text{ for $\rho \in [0,r - \delta']$}
\end{cases}
$$
and defining $\hat{\alpha}$ and $\omega$ as before, we have
$\omega \equiv d\hat{\alpha}$ where $\hat{\alpha} = \lambda_0$ on a neighborhood
$\uU := \{ \rho < r - \delta' \}$ of~$B$.

With this stipulation in place, construct the family $\lambda_\epsilon$ 
as before.  Next choose small numbers $\epsilon,\epsilon_1 > 0$ and a 
smooth function $\beta : [0,\infty) \to [0,\epsilon]$ such that
\begin{itemize}
\item $\beta(t) = 0$ for $t$ near~$0$,
\item $\beta(t) = \epsilon$ for $t \ge \epsilon_1$.
\end{itemize}
Define a $1$-form $\hat{\lambda}$ on $[0,\infty) \times M$ by
$$
\hat{\lambda}|_{(t,p)} = \lambda_{\beta(t)}|_p
$$
for all $(t,p) \in [0,\infty) \times M$, and then define
$$
\Omega = \omega + d(t \hat{\lambda})
$$
on $[0,\infty) \times M$.  Note that $\omega + d(t\lambda_0)$ is symplectic
on $[0,\epsilon_1] \times M$ if $\epsilon_1 > 0$ is sufficiently small, and
$\Omega$ is $C^\infty$-close to this if $\epsilon > 0$ is also small,
implying that $\Omega$ is also symplectic on $[0,\epsilon_1] \times M$.
It is also obviously symplectic on $[\epsilon_1,\infty) \times M$ since it then
equals
$$
\omega + d(t \lambda_\epsilon)
$$
for some $\epsilon > 0$, where $\lambda_\epsilon$ is contact and
$\omega$ is $d\lambda_\epsilon$ multiplied by a smooth positive function.
This construction thus gives a symplectic form on $[0,\infty) \times M$
which has the desired form already near $\{0\} \times M$ and 
on $[0,\infty) \times \uU$.  To define a suitable top boundary for the
cobordism, observe that $\Omega = d(\hat{\alpha} + t\hat{\lambda})$,
thus the $\Omega$-dual vector field to $\hat{\alpha} + t\hat{\lambda}$
is a Liouville vector field~$Y$:
$$
\iota_Y \Omega := \hat{\alpha} + t\hat{\lambda}.
$$
We claim that on the hypersurface $\{T\} \times M$ for $T > 0$ sufficiently
large, $dt(Y) > 0$.  Indeed, this is equivalent to the statement that
$\hat{\alpha} + t\hat{\lambda}$ defines a positive contact form on
$\{T\} \times M$, which is true if $T$ is large enough since its kernel
is then a small perturbation of $\ker\lambda_\epsilon$.
Thus fixing $T$ sufficiently large,
$\{T\} \times M$ is a convex boundary component of $[0,T] \times M$.
Moreover since the primitive of $\Omega$ is just $(1 + t)\lambda_0$
in $[\epsilon_1,\infty) \times \uU$, the vector field $Y$ takes the simple form
$(1 + t) \p_t$ in this region.  Using the flow of~$Y$ near $\{T\} \times M$,
we can now identify a neighborhood of this hypersurface in $[0,T] \times M$
symplectically with a domain of the form
$$
((1 - \epsilon_1,1] \times M, d(e^t \lambda)),
$$
where $\lambda$ is a constant multiple of the contact $1$-form
$\hat{\alpha} + T\lambda_\epsilon$, which defines a contact structure
isotopic to~$\xi_\epsilon$ due to Gray's theorem.  There is thus a
diffeomorphism of $[0,T] \times M$ to $[0,1] \times M$ that transforms
$\Omega$ into the desired form.
\end{proof}

\subsubsection{Non-generic holomorphic curves and perturbation}

Returning to the proof of Theorem~\ref{thm:openbook}, assume
$\hH_\epsilon = (\lambda_\epsilon,\omega)$ is a family of stable Hamiltonian
structures adapted to the blown up summed open 
book~$\check{\boldsymbol{\pi}}$ on $M \subset M'$ and satisfying
Lemma~\ref{lemma:dynamics}.
Choose any compatible almost complex structure
$J_0 \in \jJ(\hH_0)$ which has the following properties in
the coordinate neighborhoods $\nN(B \cup \iI \cup \p M)$:
\begin{itemize}
\item $J_0$ is invariant under the $T^2$-action defined by translating the
coordinates $(\theta,\phi)$.
\item $d\rho(J_0 \p_\rho) \equiv 0$.
\end{itemize}
Observe that $\p_\rho \in \xi_0$ always, so the second condition says
that $J_0$ maps $\p_\rho$ into the characteristic foliation defined by
$\xi_0$ on the torus $\{ \rho = \text{const} \}$.  Note
also that since $\xi_0$ is tangent to the fibers of~$P_\psi$, these fibers
naturally embed into $\RR\times M'$ as $J_0$-holomorphic curves.
The construction in \cite{Wendl:openbook}*{\S 3} now carries over
directly to the present setting and gives the following result.

\begin{prop}
\label{prop:holOBD}
For each $i=0,\ldots,N$, the interior of
$\RR \times (M_i \setminus (B_i \cup \iI_i))$ is foliated by
an $\RR$-invariant family of properly embedded surfaces
$$
\{ S^{(i)}_{\sigma,\tau} \}_{(\sigma,\tau) \in \RR \times S^1}
$$
with $J_0$-invariant tangent spaces, where
$$
S^{(i)}_{\sigma,\tau} \cap (\RR \times P_\psi) =
\{\sigma\} \times \left( \pi_i^{-1}(\tau) \cap P_\psi \right),
$$
and its intersection with each connected component of 
$\RR \times \nN(B \cup \iI \cup \p M)$ can be parametrized in
$(\theta,\rho,\phi)$-coordinates by a map of the form
$$
[0,\infty) \times S^1 \to \RR \times S^1 \times (0,r] \times S^1 :
(s,t) \mapsto (a_i(s) + \sigma,t,\rho_i(s),\tau).
$$
Here $a_i : [0,\infty) \to [0,\infty)$ is a fixed map with
$a_i(0) = 0$ and $\lim_{s \to \infty} a_i(s) = +\infty$, and 
$\rho_i : [0,\infty) \to (0,r]$ is a fixed orientation reversing
diffeomorphism.
\end{prop}

Denote by $\fF^{(i)}_0$ for $i=0,\ldots,N$ the resulting foliation on the
interior of $\RR\times (M_i \setminus (B_i \cup \iI_i))$, whose leaves can 
each be parametrized by an embedded finite energy $J_0$-holomorphic curve
$$
u^{(i)}_{\sigma,\tau} : \dot{\Sigma}_i \to \RR \times M'.
$$
The collection of all
these curves together with the trivial cylinders over their asymptotic orbits
(which include all orbits in $B \cup \iI \cup \p M$) defines a 
$J_0$-holomorphic finite energy foliation $\fF_0$ of~$M$, as defined in 
\cites{HWZ:foliations,Wendl:OTfol}.  It's
important however to be aware that this foliation is not generally 
\emph{stable}, due to the following index calculation.  From now on we
assume that $\hH_\epsilon$ has the properties specified in
Lemma~\ref{lemma:dynamics}.

\begin{prop}
\label{prop:index}
$\ind\big(u^{(i)}_{\sigma,\tau}\big) = 2 - 2 g_i$.
\end{prop}
\begin{proof}
Let $\Phi$ denote the natural trivialization of $\xi_0$ determined by the
$(\theta,\rho,\phi)$-coordinates along each of the asymptotic orbits
of~$u^{(i)}_{\sigma,\tau}$.  These orbits are in general a mix of
nondegenerate binding circles $\gamma \subset B_i$ with $\muCZ^\Phi(\gamma) = 1$
and Morse-Bott orbits that belong to $S^1$-families foliating $\iI \cup \p M$.
If $\gamma$ is one of the latter, then we observe that since 
$u^{(i)}_{\sigma,\tau}$ doesn't intersect $\RR \times (\iI \cup \p M)$,
the asymptotic winding of $u^{(i)}_{\sigma,\tau}$ as it approaches~$\gamma$
matches the winding of any nontrivial section in $\ker\mathbf{A}_\gamma$,
which is zero in the chosen coordinates.  Thus for sufficiently small
$\epsilon > 0$, the two largest negative eigenvalues of
$\mathbf{A}_\gamma - \epsilon$ both have zero winding, implying
$\alpha^\Phi_-(\gamma - \epsilon) = 0$ and
$p(\gamma - \epsilon) = 1$, hence by \eqref{eqn:CZwinding},
\begin{equation}
\label{eqn:CZMB}
\muCZ^\Phi(\gamma - \epsilon) = 2\alpha^\Phi_-(\gamma - \epsilon) +
p(\gamma - \epsilon) = 1.
\end{equation}

Since $u^{(i)}_{\sigma,\tau}$ projects to an embedding in~$M'$, it is 
everywhere transverse to the complex subspace in $T(\RR\times M')$ spanned
by $\p_t$ and $X_0$, though asymptotically $u^{(i)}_{\sigma,\tau}$ becomes
tangent to this space.  We can thus define a sensible normal bundle 
$N \to \dot{\Sigma}_i$ for $u^{(i)}_{\sigma,\tau}$ as follows: let
$X$ denote the smooth vector field on $M' \setminus (B \cup \iI \cup \p M)$
that equals $\p_\phi$ in every $(\theta,\rho,\phi)$-coordinate neighborhood
(except at $\{ \rho = 0 \}$, where this is not well defined), and $X_0$
everywhere outside of this.  Then the $J_0$-complex span of this vector field
defines a bundle that extends smoothly over $B \cup \iI \cup \p M$, and we
define the normal bundle $N \to \dot{\Sigma}_i$ to be the restriction of
this bundle to the image of~$u^{(i)}_{\sigma,\tau}$.  From this construction
it is clear that $c_1^\Phi(N) = 0$.  Now since $u^{(i)}_{\sigma,\tau}$
is embedded, its index is the index of the normal Cauchy-Riemann operator
on the bundle $N \to \dot{\Sigma}_i$, so by \eqref{eqn:normalCR},
$$
\ind\left(u^{(i)}_{\sigma,\tau}\right) = \chi(\dot{\Sigma}_i) + 2 c_1^\Phi(N) +
\sum_{\gamma} \muCZ^\Phi(\gamma - \epsilon) = \chi(\Sigma_i) = 2 - 2 g_i,
$$
where the summation is over all the asymptotic orbits of 
$u^{(i)}_{\sigma,\tau}$, whose Conley-Zehnder indices thus cancel out the
terms in $\chi(\dot{\Sigma}_i)$ resulting from the punctures.
\end{proof}

From this calculation it follows that the higher genus curves in
$\fF_0$ will vanish under generic perturbations of the data.  In contrast,
the genus zero curves have exactly the right properties to apply the
following useful perturbation result
(cf.~\cite{Wendl:thesis}*{Theorem~4.5.44}):

\begin{IFT}
Assume $M$ is any closed $3$-manifold with stable Hamiltonian structure
$\hH = (\lambda,\omega)$, $J \in \jJ(\hH)$, and
$$
u = \left(u^\RR,u^M\right) : \dot{\Sigma} \setminus \Gamma \to\RR\times M
$$
is a finite energy $J$-holomorphic curve with positive/negative punctures
$\Gamma^\pm \subset \Sigma$ and the following properties:
\begin{enumerate}
\item $u$ is embedded and asymptotic to simply covered periodic
orbits at each puncture, and satisfies $\delta_\infty(u) = 0$.
\item $\dot{\Sigma}$ has genus zero.
\item All asymptotic orbits $\gamma_z$ of $u$ for $z \in \Gamma^\pm$ 
are either nondegenerate or belong to $S^1$-parametrized Morse-Bott families
foliating tori, and
$$
p(\gamma_z \mp \epsilon) = 1
$$
for all $z \in \Gamma^\pm$ and sufficiently small $\epsilon > 0$.
\item $\ind(u) = 2$.
\end{enumerate}
Then $u$ is Fredholm regular and belongs to a
smooth $2$-parameter family of embedded curves
$$
u_{(\sigma,\tau)} = \left(u_\tau^\RR + \sigma,u_\tau^M\right) : 
\dot{\Sigma} \to \RR\times M,
\qquad (\sigma,\tau) \in \RR\times (-1,1)
$$
with $u_{(0,0)} = u$,
whose images foliate an open neighborhood of $u(\dot{\Sigma})$
in $\RR \times M$.  Moreover, the maps $u_\tau^M : \dot{\Sigma} \to M$
are all embedded and foliate an open neighborhood of $u^M(\dot{\Sigma})$
in~$M$, and if $\gamma_z^\tau$ denotes a degenerate Morse-Bott 
asymptotic orbit of $u_{(\sigma,\tau)}$ for some fixed puncture $z \in \Gamma$,
then the map $\tau \mapsto \gamma_z^\tau$ parametrizes a neighborhood of
$\gamma_z^0$ in its $S^1$-family of orbits.
\end{IFT}

Using this and a simple topological argument in \cite{Wendl:openbook},
it follows that whenever $g_i=0$, the family
$u^{(i)}_{\sigma,\tau}$ perturbs smoothly along with any sufficiently small
perturbation of~$J_0$.  In particular, picking $\epsilon > 0$ small and 
$J_\epsilon \in \jJ(\hH_\epsilon)$ close to~$J_0$, 
there is a corresponding family
of $J_\epsilon$-holomorphic curves in $\RR \times M_i$ that project to
a blown up summed open book on $M_i$ that is $C^\infty$-close to the original one.
Perturbing $\lambda_\epsilon$ a little bit further outside a suitable 
neighborhood of $B \cup \iI \cup \p M$, we can then also turn
$\lambda_\epsilon$ into a globally Morse-Bott contact form, and a
corresponding perturbation of~$J_\epsilon$ makes the latter Fredholm regular.
This proves the existence part of Theorem~\ref{thm:openbook}.  We will
continue to denote the $J_\epsilon$-holomorphic pages constructed in
this way by
$$
u^{(i)}_{\sigma,\tau} : \dot{\Sigma}_i \to \RR \times M_i,
$$
for all $i=0,\ldots,N$ with $g_i=0$.

\subsubsection{Uniqueness}

Despite their obvious instability, the higher genus curves in the
foliation $\fF_0$ are useful due to the following uniqueness result based
on intersection theory.  Here $m_0 \in \NN$ denotes the multiplicity bound
from Lemma~\ref{lemma:dynamics}, which we can assume to be arbitrarily large.

\begin{prop}
\label{prop:uniqueness}
Suppose $v : \dot{\Sigma} \to \RR \times M'$ is a somewhere injective 
finite energy $J_0$-holomorphic curve that intersects the interior of
$\RR \times M_i$ and has all its
positive ends asymptotic to orbits in $B \cup \iI \cup \p M$, where the
orbits in~$B_i$ each have covering multiplicity at most~$m_0$.
Then~$v$ parametrizes one of the surfaces~$S^{(i)}_{\sigma,\tau}$.
\end{prop}
\begin{proof}
We use the homotopy invariant intersection number $u * v \in \ZZ$ 
defined by Siefring \cite{Siefring:intersection} for asymptotically
cylindrical maps $u$ and~$v$.  If $v$ does not parametrize any
leaf of $\fF^{(i)}_0$, then its
intersection with $\RR \times M_i$ implies that it has at least one isolated
positive intersection with
some leaf $S^{(i)}_{\sigma,\tau}$ with $J_0$-holomorphic
parametrization $u^{(i)}_{\sigma,\tau}$, hence
$$
u^{(i)}_{\sigma,\tau} * v > 0.
$$
By changing $\tau$ slightly, we may assume without loss of generality
that any ends of $u^{(i)}_{\sigma,\tau}$ approaching Morse-Bott orbits in
$\iI \cup \p M$ are disjoint from the positive asymptotic orbits of~$v$.
By homotopy invariance, we can also take advantage of the lack of negative
ends for $u^{(i)}_{\sigma,\tau}$ and $\RR$-translate it until
its image lies entirely in
$[0,\infty) \times M'$.  We can likewise change $v$ by a homotopy through
asymptotically cylindrical maps so that its intersection with
$[0,\infty) \times M'$ lies entirely in the trivial cylinders over its
positive asymptotic orbits, i.e.~in $[0,\infty) \times (B \cup \iI
\cup \p M)$.  An example of this kind of homotopy is shown in
Figure~\ref{fig:OBDuniqueness}.
The intersection number above is then a sum of the form
$$
u^{(i)}_{\sigma,\tau} * v = \sum_{\gamma} u^{(i)}_{\sigma,\tau} * (\RR \times \gamma),
$$
where the summation is over some collection of orbits 
$\gamma$ in $B \cup \iI \cup \p M$, and we use $\RR\times \gamma$ as
shorthand for a $J_0$-holomorphic curve that parametrizes the trivial
cylinder over~$\gamma$.  Note that $u^{(i)}_{\sigma,\tau}$ never has an
actual intersection with $\RR \times \gamma$, so the intersections
counted by $u^{(i)}_{\sigma,\tau} * (\RR \times \gamma)$ are \emph{asymptotic},
i.e.~they are hidden intersections that could potentially emerge from
infinity under small perturbations of the data.  Since we've arranged for
$u^{(i)}_{\sigma,\tau}$ and~$v$ to have no Morse-Bott orbits in common,
the asymptotic intersections vanish except possibly for orbits 
$\gamma \subset B_i$ of covering multiplicity $m \le m_0$.
As explained in \cite{Siefring:intersection}*{\S 3.2}, each such asymptotic
intersection can be expressed in terms of the difference in the asymptotic winding of
the $m$-fold cover of the end of $u^{(i)}_{\sigma,\tau}$ about $\gamma$
from its maximum possible value, which (by standard results in
\cites{HWZ:props1,HWZ:props2}) is the winding number of the
asymptotic eigenfunction with largest negative eigenvalue.
In the natural trivialization
$\Phi$ determined by the $(\theta,\rho,\phi)$-coordinates, each
of the relevant orbits~$\gamma$ has
$\muCZ^\Phi(\gamma) = 1 = 2\alpha^\Phi_-(\gamma) + 1$, hence
$\alpha^\Phi_-(\gamma)=0$ using~\eqref{eqn:CZwinding}.  
By construction, the asymptotic winding of
$u^{(i)}_{\sigma,\tau}$ as it approaches $\gamma$ is also zero, hence this
winding is extremal, and this implies
$$
u^{(i)}_{\sigma,\tau} * (\RR \times \gamma) = 0.
$$
This is a contradiction.
\end{proof}

The above proof also works for a $J_\epsilon$-holomorphic curve
if it passes through a region that is foliated by $J_\epsilon$-holomorphic
pages.  In particular, since we've already shown this to be true in the
planar piece~$M_0$ for sufficiently small $\epsilon > 0$,
we deduce the following parallel result:

\begin{prop}
\label{prop:uniqueness2}
For all sufficiently small $\epsilon > 0$, the following holds:
if $v : \dot{\Sigma} \to \RR \times M'$ is a somewhere injective
finite energy $J_\epsilon$-holomorphic curve that intersects the interior of
$\RR \times M_0$ and has all its
positive ends asymptotic to orbits in $B \cup \iI \cup \p M$, where the
orbits in~$B_0$ have covering multiplicity at most~$m_0$, 
then~$v$ is a reparametrization of one of the 
$J_\epsilon$-holomorphic pages~$u^{(0)}_{\sigma,\tau}$.
\end{prop}

\begin{figure}
\begin{center}
\includegraphics{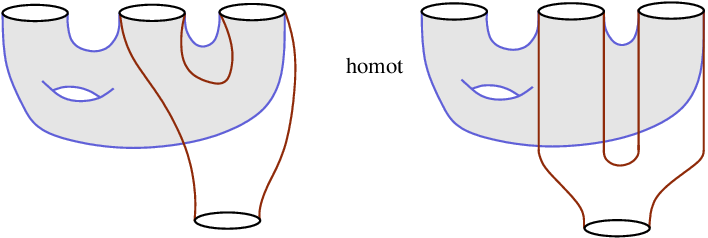}
\caption{\label{fig:OBDuniqueness}
A homotopy of two asymptotically cylindrical maps, reducing the computation
of the intersection
number to the intersection of one holomorphic
curve with the asymptotic trivial cylinders of the other.}
\end{center}
\end{figure}

We now prove the remainder of the 
uniqueness statement in Theorem~\ref{thm:openbook}.
Choose a sequence $\epsilon_k > 0$ converging to zero, denote
$\lambda_k := \lambda_{\epsilon_k}$ and $\xi_k := \ker\lambda_k$, 
and choose generic almost complex
structures $J_k \in \jJ(\hH_{\epsilon_k})$ with $J_k \to J_0$ in~$C^\infty$.
By small perturbations we can assume the forms $\lambda_k$ are all Morse-Bott
and have the properties listed in Lemma~\ref{lemma:dynamics}: in particular
the minimal periods of the orbits in $B_0 \cup \iI_0 \cup \p M_0$ are bounded by
an arbitrarily small number~$\tau > 0$, while all others are at least~$1$,
and the orbits in~$B_0$ have Conley-Zehnder index~$1$.  
We can also assume that for sufficiently large~$k$, planar 
$J_k$-holomorphic pages $u^{(i)}_{\sigma,\tau}$ in $\RR\times M_i$ 
exist whenever $g_i=0$, and hence Prop.~\ref{prop:uniqueness2} holds.
Now arguing by contradiction, suppose
that for every~$k$, there exists a finite energy $J_k$-holomorphic curve
$$
v_k : (\dot{\Sigma}_k,j_k) \to (\RR \times M',J_k)
$$
which is strongly subordinate to~$\pi_0$ and is (for large~$k$) not equivalent 
to any of the planar curves $u^{(i)}_{\sigma,\tau}$.  If $v_k$ has any
positive end asymptotic to an orbit in~$B_0$ or~$\iI_0$, then it must intersect 
the interior of $\RR\times M_0$ and Proposition~\ref{prop:uniqueness2}
already gives a contradiction.  We can therefore assume that
the positive ends of~$v_k$ approach simply covered orbits in distinct
connected components of~$\p M_0$.  This implies that they are all
somewhere injective.

\begin{lemma}
\label{lemma:compactness}
A subsequence of $v_k$ converges to one of the $J_0$-holomorphic leaves
of the foliation~$\fF_0$.
\end{lemma}
\begin{proof}
We proceed in three steps.

\textsl{Step~1: Energy bounds.}  We use the stable Hamiltonian structure
$\hH_{\epsilon_k} = (\lambda_k,\omega)$ to define the energy of~$v_k$.
To be precise, choose $c_0 > 0$ small enough so that $\omega + d(t \lambda_0)$
is symplectic on $[-c_0,c_0] \times M'$; the same is then true for all
$\omega + d(t \lambda_k)$ with $k$ sufficiently large, so following
\eqref{eqn:omegaphi} and \eqref{eqn:energy}, define
$$
E_k(v_k) = \int_{\dot{\Sigma}_k} v_k^*\omega + 
\sup_{\varphi \in \tT} \int_{\dot{\Sigma}_k} v_k^*d(\varphi\lambda_k),
$$
where $\tT = \{ \varphi \in C^\infty(\RR, (-c_0,c_0)) \ |\ \varphi' > 0 \}$.
Since $\omega$ is exact, $E_k(v_k)$ depends only
on the asymptotic behavior of~$v_k$.  Now since the positive ends all
approach simple orbits in distinct connected components of~$\p M_0$, the 
number of ends and sum of their periods are uniformly bounded,
implying a uniform bound on $E_k(v_k)$.

\textsl{Step~2: Genus bounds.}
After taking a subsequence we may assume that all the curves $v_k$ have 
the same number of positive and negative punctures.
It is still possible however that the surfaces $\dot{\Sigma}_k$ could have 
unbounded topology, i.e.~their genus could blow up as $k \to \infty$.  To 
preclude this, we apply the currents version of Gromov compactness,
see \cite{Taubes:currents}*{Prop.~3.3} or 
\cite{Hutchings:index}*{Lemma~9.9}.
The key fact is that since $E_k(v_k)$ is uniformly bounded,
$\hH_k \to \hH_0$ and $J_k \to J_0$, $v_k$ as a sequence of currents 
has a convergent subsequence, and this implies in particular that the
relative homology classes $[v_k]$ for this subsequence converge.  
We now plug this into the adjunction formula \eqref{eqn:adjunction} 
for punctured holomorphic curves, which implies
$$
v_k * v_k \ge 2\left[\delta(v_k) + \delta_\infty(v_k)\right] + c_N(v_k)
\ge c_N(v_k).
$$
Both the right and left hand
sides of this expression depend only on $[v_k]$ and on certain integer
valued winding numbers of eigenfunctions at the asymptotic orbits 
of~$v_k$.  As orbits vary in a Morse-Bott family that all have the same
minimal period, these winding numbers remain constant, thus by the
convergence of $[v_k]$, the sequence $v_k * v_k$ converges to a fixed integer,
implying an upper bound on $c_N(v_k)$ for large~$k$.
The latter can be written as $c_1^\Phi(v_k) - \chi(\dot{\Sigma}_k)$
plus more winding numbers of eigenfunctions, thus every term
other than $\chi(\dot{\Sigma}_k)$ converges, and we obtain a uniform
upper bound on $-\chi(\dot{\Sigma}_k)$, or equivalently, an upper
bound on the genus of $\dot{\Sigma}_k$.

\textsl{Step~3: SFT compactness.}
We can now assume the domains $\dot{\Sigma}_k$ are a fixed surface
$\dot{\Sigma}$, so the sequence $v_k$ with uniform energy bound
$E_k(v_k) < C$ satisfies the compactness theorem of Symplectic Field Theory
\cite{SFTcompactness}.  There is one subtle point to be careful of here:
since $X_0$ is not a Morse-Bott vector field, it is not clear at first
whether the SFT compactness theory can be applied as
$\hH_{\epsilon_k} \to \hH_0$.  What saves us is the fact that $v_k$ is
asymptotic at~$+\infty$ to orbits with arbitrarily
small period: then for energy reasons, we may assume
the only orbits that can appear
under breaking or bubbling are other orbits in $B_0 \cup \iI_0 \cup \p M_0$, 
all of which are Morse-Bott.  With this observation, the proof of SFT
compactness in \cite{SFTcompactness} goes through unchanged.  We can
thus assume that $v_k$ converges to a $J_0$-holomorphic building~$v_\infty$.
The positive asymptotic orbits of $v_\infty$ are all simply
covered and lie in distinct connected components of $\p M_0$,
thus the top level of $v_\infty$ contains at least one somewhere injective
curve $v_+$ that is strongly subordinate to~$\pi_0$.  
Then Prop.~\ref{prop:uniqueness}
implies that~$v_+$ parametrizes a leaf of the foliation $\fF_0$, so it
has no negative ends.  The same is true for every other top level
component of~$v_\infty$ unless it is a trivial cylinder, and nontrivial
curves must all be distinct since they approach distinct orbits at their
positive ends.  It follows that they do not intersect each other, so there 
is no possibility of nodes connecting them, and the building must be
disconnected unless it consists of only a single component, namely~$v_+$.
\end{proof}

We are now just about done with the proof of Theorem~\ref{thm:openbook}: 
the implicit function theorem implies that
if the limit $v_\infty = \lim v_k$ has genus zero, then $v_k$ is
always one of the $J_k$-holomorphic pages $u^{(i)}_{\sigma,\tau}$
for sufficiently large~$k$.  If on the other hand $v_\infty$ 
has genus $g > 0$, then
$\ind(v_k) = \ind(v_\infty) = 2 - 2g \le 0$ by Prop.~\ref{prop:index}, 
yet $v_k$ must be Fredholm regular since $J_k$ was chosen generically,
and this gives a contradiction.

\subsection{Deformation and compactness}
\label{sec:compactness}

We now prove a compactness result for families of holomorphic curves in
symplectic manifolds that emerge from the holomorphic pages provided
by Theorem~\ref{thm:openbook}.

We recall first that every strong
symplectic filling can be \emph{completed} by attaching a cylindrical end.
To be precise, assume $(M',\xi)$ is a closed, connected 
contact $3$-manifold with 
positive contact form $\alpha$, and for any two smooth functions 
$f, g : M' \to [-\infty,\infty]$ with $g > f$, define a subdomain of the
symplectization $(\RR \times M', d(e^t \alpha))$ by
\begin{equation}
\label{eqn:cylStein}
\sS_f^g = \{ (t,m) \in \RR \times M' \ |\ f(m) \le t \le g(m)\ \}.
\end{equation}
Here we include the cases $f \equiv -\infty$ and $g \equiv +\infty$
so that $\sS_f^g$ may be unbounded.  Now suppose $M' = \p W$, where
$(W,\omega)$ is a (not necessarily compact) symplectic manifold with
contact type boundary, and $\lambda$ is a primitive of $\omega$ defined
near $\p W$ such that
$\lambda|_{TM'} = e^f \alpha$ for some smooth function $f : M' \to \RR$.
Then using the flow of the Liouville vector field $Y$ defined by
$\iota_Y\omega = \lambda$, one can identify a neighborhood of~$M'$ in
$(W,\omega)$ symplectically with a neighborhood of $\p\sS_{-\infty}^f$
in $(\sS_{-\infty}^f,d(e^t\alpha))$.  As a consequence, one can 
symplectically glue the cylindrical end $(\sS_f^\infty,d(e^t\alpha))$ to
$(W,\omega)$ along~$M'$, giving a noncompact symplectic
manifold
$$
(W^\infty,\omega) := (W,\omega) \cup_{M'} (\sS_f^\infty,d(e^t\alpha)),
$$
which necessarily contains the half-symplectization $([T,\infty) \times M',
d(e^t\alpha))$ whenever $T \in \RR$ is sufficiently large.

Adopting the notation from the setup for Theorem~\ref{thm:openbook}, 
assume now that in addition to the above,
$(M',\xi)$ contains a partially planar domain $M \subset M'$ with irreducible
subdomains $M = M_0 \cup \ldots \cup M_N$ for $N \ge 0$, 
of which $M_0$ is a planar piece lying in the interior of~$M$.
By Theorem~\ref{thm:openbook}, we can then find
a Morse-Bott contact form $\alpha$ on~$M'$ and generic compatible almost 
complex structure~$J_+$ such that the planar pages in $M_0$ lift to an
$\RR$-invariant foliation 
by properly embedded $J_+$-holomorphic curves in $\RR\times M'$,
whose asymptotic orbits are simply covered and have
minimal period less than an arbitrarily small number
$\tau_0 > 0$, while all other closed orbits of $X_{\alpha}$ in~$M'$ have
period at least~$1$.  
Assume that $\alpha$ is the contact
form chosen for defining the symplectic cylindrical end in $(W^\infty,\omega)$.

Choose an almost complex structure $J$ on
$W^\infty$ which is compatible with~$\omega$, generic on~$W \subset W^\infty$ 
and matches $J_+$ on $\sS_f^\infty \subset W^\infty$.  
Then every leaf of the $J_+$-holomorphic foliation in $\RR\times M_0$
has an $\RR$-translation
that can be regarded as a properly embedded surface in $\sS_f^\infty \subset
W^\infty$ parametrized by a finite energy $J$-holomorphic curve.
The main idea used for the proofs in \S\ref{sec:non-fillable} is to show that 
these curves generate a moduli space of $J$-holomorphic curves that
must fill the entirety of~$W^\infty$, and leads to a contradiction in
any of the situations considered by Theorems~\ref{thm:non-fillable},
\ref{thm:complement} and~\ref{thm:nonseparating}.  
To prove this, we need a deformation result
and a corresponding compactness result to show that the region filled by 
these curves
is open and closed respectively.  We shall prove somewhat more general
versions of these results than are immediately needed, as they are
also useful for other applications (e.g.~in 
\cites{NiederkruegerWendl,LisiVanhornWendl}).

We now generalize the above setup as follows: let $u_+ :
\dot{\Sigma} \to W^\infty$
denote one of the $J$-holomorphic planar pages living in the cylindrical
end of $(W^\infty,\omega)$, and pick any open neighborhood
$\uU \subset M'$ and $T > 0$ such that
$$
u_+(\dot{\Sigma}) \subset [T,\infty) \times \uU.
$$
Choose any data $(\alpha',\omega',J')$ with the following properties:
\begin{itemize} 
\item $\alpha'$ is a Morse-Bott contact form on~$M'$ that matches
$\alpha$ on $\uU \cup \nN(B_0 \cup \iI_0 \cup \p M_0)$ and has only
Reeb orbits of period at least~$1$ outside of $\nN(B_0 \cup \iI_0 \cup \p M_0)$
\item $\omega'$ is a sympectic form on~$W^\infty$ that matches
$d(e^t \alpha')$ on $\sS_f^\infty$
\item $J'$ is an $\omega'$-compatible almost complex structure on $W^\infty$ 
that has an $\RR$-invariant restriction 
$$
J_+' := J'|_{\sS_f^\infty}
$$
that is generic and compatible with~$\alpha'$ and matches $J_+$ on 
$\RR \times (\uU \cup \nN(B_0 \cup \iI_0 \cup \p M_0))$, and
$J'$ is generic on~$W$.
\end{itemize}

The advantage of this generalization is that fairly arbitrary changes
to the data can be accommodated outside a neighborhood of a single page,
which is useful for instance in the adaptation of these 
arguments for weak fillings (cf.~\cite{NiederkruegerWendl}).
Let $\mM^*(J')$ denote the moduli space of all unparametrized
somewhere injective finite energy $J'$-holomorphic
curves in $W^\infty$, which is non-empty by construction since it 
contains~$u_+$, and define
$$
\mM_0^*(J') \subset \mM^*(J')
$$
to be the connected component of this space containing~$u_+$.  The curves
$u \in \mM_0^*(J')$ share all homotopy invariant properties of the
planar $J_+$-holomorphic pages in $\RR\times M'$, in particular:
\begin{enumerate}
\item $\ind(u) = 2$,
\item $u * u = \delta(u) + \delta_\infty(u) = 0$.
\end{enumerate}
It follows that all curves in $\mM_0^*(J')$ are embedded. 
This situation is a slight variation on the setup that was considered
in \cite{AlbersBramhamWendl}*{\S 4}, only with the added complication
that curves in $\mM_0^*(J')$ may have two ends approaching the same
Morse-Bott Reeb orbit, which presents the danger of degeneration to
holomorphic buildings with multiply covered components.  
The required deformation result is however
exactly the same: it depends on the fact that a neighborhood
of each embedded curve $u \in \mM_0^*(J')$ can be described by 
sections of its normal bundle which are nowhere vanishing, because they
satisfy a Cauchy-Riemann type equation and have vanishing first Chern 
number with respect to certain special trivializations at the ends.  

\begin{prop}[\cite{AlbersBramhamWendl}*{Theorem~4.7}]
\label{prop:IFT}
The moduli space $\mM_0^*(J')$ is a smooth $2$-dimensional manifold containing
only proper embeddings that never intersect each other:
in particular they foliate an open subset of~$W^\infty$.
\end{prop}

The compactness result we need is a variation on
\cite{AlbersBramhamWendl}*{Theorem~4.8}, but somewhat more complicated
due to the appearance of multiple covers.  For the statement of the result,
recall that the compactification in \cite{SFTcompactness} for the space
of finite energy holomorphic curves in an almost complex manifold with
cylindrical ends consists of so-called \emph{stable holomorphic buildings}, 
which have one \emph{main level} and
potentially multiple \emph{upper} and \emph{lower levels}, 
each of which is a (perhaps disconnected) nodal holomorphic curve.  We will
be considering sequences of curves in $W^\infty$ that stay within a
bounded distance of the positive end, so there will be no lower levels
in the limit.  We shall use the
term ``smooth holomorphic curve'' to mean a holomorphic building with only
one level and no nodes.  The following variation on
Definition~\ref{defn:subordinate} will be convenient.

\begin{defn}
\label{defn:subordinate2}
A $J'$-holomorphic curve $u : \dot{\Sigma} \to W^\infty$ will be called
\defin{subordinate to~$\pi_0$} if it has only positive ends, all of
which approach Reeb orbits in $B_0 \cup \iI_0 \cup \p M_0$, with total
multiplicity at most~$1$ for each connected component of
$B_0 \cup \p M_0$ and at most~$2$ for each connected component of~$\iI_0$.
\end{defn}

Observe that all the curves in $\mM_0^*(J')$ are subordinate to~$\pi_0$.
The intersection argument in the proof of Prop.~\ref{prop:uniqueness}
now implies:
\begin{lemma}
\label{lemma:intersection2}
If $u \in \mM^*(J)$ is subordinate to~$\pi_0$, then $u * u_+ = 0$.
\end{lemma}

\begin{thm}
\label{thm:compactness}
Choose an open subset $W_0 \subset W$ that contains $\p W$ and has compact
closure, and let $W_0^\infty = W_0 \cup_{M'} \sS_f^\infty$.  Then there is
a finite set of index~$0$ curves 
$\Theta(W_0) \subset \mM^*(J')$ subordinate to~$\pi_0$ and 
with images in $\overline{W}_0^\infty$ such that the following holds.
Any sequence of
curves $u_k \in \mM_0^*(J')$ with images in $\overline{W}_0^\infty$
has a subsequence convergent (in the sense of \cite{SFTcompactness}) 
to one of the following:
\begin{enumerate}
\item A curve in $\mM_0^*(J')$
\item A holomorphic building with empty main level and one nontrivial
upper level consisting of a single connected curve that can be identified
(up to $\RR$-translation) with a curve in $\mM_0^*(J')$ with image
in~$\sS_f^\infty$
\item A $J'$-holomorphic building whose upper levels contain only 
covers of trivial cylinders, and 
main level consists of a connected double cover of a curve
in~$\Theta(W_0)$
\item A $J'$-holomorphic building whose upper levels contain only 
covers of trivial cylinders, and 
main level contains at most two connected components, which are 
curves in~$\Theta(W_0)$.
\end{enumerate}
\end{thm}
\begin{proof}
Assume $u_k$ is a sequence of either index~$2$ curves in $\mM_0^*(J')$ or
index~$0$ curves subordinate to~$\pi_0$ with images in $\overline{W}_0^\infty$
and only simply covered asymptotic orbits.
By \cite{SFTcompactness}, $u_k$ has a subsequence converging to a 
stable $J'$-holomorphic building~$u_\infty$.  The main idea is to add up the indices
of all the connected components of~$u_\infty$ and use genericity to derive
restrictions on the configuration of~$u_\infty$.
To facilitate this, we introduce a variation on the
usual Fredholm index formula \eqref{eqn:index}: for any finite energy
holomorphic curve~$v : \dot{\Sigma} \to \RR \times M'$ with positive
and negative asymptotic orbits $\{ \gamma_z\}_{z \in \Gamma^\pm}$,
choose a small number $\epsilon > 0$ and trivializations~$\Phi$ of the 
contact bundle along each~$\gamma_z$ 
and define the \emph{constrained} index
$$
\widehat{\ind}(v) = -\chi(\dot{\Sigma}) + 2 c_1^\Phi(v) + 
\sum_{z \in \Gamma^+}
\muCZ^\Phi(\gamma_z - \epsilon) - \sum_{z \in \Gamma^-}
\muCZ^\Phi(\gamma_z - \epsilon).
$$
The only difference here from \eqref{eqn:index} is that at the negative
punctures we take $\muCZ^\Phi(\gamma_z - \epsilon)$ instead of
$\muCZ^\Phi(\gamma_z + \epsilon)$, which geometrically means we compute
the virtual dimension of a space of curves whose negative ends have all
their Morse-Bott orbits fixed in place.  So for curves without negative
ends $\widehat{\ind}(v) = \ind(v)$, and the constrained index otherwise
has the advantage of being \emph{additive} across levels, i.e.~if the
building~$u_\infty$ has no nodes, then we obtain $\ind(u_k) = \ind(u_\infty)$ if the
latter is defined as the sum of the constrained indices for all its
connected components.  Observe that trivial cylinders over Reeb orbits always
have constrained index~$0$.  If~$u_\infty$ does have nodes, the formula remains true
after adding~$2$ for each node in the building, so we then take this 
as a definition of the index for a nodal curve or nodal holomorphic building.
We now proceed in several steps.

\textsl{Step~1: Curves in upper levels.}  We claim that every connected
component of~$u_\infty$ either has no negative ends or is a
cover of a trivial cylinder (in an upper level).  Indeed, curves in the
main level obviously have no negative ends, and if $v$ is an upper level
component with negative ends, the smallness of the periods in 
$B_0 \cup \iI_0 \cup \p M_0$
constrains these to approach other orbits in $B_0 \cup \iI_0 \cup \p M_0$,
as otherwise~$v$ would have negative energy.  Then if~$v$ does not cover
a trivial cylinder, an intersection argument carried out in
\cite{AlbersBramhamWendl}*{Proof of Theorem~4.8} implies that $v$
must intersect~$u_+$, contradicting Lemma~\ref{lemma:intersection2} above.
The key idea here is to consider the asymptotic winding numbers that
control holomorphic curves approaching orbits at $B_0 \cup \iI_0 \cup \p M_0$,
which differ for positive and negative ends at each of these orbits, and
thus force~$v$ to intersect~$u_+$ in the projection to~$M'$.
We refer to \cite{AlbersBramhamWendl} for the details; note that a
similar argument has also appeared in \cite{Momin:thesis}.

\textsl{Step~2: Indices of connectors.}
Borrowing some terminology from Embedded Contact Homology, we refer to
branched multiple covers of trivial cylinders as \defin{connectors}.
These can appear in the upper levels of~$u_\infty$, but can never have any
curves above them except for further covers of trivial cylinders, 
due to Step~1.  Since the positive ends
of~$u_\infty$ approach any given orbit in $B_0 \cup \iI_0 \cup \p M_0$ 
with total multiplicity at most~$2$, only the following
types of connectors can appear, both with genus zero:
\begin{itemize}
\item \emph{Pair-of-pants} connectors: these have one positive end at a
doubly covered orbit and two negative ends at the same simply covered orbit.
\item \emph{Inverted pair-of-pants} connectors: with two positive ends
at the same simply covered orbit and one negative end at its double cover.
\end{itemize}
The second variety will be especially important, and we'll refer to it
for short as an \defin{inverted connector}.
As we computed in \eqref{eqn:CZMB}, all of the simply covered Morse-Bott
orbits under consideration have $\muCZ^\Phi(\gamma - \epsilon) = 1$ in
the natural trivialization, and
in fact exactly the same argument produces the same result for
their multiple covers.  We thus find that the constrained
Fredholm index is~$0$ for a pair-of-pants connector and~$2$ for the inverted
variant.

\textsl{Step~3: Indices of multiple covers.}
Suppose~$v$ is a connected component of~$u_\infty$ which is not a cover of a
trivial cylinder: then it has no negative ends, and all its positive ends
must approach orbits in $B_0 \cup \iI_0 \cup \p M_0$ with total multiplicity
at most~$2$.  Thus if~$v$ is a
$k$-fold cover of a somewhere injective curve~$v'$, we have $k \in \{1,2\}$,
and all the asymptotic orbits of both $v$ and~$v'$ have
$\muCZ^\Phi(\gamma - \epsilon) = 1$ in the natural trivialization.
Assume $k=2$, and label the positive punctures of~$v$ as 
$\Gamma = \Gamma_1 \cup \Gamma_2$,
where a puncture is defined to belong to $\Gamma_2$ if its asymptotic orbit is
doubly covered, and $\Gamma_1$ otherwise.  For $i=1,2$, let $\Gamma'_i$
denote the punctures of~$v'$ that are covered by~$\Gamma_i$, so the set of
all punctures $\Gamma'$ of $v'$ is $\Gamma'_1 \cup \Gamma'_2$.
Note that in this situation all the asymptotic orbits of $v$ must have 
total multiplicity exactly~$2$, which implies that all asymptotic orbits of
$v'$ are distinct and simply covered, and we have
$\#\Gamma_2 = \#\Gamma_2'$ and $\#\Gamma_1 = 2\#\Gamma_1'$.  Both domains
must also have genus zero, so we have
\begin{equation*}
\begin{split}
\ind(v) &= -(2 - \#\Gamma) + 2 c_1^\Phi(v) + \#\Gamma =
-2 + 2(\#\Gamma_2' + 2\#\Gamma_1') + 2k c_1^\Phi(v'), \\
\ind(v') &= -(2 - \#\Gamma') + 2 c_1^\Phi(v') + \#\Gamma'
= -2 + 2(\#\Gamma_2' + \#\Gamma_1') + 2 c_1^\Phi(v'),
\end{split}
\end{equation*}
hence
\begin{equation}
\label{eqn:indexCover}
\ind(v) = k \ind(v') + 2(k-1) (1 - \#\Gamma_2).
\end{equation}
This formula also trivially holds if $k=1$.  This gives a
lower bound on $\ind(v)$ since $\ind(v')$ is bounded from below by
either~$1$ (in $\RR\times M'$) or~$0$ (in $W^\infty$) due to genericity.
Now observe that whenever $\Gamma_2$ is non-empty, the doubly covered
orbit must connect~$v$ to an inverted connector, whose
constrained index is~$2$, so for $k=2$ we have
\begin{equation}
\label{eqn:Big!}
\ind(v) + \sum_C \widehat{\ind}(C) = k\ind(v') + 2(k-1) \ge 2,
\end{equation}
where the sum is over all inverted connectors that connect
to~$v$ along doubly covered breaking orbits.

\textsl{Step~4: Indices of bubbles.}
There may also be closed components in the main level of~$u_\infty$: these are 
$J'$-holomorphic spheres~$v$ which are either constant (\emph{ghost bubbles})
or are $k$-fold covers of somewhere injective spheres~$v'$ for some
$k \in \NN$.  In the latter case, \eqref{eqn:indexCover} also holds
with $\#\Gamma_2=0$, implying $\ind(v) \ge 0$, and the inequality is
strict whenever $k > 1$.  

If $v$ is a ghost bubble, then $\ind(v) = -2$,
but then the stability condition implies the existence of at least three
nodes connecting~$v$ to other components; let us refer to nodes of this
type as \emph{ghost nodes}.  There is then a graph with vertices representing
the ghost bubbles in~$u_\infty$ 
and edges representing the ghost nodes that connect
two ghost bubbles together, and since~$u_\infty$ 
has arithmetic genus zero, every
connected component of this graph is a tree.  Let~$G$ denote such a connected
component, with $V$ vertices and $E_i$ edges, which therefore satisfy
$V - E_i = 1$, and suppose there are also $E_e$ nodes connecting the ghost
bubbles represented by~$G$ to nonconstant components; we can think of these
as represented by ``external'' edges in~$G$.  By the stability condition,
we have
$$
2E_i + E_e \ge 3V,
$$
which after replacing $E_i$ by $V - 1$, becomes $E_e - 2\ge V$.  Then
the total contribution to $\ind(u_\infty)$ from all the ghost bubbles and
ghost nodes represented by~$G$ is
\begin{equation}
\label{eqn:ghostBubbles}
\begin{split}
-2 V + 2 (E_i + E_e) &= \left[-2V + (2E_i + E_e)\right] + 
E_e \ge V + (2 + V)\\
&= 2V + 2 \ge 4,
\end{split}
\end{equation}
unless $u_\infty$ has no ghost bubbles at all.

\textsl{Step~5: The total index of~$u_\infty$.}
We can now break down $\ind(u_\infty) \in \{0,2\}$ into a sum of nonnegative terms and
use this to rule out most possibilities.  Ghost bubbles are excluded
immediately due to \eqref{eqn:ghostBubbles}.
Similarly, there cannot be any multiply covered
bubbles, because these imply the existence of at least one node and thus
contribute at least $4$ to $\ind(u_\infty)$.  The only remaining possibility
for multiple covers (aside from connectors) is a component with only
positive ends, whose index together with contributions from attached inverted
connectors is given by \eqref{eqn:Big!} and is thus already at least~$2$.  
In fact, if this component exists in an upper level, then the underlying
simple curve must have index at least~$1$, implying an even larger lower
bound in \eqref{eqn:Big!} and hence a contradiction.  The remaining
possibility, which occurs in the case $\ind(u_\infty)=2$, 
is therefore that the main level consists only of 
a connected double cover, and there are no nodes at all,
nor anything other than trivial cylinders and connectors in the upper
levels (Figure~\ref{fig:bubbly1}).  
The underlying simple curve in the main level has index~$0$
and has only simply covered asymptotic orbits, all in separate connected 
components of~$B_0 \cup \iI_0 \cup \p M_0$, thus it is subordinate to~$\pi_0$.

\begin{figure}
\begin{center}
\includegraphics{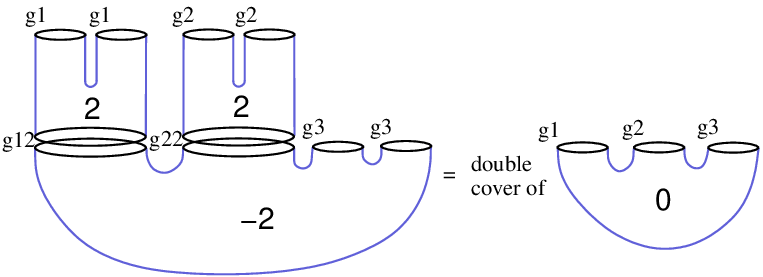}
\caption{\label{fig:bubbly1}
The limit building $u_\infty$ in a case where all asymptotic orbits have total
multiplicity two, so the main level may be a double cover of an index~$0$ curve,
while the upper level includes connectors and trivial cylinders (the latter not
shown in the picture).  The numbers inside each component
indicate the constrained index.}
\end{center}
\end{figure}

Assume now that~$u_\infty$ contains no multiply covered components except possibly
for connectors.  If there is an upper level component~$v$ that is not a
cover of a trivial cylinder, then genericty implies $\ind(v) \ge 1$,
and in fact the index must also be even since all the asymptotic orbits
satisfy $\muCZ^\Phi(\gamma - \epsilon) = 1$.  Then $\ind(u_\infty) = \ind(v)=2$ 
and there are no nodes or inverted connectors; the latter implies that all 
positive asymptotic orbits of~$v$ must be simply covered.
Then there also cannot
be any doubly covered breaking orbits, leaving only the possibility that~$v$
is the only nontrivial component in~$u_\infty$.

Next assume there are only covers of trivial cylinders in the upper levels,
in which case the main level is necessarily non-empty.  Each component
in the main level has a nonnegative even index, so there can be
at most one node or one inverted connector in~$u_\infty$, and only if
$\ind(u_\infty)=2$.  If the main level contains a component~$v$ of index~$2$,
then there are no nodes or inverted connectors.  The latter precludes
doubly covered breaking orbits, thus there are no connectors at all,
and since~$v$ cannot have negative ends, we conclude that $u_\infty=v$
(Figure~\ref{fig:bubbly1a}).
Otherwise all main level components in $u_\infty$ have index~$0$ and
are subordinate to~$\pi_0$.  Examples of the possible configurations
are shown in Figures~\ref{fig:bubbly2}--\ref{fig:bubbly5}.

\begin{figure}
\begin{minipage}[t]{5in}
\begin{center}
\includegraphics{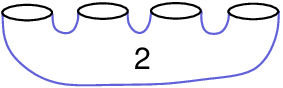}\par
\caption{\label{fig:bubbly1a} The limit building $u_\infty$ in the
simplest case, a smooth index~$2$ curve in the main level.}
\end{center}
\end{minipage}\par
\medskip
\begin{minipage}[t]{2.4in}
\begin{center}
\includegraphics{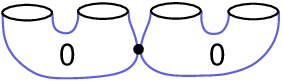}\par
\caption{\label{fig:bubbly2} A nodal curve with two index~$0$ components.}
\end{center}
\end{minipage}
\hfill
\begin{minipage}[t]{2.4in}
\begin{center}
\includegraphics{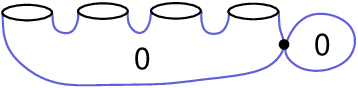}\par
\caption{\label{fig:bubbly3} Another nodal curve, including a bubble.}
\end{center}
\end{minipage}\par
\medskip
\begin{minipage}[t]{2.4in}
\begin{center}
\includegraphics{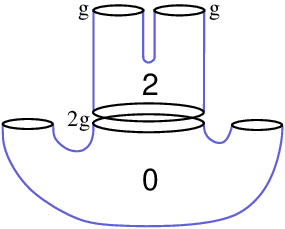}\par
\caption{\label{fig:bubbly4} An inverted connector can appear in an
upper level.}
\end{center}
\end{minipage}
\hfill
\begin{minipage}[t]{2.4in}
\begin{center}
\includegraphics{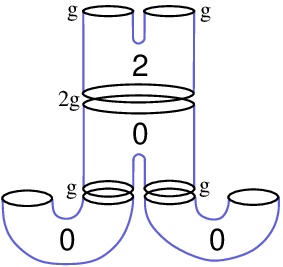}\par
\caption{\label{fig:bubbly5} Both types of connectors can appear.}
\end{center}
\end{minipage}
\end{figure}

\textsl{Step~6: Compactness for index~$0$ curves.}
If $\ind(u_\infty)=2$, then the somewhere injective index~$0$ curves
that can appear in the building~$u_\infty$ are all subordinate to~$\pi_0$
and come in two types:
\begin{itemize}
\item Type~1: Curves with only simply covered asymptotic orbits.
\item Type~2: Curves with exactly one doubly covered asymptotic orbit and
all others simply covered, and satisfying $v * v = 0$.
\end{itemize}
Indeed, the second type can occur as the unique main level curve in~$u_\infty$ 
if there is a single inverted
connector in an upper level, attached along the doubly covered orbit
(Figure~\ref{fig:bubbly4}).
To see that $v * v=0$ for such a curve, we use the continuity of the
intersection number under convergence to buildings, and the fact that
$u_k * u_k = 0$ since $u_k \in \mM_0^*(J')$; a computation shows that 
the contribution to $u_\infty * u_\infty$
from trivial cylinders and connectors in the upper level
plus breaking orbits adds up to~$0$.
The index counting argument of the previous steps shows already that the 
curves of Type~1 form a compact and hence finite set.
To finish the proof, we must show that the same is true for the
Type~2 curves.

Suppose $v_k$ is a sequence of Type~2 curves converging to a
holomorphic building~$v_\infty$.  Applying the index counting argument from
the previous steps, $v_\infty$ cannot contain any nodes or inverted
connectors; the worst case scenario is that the upper levels contain
only trivial cylinders and a single pair-of-pants connector, whose two
negative ends connect to two main level components~$v_-^1$ and $v_-^2$ 
that are both Type~1 curves (Figure~\ref{fig:bubbly6}).  
Since there are finitely many Type~1 curves, we may
assume by genericity of~$J'$ that no two of them approach a common orbit
in the Morse-Bott families~$\iI_0$, but this must be the case for
$v_-^1$ and~$v_-^2$ as they are both attached to a connector over an
orbit in~$\iI_0$, so we conclude that both are the same curve, which
we'll call~$v_-$.  We can rule out this scenario by computing the
self-intersection number $v_\infty * v_\infty$, which must a priori equal
$v_k * v_k = 0$.  Once more the connectors, trivial cylinders and breaking
orbits contribute zero in total, so since the main level includes two
copies of~$v_-$, we deduce
$$
0 = v_\infty * v_\infty = 4(v_- * v_-).
$$
But we can also compute $v_- * v_-$ directly from the adjunction formula
\eqref{eqn:adjunction}; indeed,
$$
v_- * v_- = 2\left[\delta(v_-) + \delta_\infty(v_-)\right] + c_N(v_-),
$$
where we've dropped the last term in \eqref{eqn:adjunction} since all
the asymptotic orbits are simple.  The constrained normal Chern number
$c_N(v_-)$ is defined in \eqref{eqn:cN} and can be deduced from the 
fact that $\ind(v_-) = 0$: since all of the relevant orbits
satisfy $\muCZ^\Phi(\gamma-\epsilon)=1$ and $\alpha^\Phi_-(\gamma+\epsilon)=0$,
we find $2c_1^\Phi(v_-) = \ind(v_-) + \chi(\dot{\Sigma}) - \sum_{z\in\Gamma}
\muCZ^\Phi(\gamma_z - \epsilon) = 2 - 2\#\Gamma$, hence
$$
c_N(v_-) = c_1^\Phi(v_-) - \chi(\dot{\Sigma}) +
\sum_{z \in \Gamma^+} \alpha^\Phi_-(\gamma_z + \epsilon) =
1 - \#\Gamma - (2 - \#\Gamma) = -1.
$$
This implies that $v_- * v_-$ is odd, and is thus a contradiction.
\begin{figure}
\begin{center}
\includegraphics{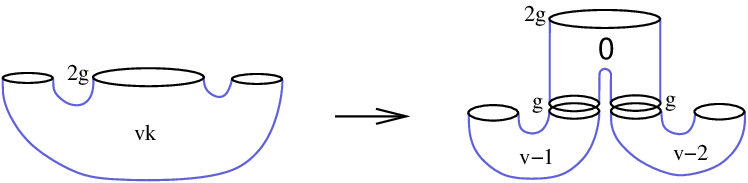}
\caption{\label{fig:bubbly6}
A possible limit of the sequence~$v_k$.}
\end{center}
\end{figure}
\end{proof}

\section{Proofs of the main results}
\label{sec:proofs}

\subsection{Non-fillability}
\label{sec:non-fillable}

We are now in a position to prove the main results on symplectic fillings.

\begin{proof}[Proof of Theorem~\ref{thm:nonseparating} and Corollary~\ref{cor:semifillings}]
Given Proposition~\ref{prop:IFT} (implicit function theorem) 
and Theorem~\ref{thm:compactness} (compactness) above,
the result follows from the same argument as in \cite{AlbersBramhamWendl}.
For completeness let us briefly recall the main idea: if $(M,\xi)$ is a
closed contact $3$-manifold which embeds as
a non-separating contact type hypersurface into some closed symplectic
$4$-manifold $(W,\omega)$, then by cutting $W$ open along~$M$ and gluing
together an infinite chain of copies of the resulting symplectic cobordism
between $(M,\xi)$ and itself, we obtain a noncompact but geometrically 
bounded symplectic manifold  $(\wW,\omega)$ with contact type boundary 
$(M,\xi)$.  Attaching a cylindrical end and considering the moduli space
$\mM_0(J)$ that arises from a partially planar domain, one can use the
monotonicity lemma to prevent the curves in $\mM_0(J)$ from escaping beyond
a compact subset of $\wW$, thus the compactness result 
Theorem~\ref{thm:compactness} applies.  In combination with
Prop.~\ref{prop:IFT}, this implies that outside a subset of codimension~2
(the images of finitely many curves from Theorem~\ref{thm:compactness}),
the set of all points in $\wW$ filled by curves in $\mM_0(J)$ must be 
open and closed, and is therefore
everything; since those curves are confined to a compact subset, this
implies $\wW$ is compact and is thus a contradiction.

By a similar argument one can prove Corollary~\ref{cor:semifillings}
independently of Theorem~\ref{thm:nonseparating}, for if $(W,\omega)$ is
a strong filling with at least two boundary components $(M,\xi)$ and
$(M',\xi')$, then the curves in $\mM_0(J)$ emerging from the cylindrical
end at $M$ will foliate~$W^\infty$ except at a subset of codimension~2;
yet they cannot enter the cylindrical end at $M'$ due to convexity, and
this is again a contradiction.
\end{proof}

\begin{proof}[Proof of Theorem~\ref{thm:non-fillable}]
Assume $(W,\omega)$ is a strong filling of $(M,\xi)$ and the partially planar
domain $M_0 \subset M$ is a planar torsion domain.  It therefore has a planar
piece $M_0^P \subset M_0$, which is a proper subset of its interior.
Combining Prop.~\ref{prop:IFT} (implicit function theorem) and 
Theorem~\ref{thm:compactness} (compactness)
as in the proof of Theorem~\ref{thm:nonseparating} above, the curves
in $\mM_0(J)$ that emerge from $M_0^P$ in the cylindrical end of 
$W^\infty$ form a
foliation of~$W^\infty$ outside a subset of codimension~2.  We can
therefore pick a point $p \in M \setminus M_0^P$ and find a sequence 
of curves $u_k \in \mM_0(J)$ for $k \to \infty$ whose images contain
$(k,p) \in [T,\infty) \times M \subset W^\infty$.  Applying
Theorem~\ref{thm:compactness} again, these have a subsequence which
converges to a
$J_+$-holomorphic curve $u'$ in $\RR\times M$, whose asymptotic orbits
are in the same Morse-Bott families as the curves in $\mM_0(J)$.  
The uniqueness statement in the holomorphic open book result
(Theorem~\ref{thm:openbook}) then implies that $u'$ is a lift of a page
in the blown up summed open book on~$M_0$, which proves that $M_0 = M$, and
$M_0 \setminus M_0^P$ consists of a single family of pages diffeomorphic
to the planar pages in $M_0^P$ and approaching the same Reeb orbits at
their boundaries.  In other words, $M_0$ is a symmetric summed open
book, which contradicts the definition of a planar torsion domain.
\end{proof}

\begin{proof}[Proof of Theorem~\ref{thm:complement}]
The idea is much the same as in the proof of Theorem~\ref{thm:non-fillable},
but instead of working in the compact context of a symplectic filling,
we work in a noncompact symplectic cobordism diffeomorphic to $\RR\times M$,
in which the negative end is ``walled off'' so that curves in
$\mM_0(J)$ cannot reach it.  This wall is created by a family of
holomorphic curves, namely a subset of the generally non-generic family
arising from an open book decomposition 
(see Figure~\ref{fig:torsionCobordism}).

\begin{figure}
\begin{center}
\includegraphics{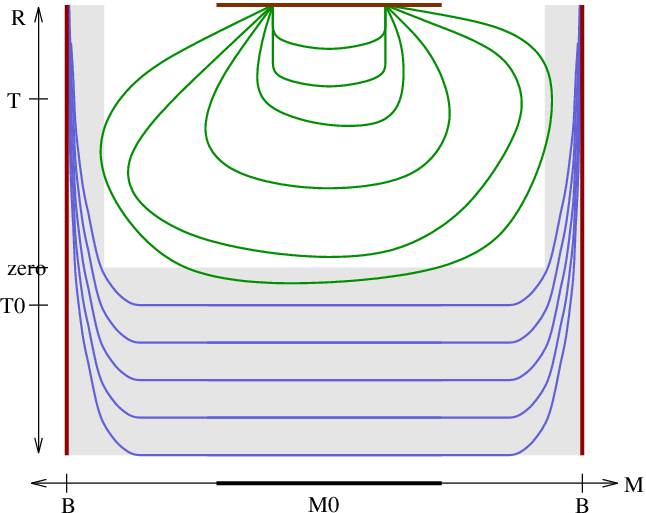}
\caption{\label{fig:torsionCobordism}
The symplectic cobordism used in the proof of Theorem~\ref{thm:complement},
with the negative end ``walled off'' by holomorphic pages of an open book.
The almost complex structure in the shaded region is a non-generic
one for which holomorphic open books always exist.}
\end{center}
\end{figure}

Specifically, suppose $\pi : M\setminus B \to S^1$ is an open book 
decomposition.  Recall from Prop.~\ref{prop:cobordism} that there is a
symplectic cobordism $(W,\Omega) = ([0,1] \times M,\Omega)$ where
$\Omega$ has the form $\omega + d(t\lambda_0)$ near $\{0\}\times M$,
$d(e^t\lambda)$ near $\{1\} \times M$ and $d(\varphi(t)\lambda_0)$ in
a neighborhood of $[0,1] \times B$ for some positive increasing
function~$\varphi$.  Here $\hH_\epsilon = (\lambda_\epsilon,
\omega)$ is a family of stable Hamiltonian structures adapted to the
open book, so $\xi_\epsilon = \ker\lambda_\epsilon$ for some small $\epsilon > 0$
is a supported contact structure and $\lambda$ is a contact form for 
$\xi_\epsilon$.

Arguing by contradiction, assume $(M,\xi_\epsilon)$ contains a planar torsion
domain $M_0$ that is disjoint from~$B$.  We can then find a neighborhood
$\uU \subset M$ of~$B$ such that $M_0 \subset M \setminus \uU$ and
$\Omega = d(\varphi(t)\lambda_0)$ on $[0,1] \times \uU$.  Extend $W$ to a
noncompact symplectic manifold as follows: first attach to $\{1\} \times M$
a positive cylindrical end that contains a half-symplectization of the form
$$
([T \times \infty) \times M, d(e^t\alpha)).
$$
Note that since $\{1\} \times M$ is a convex boundary component of
$(W,\Omega)$, we are free here to choose $\alpha$ as any contact 
form with $\ker\alpha = \xi_\epsilon$: in particular on~$M_0$ we can assume
it is the special Morse-Bott
contact form provided by Theorem~\ref{thm:openbook}, and since
$M_0 \cap \uU = \emptyset$, we can also assume $\alpha = \lambda_0$
in $\uU$ and
$\Omega = d(e^t\lambda_0)$ on $[1,\infty) \times \uU$.  Secondly,
attach to $\{0\} \times M$ a negative cylindrical end of the form
$$
((-\infty,0] \times M,\omega + d(\psi(t)\lambda_0)),
$$
where $\psi : (-\infty,0] \to \RR$ is an increasing function with
sufficiently small magnitude to make the form symplectic.  Denote the
resulting noncompact symplectic manifold by $(W^\infty,\omega)$.

Recall the special almost complex structure $J_0 \in \jJ(\hH_0)$ constructed
in \S\ref{subsec:bigTheorem}, for which all the pages of~$\pi$ admit
$J_0$-holomorphic lifts in $\RR\times M$.
We now can choose an almost complex structure~$J$ on $(W^\infty,\omega)$
that has the following properties:
\begin{enumerate}
\item $J$ is everywhere compatible with $\omega$
\item $J = J_0$ on both $\RR \times \uU$ and $(-\infty,0] \times M$
\item On $[T,\infty) \times M$, $J$ is the special almost complex 
structure compatible with~$\alpha$ provided by Theorem~\ref{thm:openbook}.
\end{enumerate}
Now the moduli space $\mM_0(J)$ of $J$-holomorphic curves emerging from
$M_0$ in the positive end can be defined as in the previous proof.
The important new feature is that we also have $J$-holomorphic curves
in $W^\infty$ coming from the $J_0$-holomorphic lifts of pages of the
open book: in fact for some $T_0 \in \RR$ sufficiently close to~$-\infty$,
every point in $(-\infty,T_0] \times M$ is contained in such a curve
(see Figure~\ref{fig:torsionCobordism}).
The leaves of the foliation in $[T,\infty) \times M_0$ obviously do not
intersect these curves, so positivity of intersections implies that
no curve in $\mM_0(J)$ may intersect them.  It follows that the curves
in $\mM_0(J)$ can never enter $(-\infty,T_0] \times M$,
so the compactness result Theorem~\ref{thm:compactness}
applies, and we conclude as before that $\mM_0(J)$ fills an open and closed
subset of~$W^\infty$ outside a subset of comdimension~2.  
But this forces some curve in $\mM_0(J)$ to enter
the negative end eventually, and we have a contradiction.
\end{proof}

\begin{remark}
For an arguably easier proof of Theorem~\ref{thm:complement}, one can
present it as a corollary of Theorem~\ref{thm:non-fillable} by showing that
whenever $(M,\xi)$ is supported by an open book $\pi : M \setminus B \to S^1$
and $\uU \subset M$ is a neighborhood of the binding, $(M\setminus \uU,\xi)$
can be embedded into a strongly fillable contact manifold.  This can be
constructed by a doubling trick using the binding sum: if $(M',\xi')$ is
supported by an open book that has the same page~$P$ as $\pi$ but inverse
monodromy, then one can construct a larger contact manifold by 
summing every binding component in~$M$ to a binding component in~$M'$.
The result is a symmetric summed open book which
has a strong symplectic filling homeomorphic to
$[0,1] \times S^1 \times P$, in which the natural projection to
$[0,1] \times S^1$ forms a symplectic fibration.  The details of this
construction are carried out in \cite{LisiVanhornWendl}; see also
the appendix of \cite{BaykurVHM}.
\end{remark}

\subsection{Embedded Contact Homology}
\label{sec:ECH}

Our goal in this section is to prove Theorems~\ref{thm:ECH},
\ref{thm:twisted}, \ref{thm:Umap} and~\ref{thm:UmapTwisted}.  We begin with
a quick review of the essential definitions of Embedded Contact Homology,
mainly following the discussions
in \cite{HutchingsSullivan:T3}*{\S 11} and \cite{Taubes:ECH=SWF5}.

\subsubsection{Review of twisted and untwisted ECH}
\label{subsec:twisted}

Assume $(M,\xi)$ is a closed contact $3$-manifold with nondegenerate contact 
form $\lambda$, and $J$ is a generic almost complex structure on $\RR\times M$
compatible with~$\lambda$.  We will refer to Reeb orbits as \defin{even} or
\defin{odd} depending on the parity of their Conley-Zehnder indices:
in dynamical terms, an even orbit is always hyperbolic, while an odd orbit
can be either elliptic or hyperbolic, the latter if and only if its
double cover is even.  In \S\ref{subsec:definitions} we defined the notion
of an \emph{orbit set} $\boldsymbol{\gamma} = \{ (\gamma_1,m_1),\ldots,
(\gamma_N,m_N) \}$, and we say that $\boldsymbol{\gamma}$ is \defin{admissible}
if $m_i = 1$ whenever $\gamma_i$ is hyperbolic.
Given $h \in H_1(M)$, choose a 
\emph{reference cycle}, i.e.~a $1$-cycle $\boldsymbol{\rho}_h$ in $M$ 
with $[\boldsymbol{\rho}_h] = h$; without loss of generality we can assume
$\boldsymbol{\rho}_h$ is represented by an embedded oriented knot in~$M$
that is not contained in any closed
Reeb orbit.  Then adapting the definition of $H_2(M,\boldsymbol{\gamma}^+
- \boldsymbol{\gamma}^-)$ from \S\ref{subsec:definitions}, it makes sense
to speak of relative homology classes in
$H_2(M , \boldsymbol{\rho}_h - \boldsymbol{\gamma})$ for any orbit
set~$\boldsymbol{\gamma}$ with $[\boldsymbol{\gamma}] = h$.

Given two orbit sets 
$\boldsymbol{\gamma}^\pm = \{ (\gamma_1^\pm,m_1^\pm),\ldots,
(\gamma_{N_\pm}^\pm,m_{N_\pm}^\pm) \}$ and a relative homology class
$A \in H_2(M , \boldsymbol{\gamma}^+ - \boldsymbol{\gamma}^-)$
one defines the \defin{ECH index} $I(A) \in \ZZ$ by choosing any
trivialization~$\Phi$ of~$\xi$ along the orbits in $\boldsymbol{\gamma}^\pm$
and setting
\begin{equation}
\label{eqn:ECHindex}
I(A) = c_1^\Phi(\xi|_A) + A \inter_\Phi A + 
\sum_{i=1}^{N_+} \sum_{k=1}^{m_i^+} \muCZ^\Phi(k \gamma_i^+) -
\sum_{i=1}^{N_-} \sum_{k=1}^{m_i^-} \muCZ^\Phi(k \gamma_i^-),
\end{equation}
where the various symbols are to be interpreted as follows:
\begin{itemize}
\item $c_1^\Phi(\xi|_A)$ is the relative first Chern number
$c_1^\Phi(u^*\xi)$ for any asymptotically cylindrical map~$u$ 
representing~$A$,
\item $A \inter_\Phi A$ is the \emph{relative self-intersection number},
computed as an algebraic count of intersections of some asymptotically
cylindrical representative~$u$ with a generic push-off of~$u$ that is
pushed in the direction of~$\Phi$ at the cylindrical ends,
\item $k\gamma$ denotes the $k$-fold cover of a Reeb orbit~$\gamma$.
\end{itemize}
One can check that this expression does not depend on the choice of
trivializations~$\Phi$.
Since every finite energy $J$-holomorphic curve~$u$ in $\RR\times M$
represents a relative homology class, we can define the ECH index of~$u$
as $I(u) := I([u])$.  

\begin{defn}
\label{defn:flowLine}
A (possibly disconnected) finite energy $J$-holomorphic curve 
$u : \dot{\Sigma} \to \RR\times M$ is called a \defin{flow line} if it
is a disjoint union of two curves $u_0$ and~$C$, where $u_0$ is embedded, 
and~$C$ is any collection of trivial cylinders that do not
intersect~$u_0$.
\end{defn}

Hutchings \cite{Hutchings:index} has shown that for generic~$J$,
a flow line~$u$ always satisfies $1 \le \ind(u) \le I(u)$.
Embedded Contact Homology is defined by counting specifically the flow 
lines for which this inequality is an equality.
For any subgroup $G \subset H_2(M)$, define
$$
\widetilde{C}_*(M , \lambda ; h, G)
$$
to be the free $\ZZ$-module generated by symbols of the
form $e^A \boldsymbol{\gamma}$, where $\boldsymbol{\gamma}$ is an
admissible orbit set with $[\boldsymbol{\gamma}] = h$ and
$A \in H_2(M,\boldsymbol{\rho}_h - \boldsymbol{\gamma}) / G$, meaning
$A \sim A'$ whenever $A - A' \in G$.  A differential
$\p : \widetilde{C}_*(M, \lambda ; h ,G) \to 
\widetilde{C}_{*-1}(M,\lambda ; h,G)$ is defined by
$$
\p \left( e^A \boldsymbol{\gamma}\right) = \sum_{\boldsymbol{\gamma}',A'} 
\# \left(\frac{\mM_{\text{emb}}^1(\boldsymbol{\gamma},\boldsymbol{\gamma}',A')}{\RR} \right) 
e^{A + A'} \boldsymbol{\gamma}',
$$
where the sum ranges over all admissible orbit sets $\boldsymbol{\gamma'}$
and $A' \in H_2(M , \boldsymbol{\gamma} - \boldsymbol{\gamma}') / G$,
and $\mM_\text{emb}^1(\boldsymbol{\gamma},\boldsymbol{\gamma}',A') \subset
\mM(J)$ is the
oriented $1$-manifold of (possibly disconnected) finite energy $J$-holomorphic
curves $u : \dot{\Sigma} \to \RR \times M$ satisfying the
following conditions:
\begin{enumerate}
\renewcommand{\labelenumi}{(\roman{enumi})}
\item $I(u) = 1$,
\item $[u] \sim A'$ in $H_2(M,\boldsymbol{\gamma} - \boldsymbol{\gamma}') / G$,
\item $u$ is a flow line in the sense of Definition~\ref{defn:flowLine}.
\end{enumerate}

The orientation of 
$\mM_\text{emb}^1(\boldsymbol{\gamma},\boldsymbol{\gamma}',A')$
is chosen in accordance with \cite{BourgeoisMohnke}, which requires first
choosing an ordering for all the even orbits in~$M$, then ordering the 
punctures of any 
$u \in \mM_\text{emb}^1(\boldsymbol{\gamma},\boldsymbol{\gamma}',A')$
accordingly.  The signed count above is then
finite due to the index inequality and compactness theorem in
\cite{Hutchings:index}.\footnote{The results in \cite{Hutchings:index} are
stated only for a very special class of stable Hamiltonian structures
arising from mapping tori, but they extend to the contact case due to 
the relative asymptotic formulas of Siefring \cite{Siefring:asymptotics}.}  
These same results together with the gluing construction of
\cites{HutchingsTaubes:gluing1,HutchingsTaubes:gluing2} imply that
$\p^2 = 0$, and the resulting homology is denoted by
$\widetilde{\ECH}_*(M,\lambda,J ; h,G)$.  We have two natural choices for
the subgroup~$G$: if $G = H_2(M)$, then the terms $e^A$ are all trivial and
we obtain the usual \emph{untwisted} Embedded Contact Homology,
$$
\ECH_*(M,\lambda,J ; h) := \widetilde{\ECH}_*(M,\lambda,J ; h,H_2(M)).
$$
At the other end of the spectrum, taking $G$ to be the trivial subgroup
leads to the \emph{fully twisted} variant of ECH,
$$
\widetilde{\ECH}_*(M,\lambda,J ; h) :=
\widetilde{\ECH}_*(M,\lambda,J ; h,\{0\}).
$$
Since every nontrivial finite energy $J$-holomorphic curve in $\RR\times M$ 
has at least one positive puncture, the empty orbit set 
$\boldsymbol{\emptyset}$  always satisfies $\p\boldsymbol{\emptyset} = 0$,
and thus represents a homology class which we call the (untwisted)
\defin{contact class},
$$
c(\lambda,J) = [\boldsymbol{\emptyset}] \in \ECH_*(M,\lambda,J ; 0).
$$
To define the twisted contact class, we note that for $h=0$ there is a
canonical choice of reference cycle $\boldsymbol{\rho}_0$, namely the
empty set, so 
$H_2(M, \boldsymbol{\rho}_0 - \boldsymbol{\emptyset}) = H_2(M)$ and it is
natural to define
$$
\tilde{c}(\lambda,J) = [e^0 \boldsymbol{\emptyset}] \in
\widetilde{\ECH}_*(M,\lambda,J ; 0).
$$
A chain map $U : \widetilde{C}_*(M, \lambda ; h ,G) \to 
\widetilde{C}_{*-2}(M,\lambda ; h,G)$ can be defined by choosing a
generic point $p \in M$ and counting index~$2$ holomorphic curves that
pass through the point $(0,p)$, that is
$$
U \left( e^A \boldsymbol{\gamma}\right) = \sum_{\boldsymbol{\gamma}',A'} 
\# \left(\mM_{\text{emb}}^2(\boldsymbol{\gamma},\boldsymbol{\gamma}',A' ; p) \right) 
e^{A + A'} \boldsymbol{\gamma}',
$$
where $\mM_{\text{emb}}^2(\boldsymbol{\gamma},\boldsymbol{\gamma}',A' ; p)$
consists of $J$-holomorphic flow lines~$u$ with $I(u)=2$ and
one marked point which is mapped to the point $(0,p)$.
We denote by
$$
U : \ECH_*(M,\lambda,J ; h) \to \ECH_{*-2}(M,\lambda,J ; h)
$$
and
$$
\widetilde{U} : \widetilde{\ECH}_*(M,\lambda,J ; h) \to 
\widetilde{\ECH}_{*-2}(M,\lambda,J ; h)
$$
respectively the untwisted and fully twisted variants of the resulting
map on homology.

It follows from Taubes's isomorphism \cites{Taubes:ECH=SWF1,Taubes:ECH=SWF5}
that none of the above depends on the choice of~$\lambda$ and~$J$, and
the $U$-map also does not depend on the choice of generic point $p \in M$.

\subsubsection{Proof of the vanishing theorems}
\label{subsubsec:vanishing}

We now prove Theorems~\ref{thm:ECH} and~\ref{thm:twisted}.
Assume $(M,\xi)$ contains a planar $k$-torsion domain $M_0$ with planar piece
$M_0^P \subset M_0$.  Note that for some planar torsion domains, there may
be multiple subsets of $M_0$ that could sensibly be called the planar piece
(e.g.~$M_0$ could contain multiple planar open books summed together as
in Figure~\ref{fig:torsionAmbiguity}), 
so whenever such
an ambiguity exists, we choose $M_0^P$ to make~$k$ as small as possible.
Let $\lambda$ and~$J$ denote the special Morse-Bott
contact form and compatible Fredholm regular almost complex structure
provided by Theorem~\ref{thm:openbook}.  Then $\p M_0^P$ is a non-empty 
union of tori
$$
\p M_0^P = T_1 \cup \ldots \cup T_n
$$
which are Morse-Bott families of Reeb orbits, and the interior
of~$M_0^P$ may also contain interface tori, which we denote by
$$
\iI_0 = T_{n+1} \cup \ldots \cup T_{n+r},
$$
and binding circles
$$
B_0 = \beta_1 \cup\ldots \cup \beta_m.
$$
The planar pages in $M_0^P$
have embedded $J$-holomorphic lifts to $\RR \times M$, forming a
family of curves,
$$
u_{\sigma,\tau} \in \mM(J),
\qquad (\sigma,\tau) \in \RR \times S^1,
$$
which have no negative punctures and $m+n+2r$ positive punctures, each
asymptotic to simply covered orbits in $B_0 \cup \iI_0 \cup \p M_0^P$,
exactly one in each connected component of $B_0 \cup \p M_0^P$ and two
in each component of~$\iI_0$.
Moreover, other than these curves and the obvious trivial cylinders, 
there is no other 
connected finite energy $J$-holomorphic curve in $\RR\times M$ with
its positive ends approaching any subcollection of the asymptotic orbits 
of~$u_{\sigma,\tau}$.

\begin{figure}
\begin{center}
\includegraphics{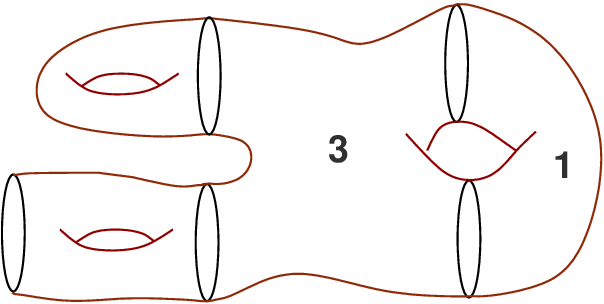}
\caption{\label{fig:torsionAmbiguity}
A planar torsion domain for which the order is not uniquely defined:
depending on the choice of planar piece, the order could be either~$1$
or~$3$.}
\end{center}
\end{figure}

We now perturb $\lambda$ to a nondegenerate contact form $\lambda'$ by the
scheme described in \cite{Bourgeois:thesis}, so that each of the original
Morse-Bott tori $T_j \subset \iI_0 \cup \p M_0^P$ 
contains exactly two nondegenerate
Reeb orbits, one elliptic and one hyperbolic,
$$
\gamma_j^e \cup \gamma_j^h \subset T_j.
$$
Denoting by~$\Phi_0$ the natural trivialization along these orbits 
determined by the $(\theta,\rho,\phi)$-coordinates, 
they satisfy $\muCZ^{\Phi_0}(\gamma_j^e) = 1$
and $\muCZ^{\Phi_0}(\gamma_j^h) = 0$, and for any number $k_0 \in \NN$ we
can also arrange that $\muCZ^{\Phi_0}(k\gamma_j^e) = 1$ for all $k \le k_0$.
Perturbing $J$ to a generic $J'$ compatible with $\lambda'$, the family of 
curves $u_{\sigma,\tau}$ gives rise to embedded
$J'$-holomorphic curves (Figure~\ref{fig:MorseBott})
asymptotic to various combinations of these orbits
and the components of~$B_0$.  If $u : \dot{\Sigma} \to \RR \times M$
is such a curve, then genericity implies $\ind(u) \ge 1$, so we deduce from
the index formula that such curves come in two types:
\begin{itemize}
\item $\ind(u) = 2$ if all ends approaching $\iI_0 \cup \p M_0^P$ 
approach elliptic orbits,
\item $\ind(u) = 1$ if $u$ has exactly one end approaching a hyperbolic
orbit in~$\iI_0 \cup \p M_0^P$.
\end{itemize}
All of these curves also have genus zero and satisfy 
$u \bullet_{\Phi_0} u = 0$ and $c_N(u) = 
c_1^{\Phi_0}(u) - \chi(\dot{\Sigma}) = 0$, so one can then deduce from
\eqref{eqn:ECHindex} and the index formula \eqref{eqn:index} that
$I(u) = \ind(u)$.

Up to $\RR$-translation there is now exactly one $J'$-holomorphic flow line
$u_0 : \dot{\Sigma} \to \RR\times M$ with all punctures positive and asymptotic
to the orbits
$$
\gamma_1^h,\gamma_2^e,\ldots,\gamma_n^e,\gamma_{n+1}^e,\gamma_{n+1}^e,\ldots,
\gamma_{n+r}^e,\gamma_{n+r}^e,\beta_1,\ldots,\beta_m.
$$
Let us therefore define the orbit set
$$
\boldsymbol{\gamma}_0 = \{ (\gamma_1^h,1),(\gamma_2^e,1),\ldots,
(\gamma_{n}^e,1),(\gamma_{n+1}^e,2),\ldots,(\gamma_{n+r}^e,2),
(\beta_1,1),\ldots,(\beta_m,1) \},
$$
for which $[\boldsymbol{\gamma}_0] = 0$, and define also
the relative homology class
$$
A_0 = - [u_0] \in H_2(M , \boldsymbol{\rho}_0 - \boldsymbol{\gamma}_0).
$$
The perturbation from $J$ to $J'$ creates some additional $J'$-holomorphic 
cylinders which arise from gradient flow lines along the Morse-Bott
families of orbits, as described in \cite{Bourgeois:thesis}.
Namely for each $j=1,\ldots,n+r$, there are two
embedded cylinders
$$
v_j^+ , v_j^- : \RR \times S^1 \to \RR \times M,
$$
each with positive end at $\gamma_j^e$ and negative end at $\gamma_j^h$;
the images of these cylinders in~$M$ are the two connected components of
$T_j \setminus (\gamma_j^e \cup \gamma_j^h)$,
thus after choosing the labels appropriately, we can assume their
relative homology classes are related by
$$
[v_j^+] - [v_j^-] = [T_j] \in H_2(M).
$$
These cylinders satisfy $\ind(v_j^\pm) = I(v_j^\pm) = 1$.
\begin{figure}
\begin{center}
\includegraphics{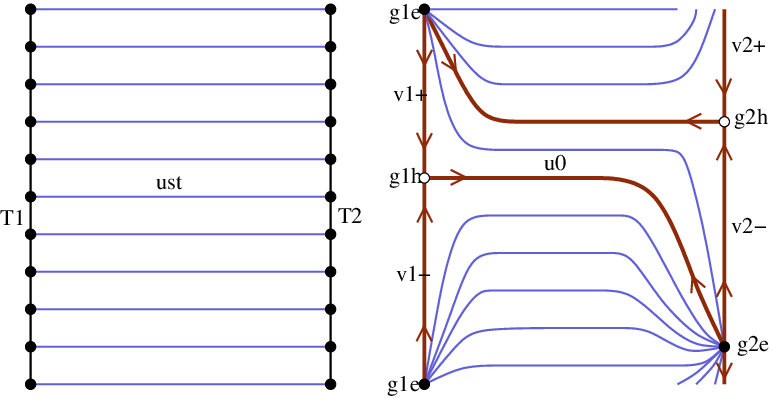}
\caption{\label{fig:MorseBott}
The perturbation from Morse-Bott (left) to nondegenerate (right), shown
here in the simple case where $u_{\sigma,\tau}$ is a family of cylinders
asymptotic to two Morse-Bott tori.  All the orbits in the picture point
along an $S^1$-factor through the page, and the top and bottom are identified.
Arrows indicate the signs of the ends of the rigid curves in the nondegenerate
picture: an end is positive if and only if the arrow points \emph{away
from} the orbit.}
\end{center}
\end{figure}

Now in the twisted ECH complex, the only curves other than $u_0$
counted by $\p\left( e^{A_0} \boldsymbol{\gamma}\right)$ 
are the disjoint unions of
$v_j^\pm$ with collections of trivial cylinders for $j=2,\ldots,n+r$.
The negative ends of such a disjoint union give rise to the orbit set
\begin{equation*}
\begin{split}
\boldsymbol{\gamma}_j := \{ & (\gamma_1^h,1),(\gamma_2^e,1),\ldots,
(\gamma_{j-1}^e,1),(\gamma_j^h,1),(\gamma_{j+1}^e,1),\ldots,(\gamma_{n}^e,1),\\
& (\gamma_{n+1}^e,2),\ldots,(\gamma_{n+r}^e,2) , (\beta_1,1),\ldots,(\beta_m,1) \}
\end{split}
\end{equation*}
for $j=1,\ldots,n$, and a similar expression for $j=n+1,\ldots,n+r$
which will appear twice due to the multiplicity attached to~$\gamma_j^e$.
Choosing appropriate coherent orientations and adding all this together,
we find
\begin{equation*}
\begin{split}
\p\left( e^{A_0} \boldsymbol{\gamma}_0\right) &= 
e^0 \boldsymbol{\emptyset} + \sum_{j=2}^n e^{A_0 + [v_j^-]}
\left( e^{[T_j]} - 1 \right) \boldsymbol{\gamma}_j \\
&\qquad + \sum_{j=n+1}^{n+r} 2 e^{A_0 + [v_j^-]}
\left( e^{[T_j]} - 1 \right) \boldsymbol{\gamma}_j.
\end{split}
\end{equation*}
We thus have $\p\left( e^{A_0} \boldsymbol{\gamma}_0\right) =
e^0\boldsymbol{\emptyset}$ whenever $[T_j] = 0 \in H_2(M)$ for all
$j=2,\ldots,n+r$, which proves Theorem~\ref{thm:twisted}.  For untwisted
coefficients, we divide the entire calculation by $H_2(M)$ so that
$e^{[T_j]} - 1 = 0$ always, thus $\p\boldsymbol{\gamma}_0 = 
\boldsymbol{\emptyset}$ holds with no need for any topological condition.
With that, the proof of Theorem~\ref{thm:ECH} is complete.

\subsubsection{The $U$-map}
\label{subsubsec:Umap}

The proof of Theorems~\ref{thm:Umap} and~\ref{thm:UmapTwisted} is a minor
variation on the argument given above.  Assume 
$(M,\xi)$ contains a partially planar domain $M_0$ with planar piece
$M_0^P \subset M_0$, and choose the Morse-Bott data $\lambda,J$ and
nondegenerate perturbation $\lambda',J'$ exactly as described in the
previous section, but adding the following condition: for any
given $d \in \NN$, Theorem~\ref{thm:openbook} allows us to choose
$\lambda$ so that the uniqueness statement for holomorphic curves
subordinate to the planar piece up to multiplicity~$k$ holds for any
$k \le d$.

Now consider the $J'$-holomorphic curves of index~$2$ with positive ends
asymptotic to the elliptic orbits,
$$
\gamma_1^e,\ldots,\gamma_n^e,\gamma_{n+1}^e,\gamma_{n+1}^e,\ldots,
\gamma_{n+r}^e,\gamma_{n+r}^e,\beta_1,\ldots,\beta_m.
$$
These curves have embedded projections to~$M$ which foliate an open 
subset of $M_0^P$,
thus if we choose~$p$ in this open subset, there is exactly one curve
with the given asymptotics that passes through $(0,p)$.  Denote this
curve by~$u_p$, and for any $k \in \{1,\ldots,d\}$, 
define the orbit set
$$
\boldsymbol{\gamma}^{(k)} = \{ (\gamma_1^e,k),\ldots,
(\gamma_{n}^e,k),(\gamma_{n+1}^e,2k),\ldots,(\gamma_{n+r}^e,2k),
(\beta_1,k),\ldots,(\beta_m,k) \}
$$
with $[\boldsymbol{\gamma}^{(k)}] = 0$, and the relative homology class
$$
k A_p = - k [u_p] \in H_2(M, \boldsymbol{\rho}_0 - \boldsymbol{\gamma}^{(k)}).
$$
The uniqueness statement in
Theorem~\ref{thm:openbook} for curves subordinate to the planar piece up
to multiplicity~$d$ now implies that 
$\p\left( e^{k A_p} \boldsymbol{\gamma}^{(k)} \right)$ counts only the
disjoint unions of the embedded index~$1$ cylinders $v_j^\pm$ with trivial
cylinders.  As in the previous section, the contributions from
$v_j^+$ and $v_j^-$ cancel each other out in the untwisted theory, and also
in the twisted theory if $[T_j] = 0 \in H_2(M)$, so we conclude in either
case that $e^{k A_p} \boldsymbol{\gamma}^{(k)}$ is a cycle in the chain
complex.  The uniqueness result also implies that there is exactly one
curve counted by $U\left( e^{k A_p} \boldsymbol{\gamma}^{(k)} \right)$,
namely the disjoint union of $u_p$ with a collection of trivial cylinders.
We thus find,
$$
U\left( e^{k A_p} \boldsymbol{\gamma}^{(k)} \right) = e^{(k-1) A_p}
\boldsymbol{\gamma}^{(k-1)}
$$
for each $k \in \{2,\ldots,d\}$, and for $k=1$,
$$
U\left( e^{A_p} \boldsymbol{\gamma}^{(1)} \right) = e^{0}
\boldsymbol{\emptyset}.
$$
Since the ECH does not depend on the choice of contact form, this shows
that for all $d \in \NN$ the homology contains an element whose image
under~$d$ iterations of the $U$-map is the contact class.
The proof of Theorems~\ref{thm:Umap} and~\ref{thm:UmapTwisted} is
thus complete.

\subsection*{Acknowledgments}
I am grateful to Dietmar Salamon, Peter Albers, Barney Bramham,
Klaus Niederkr\"uger
and Paolo Ghiggini for helpful conversations, and especially 
to Michael Hutchings
for explaining to me the compactness theorem for holomorphic currents,
and Janko Latschev and Patrick Massot for several useful comments on 
earlier versions of the paper.
The idea of blowing up binding orbits of open books was inspired by a
talk I heard Hutchings give at Stanford's Holomorphic Curves FRG Workshop
in August 2008; I'd like to thank the organizers of that
workshop for the invitation.

\end{document}